\documentclass[a4paper,reqno,10pt]{amsart}

\usepackage{amsmath}
\usepackage{amsthm}
\usepackage{amssymb}
\usepackage{amscd} %diagrams
\usepackage{bbm} %bold numbers
\usepackage{mathrsfs} %nice calligraphy
\usepackage{lmodern} 
\usepackage[T1]{fontenc} %allows for hyphenation of words with accents
\usepackage{newtxtext,newtxmath}
\usepackage{verbatim} %DO NOT comment this line (comment environment)
\usepackage{comment}
\usepackage[svgnames]{xcolor}
\usepackage{pdfcolmk}
\usepackage[hidelinks]{hyperref}
\usepackage{tikz-cd} \usepackage{pifont}

\theoremstyle{plain}
\newtheorem*{thm*}{Theorem} %no numbering for Theorem*
\newtheorem{thm}{Theorem}
\newtheorem{lemma}[thm]{Lemma}
\newtheorem{coro}[thm]{Corollary}
\newtheorem{prop}[thm]{Proposition}

\newtheorem{defi}[thm]{Definition}

\theoremstyle{remark}

\newtheorem*{rmk*}{Remark}

\newtheorem*{Claim*}{Claim}

\numberwithin{thm}{section}

\newcommand{\R}{\mathbb{R}} %real numbers

\newcommand{\Z}{\mathbb{Z}} %integers
\newcommand{\ind}{\mathbbm{1}} %indicatrix function

\newcommand{\diff}{\mathop{}\mathopen{}\mathrm{d}} %différentielle

\newcommand{\calD}{\mathscr{D}}

\newcommand{\calH}{\mathscr{H}}

\newcommand{\calL}{\mathscr{L}}

\newcommand{\calP}{\mathscr{P}}
\newcommand{\calQ}{\mathscr{Q}}

\newcommand{\bCH}{\mathbf{CH}} %space of charges
\newcommand{\bch}{\mathbf{ch}}

\newcommand{\bF}{\mathbf{F}}

\newcommand{\bM}{\mathbf{M}}
\newcommand{\bN}{\mathbf{N}}

\newcommand{\bP}{\mathbf{P}}

\DeclareMathOperator{\rmLip}{\operatorname{Lip}} %Lipschitz constant

\DeclareMathOperator{\rmloc}{\mathrm{loc}}

\DeclareMathOperator{\spt}{\operatorname{spt}} %support
\newcommand{\hel} {
\hskip2.5pt{\vrule height7pt width.5pt depth0pt}
\hskip-.2pt\vbox{\hrule height.5pt width7pt depth0pt}
\, }

\makeatletter
\newcommand*{\cone}{\mathrel{\mathpalette\@foo\relax}} % or \mathbin? or nothing?
\newcommand*{\@foo}[2]{%
   \setbox\z@=\hbox{\m@th$#1\times$}%
   \ooalign{\raise.25\ht\z@\copy\z@\cr\box\z@}% the .25 might depend on the used font
}
\makeatother

\def\XXint#1#2#3{{%
\setbox0=\hbox{$#1{#2#3}{\int}$}
\vcenter{\hbox{$#2#3$}}\kern-.5\wd0}}

\renewcommand{\leq}{\leqslant}
\renewcommand{\geq}{\geqslant}
\renewcommand{\subset}{\subseteq}

\usepackage{newtxtext}
\usepackage{newtxmath}
\DeclareMathAlphabet\euscr{U}{eus}{m}{n}

\newlength{\drop}

\usepackage{trajan}

\newcommand{\defeq}{\mathrel{\mathop:}=}

\begin{document}

%=================
% TITLE AND AUTHOR
%=================

%\titleAT %see above / page de garde d'un livre

\title{Fractional currents and Young geometric integration}

\author[Ph. Bouafia]{Philippe Bouafia}

\address{Fédération de Mathématiques FR3487 \\
  CentraleSupélec \\
  3 rue Joliot Curie \\
  91190 Gif-sur-Yvette
}

\email{philippe.bouafia@centralesupelec.fr}


\begin{abstract}
  We introduce a class of flat currents with fractal properties,
  called fractional currents, which satisfy a compactness theorem and
  remain stable under pushforwards by Hölder continuous maps. In top
  dimension, fractional currents are the currents represented by
  functions belonging to a fractional Sobolev space.

  The space of $\alpha$-fractional currents is in duality with a class
  of cochains, $\alpha$-fractional charges, that extend both Whitney's
  flat cochains and $\alpha$-Hölder continuous forms. We construct a
  partially defined wedge product between fractional charges,
  enabling a generalization of the Young integral to arbitrary
  dimensions and codimensions. This helps us identify
  $\alpha$-fractional $m$-currents as metric currents of the
  snowflaked metric space $(\mathbb{R}^d,
  \mathrm{d}_{\mathrm{Eucl}}^{(m+\alpha)/(m+1)})$.
\end{abstract}

\maketitle

\setcounter{tocdepth}{1}
\tableofcontents

\section{Introduction}
In this article, we define a class of currents in Euclidean space,
referred to as fractional currents, on which a form of Hölder
non-smooth calculus can be performed. By this, we mean, in particular,
the ability to push forward fractional currents by Hölder maps
(subject to certain restrictions on the Hölder exponent) or to
integrate ``rough'' differential forms over fractional currents. We
summarize our main results.

\subsection{Fractional currents}
To motivate the definition of fractional currents, we begin by
highlighting the importance of fractional Sobolev spaces in the Hölder
analysis of $\mathbb{R}^d$. Notably, a recent result of De Lellis and
Inauen~\cite{DeLeInau} demonstrates the fractional Sobolev regularity
of the Brouwer degree $\operatorname{deg}(f,U,\cdot)$ for a Hölder
continuous map $f \colon \operatorname{cl} U \to \mathbb{R}^d$ and a
bounded open set $U$ with a fractal boundary (under some restrictions
on the Hölder exponent of $f$ and the regularity of $\partial U$).
This result suggests that, in order to define an appropriate class of
currents for which pushforwards by Hölder maps can be meaningfully
constructed, such currents should be representable in codimension $0$
by fractional Sobolev functions.

It is in fact possible to characterize the fractional Sobolev space
$W^{1-\alpha, 1}_c(\R^d)$, with $0 < \alpha < 1$ (the subscript $c$
stands for compact support), in the language of currents. Recall that
a function $u \in L^1_c(\R^d)$ belongs to $W^{1-\alpha,1}_c(\R^d)$ if
its Gagliardo norm
\[
\|u\|_{W^{1-\alpha,1}(\R^d)} = \int_{\R^d} \int_{\R^d} \frac{|u(x) -
  u(y)|}{|x - y|^{d+1-\alpha}} \diff x \diff y
\]
is finite. It was noted in~\cite{BrasLindPari, Ponc} that
$\|u\|_{W^{1-\alpha, 1}(\R^d)}$ represents an intermediate quantity
between the total variation $\| Du\|(\R^d)$ and the $L^1$ norm of
$u$. Specifically, it satisfies the interpolation inequality
\[
\|u\|_{W^{1-\alpha,1}(\R^d)} \leq C \| Du \|(\R^d)^{1 - \alpha}
\|u\|_1^\alpha
\]
which, in fact, almost characterizes the Gagliardo norm $\| \cdot
\|_{W^{1 -\alpha, 1}(\R^d)}$ in the sense that we now describe.

Let $N \colon L^1_c(\R^d) \to [0, \infty]$ denote the largest
seminorm that satisfies
\begin{itemize}
\item[(A)] $N(u) \leq \| Du \|(\R^d)^{1 - \alpha} \|u\|_1^\alpha$ for all $u \in L^1_c(\R^d)$;
\item[(B)] Lower semicontinuity: $N(u) \leq \liminf_{k \to \infty}
  N(u_k)$ for all sequences $(u_k)$ in $L^1_c(\R^d)$ that
  $L^1$-converge to $u$, with the additional condition that there is a
  compact set $K \subset \R^d$ such that $\spt u_k \subset K$ for all
  $k$.
\end{itemize}
It is easy to prove that $N$ is given by the formula
\[
N(u) = \inf \sum_{k=0}^\infty \| Du_k \|(\R^d)^{1 - \alpha}
\|u_k\|_1^\alpha
\]
where the infimum is taken over all decompositions of $u$ as an
$L^1$-convergent series $u = \sum_{k=0}^\infty u_k$ into components
$u_k$ of bounded variation and supported in a common compact subset of
$\R^d$. We will establish in Section~\ref{sec:codim0} that $N$ is
Lipschitz equivalent to the norm $\| \cdot
\|_{W^{1-\alpha,1}(\R^d)}$. Note that this result is closely linked to
the idea presented in \cite{PoncSpec}, which considers the fractional
Sobolev space $W^{1-\alpha,1}$ ($0 < \alpha < 1$) as an interpolation
space between $L^1$ and $BV$.

In positive codimension, we obtain a suitable definition of a
fractional current by replacing, in the discussion above, the total
variation and the $L^1$-norm by higher-codimensional analogs, namely
the normal mass $\bN$ and Whitney's flat norm $\bF$. More precisely,
an $\alpha$-fractional $m$-current is a flat current $T \in
\bF_m(\R^d)$ that admits a decomposition into an $\bF$-convergent
series $T = \sum_{k=0}^\infty T_k$, where the components $T_k$ are
normal $m$-currents supported in a common compact subset of $\R^d$
such that
\[
\sum_{k=0}^\infty \bN(T_k)^{1 - \alpha} \bF(T_k)^\alpha < \infty.
\]
The space of $\alpha$-fractional currents $\bF^\alpha_m(\R^d)$,
equipped with an adequate norm, lies between the space of normal
currents and the space of flat currents. This space extends the already
quite general class of fractal currents introduced by Züst
in~\cite{Zust}, see Subsection~\ref{subsec:zust}.

\subsection{Related spaces of cochains}
We establish a compactness theorem for fractional currents, extending
the classical Federer-Fleming compactness theorem for normal currents
\cite[4.2.17]{Fede}. This result will play a key role in the
construction of the pushforward operator by Hölder maps.

The proof of the compactness theorem partly relies on
functional-analytical arguments. We show indeed that
$\bF_m^\alpha(K)$, the space of $\alpha$-fractional $m$-currents
supported in a compact subset $K \subset \R^d$, is a dual space
(hence, a form of weak compactness is provided by the
Banach-Alao\u{g}lu theorem). As any predual of a Banach space $X$ is
necessarily embedded in $X^*$, this calls for a description of the
dual space $\bF_m^\alpha(K)^*$. Elements of $\bF_m^\alpha(K)^*$ are
cochains, that are of independent interest. Loosely speaking, they can
be viewed as Hölder versions of Whitney's flat cochains. The next
paragraph provides an alternative description of $\bF_m^\alpha(K)^*$.

An $\alpha$-fractional charge $\omega$ over $K$ is a linear functional
on $\bN_m(K)$, the space of normal $m$-currents supported in $K$, that
satisfies the continuity condition
\begin{equation}
  \label{eq:introCF}
  |\omega(T)| \leq C \bN(T)^{1-\alpha} \bF(T)^\alpha
\end{equation}
for all $T \in \bN_m(K)$. When $K$ is convex, $\omega$ extends by
density to a continuous linear functional on
$\bF_m^\alpha(K)$. Conversely, any element of $\bF_m^\alpha(K)^*$,
restricted to $\bN_m(K)$, satisfies~\eqref{eq:introCF}. Thus, the
space $\bCH^{m, \alpha}(K)$ of $\alpha$-fractional $m$-charges is
(isomorphic to) the dual of $\bF_m^\alpha(K)$ in the case where $K$ is
convex.

One important class of $\alpha$-fractional charges consists of
$\alpha$-Hölder continuous differential forms. More precisely, to each
such form $\eta \in \rmLip^\alpha(K; \wedge^m \R^d)$, one associates
the functional $\Lambda(\eta)$ that acts on normal currents $T \in
\bN_m(K)$ by
\[
\Lambda(\eta)(T) = \int_K \langle \eta(x), \vec{T}(x) \rangle
\diff \| T \|(x).
\]
In the simplest case where $m = 0$, the map $\Lambda$ is in fact a
Banach space isomorphism between $\rmLip^\alpha(K)$ and $\bCH^{0,
  \alpha}(K)$. As a result, we can identify $\alpha$-fractional
$0$-charges over $K$ with $\alpha$-Hölder continuous functions on
$K$. In this case, it is known that $\rmLip^\alpha(K)$ is the double
dual of the little Hölder space, described for example in
\cite[Chapter~4]{Weav}. We treat the general case $m \geq 0$ by
analogy. We define the space of little $\alpha$-fractional $m$-charges
$\bch^{m, \alpha}(K)$ as the closure in $\bCH^{m, \alpha}(K)$ of
$\{\Lambda(\omega_{\mid K}) : \omega \text{ smooth} \}$, and we
establish that $\bF_m^\alpha(K) \cong \bch^{m,\alpha}(K)^*$. Thus any
bounded sequence in $\bF_m^\alpha(K)$ possesses a weakly convergent
subsequence (in the sense of currents), by the Banach-Alao\u{g}lu
theorem. Convergence in flat norm is later obtained by adapting the
classical deformation theorem to the setting of fractional currents.

Our initial motivation for introducing fractional charges was to
provide a generalization in positive codimension of the Hölder charges
introduced in~\cite{BouaDePa} and further studied in~\cite{Boua}.
The reader is referred to the survey~\cite{Bouafia2026nonabsolute} for
an overview of these works, as well as for examples of fractional
charges arising from stochastic processes.

\subsection{Generalized Stokes' formula}
The exterior derivative operator for fractional charges is
well-defined as the adjoint of the boundary operator. Specifically,
$\diff \omega$ is the $\alpha$-fractional $(m+1)$-charge over $K$,
given by
\begin{equation}
  \label{eq:stokesIntro}
  \langle \diff \omega, T \rangle = \langle \omega, \partial T\rangle
\end{equation}
for $T \in \bN_m(K)$. If $K$ is convex, this formula extends more
generally to $\alpha$-fractional currents over $K$.

A particular case of this formula deserves attention, as it results in
a generalized Gauss-Green formula for Hölder vector fields over
fractal boundaries. It extends the work of Harrison and
Norton~\cite{HarrNort} (it should be however stressed that
formula~\eqref{eq:stokesIntro} is purely definitional, contrary to the
more constructive approach taken in~\cite{HarrNort, Guse}). Indeed, we
can take for $U$ a bounded open set with finite $(1-\alpha)$-perimeter
(meaning that $\ind_U \in W_c^{1-\alpha,1}(\R^d)$, or equivalently,
$\llbracket U \rrbracket$ is $\alpha$-fractional). We shall prove that
this condition is weaker than the $(d-1+\alpha)$-summability
(see~\cite[\S{3}]{HarrNort} or Subsection~\ref{subsec:firstexample})
of $\partial U$ that is required in the Gauss-Green formula of
Harrison-Norton.  In this case, formula~\eqref{eq:stokesIntro} applies
for $T = \llbracket U \rrbracket$ and $\omega$ that is associated to
an $\alpha$-Hölder continuous $(d-1)$-form.

\subsection{Pushforward by Hölder maps}
The pushforward of currents by Hölder maps $f$ has been explored in
the literature, notably in \cite{FlanGiacGubiTort} and
\cite{Zust3}. The results in~\cite{Zust3} complement ours, as they
address the case where the components $f_i$ of $f$ have different
Hölder exponents.

Our contribution is to provide a functional-analytical setting to this
construction by showing that $f_\#$ is a bounded linear operator between
spaces of fractional currents. Indeed, in Section~\ref{sec:holderpf},
we define the pushforward $f_\# T$ of an $\alpha$-fractional
$m$-current $T$ by a $\gamma$-Hölder continuous map $f$. The
construction makes sense as long as $m+\alpha < \gamma (m+1)$. Under
this assumption, we show that $f_\#T$ is $\beta$-fractional, where
$\beta$ is given by the equation
\[
\frac{m+\alpha}{\gamma} = m + \beta.
\]
The above inequality provides an interpretation of the parameter
$\alpha$ as an upper bound for the fractional part of the fractal
dimension of both $T$ and its boundary $\partial T$.

In top dimension, \textit{i.e.}, whenever $T$ is a $d$-dimensional
current in $\R^d$ and $f$ is $\R^d$-valued, the pushforward is defined
under the less restrictive condition $d-1 + \alpha <
\gamma(d-1)$. When applied to $T = \llbracket U \rrbracket$, this
pushforward construction allows us to recover the result of De Lellis
and Inauen on the fractional Sobolev regularity of the Brouwer degree
$\operatorname{deg}(f, U, \cdot)$, and even strengthens it slightly by
treating the critical case where $U$ has finite fractional perimeter.

\subsection{Young geometric integration}
Züst introduced in~\cite{Zust2} a notable generalization of the
one-dimensional Young integration to Hölder forms, making sense of
geometric integrals like
\begin{equation}
  \label{eq:zustI}
  \int_{[0, 1]^d} g^0 \diff g^1 \wedge \cdots \wedge \diff g^d
\end{equation}
where $g^0, \dots, g^d$ are Hölder-continuous functions over $[0,
  1]^d$, with Hölder exponents $\alpha_0, \dots, \alpha_d \in
\mathopen{(} 0, 1 \mathclose{]}$ such that
  \begin{equation}
    \label{eq:condZ}
  \alpha_0 + \cdots +
  \alpha_d > d.
  \end{equation}
The integral~\eqref{eq:zustI} is there defined via approximation by
Riemann sums, and the wedge products that appear in~\eqref{eq:zustI}
are purely formal. The condition~\eqref{eq:condZ} was shown to be
sharp for the convergence of these sums; see
\cite[Section~3.2]{Zust2}. In the context of (zero-codimensional)
charges, a generalization of the Züst integral was proposed
in~\cite{Boua}. Recently, T. Jaffard described precisely the
regularity of the differential forms in~\ref{eq:zustI} in~\cite{Jaff},
in terms of negative Hölder spaces.

Our goal in Section~\ref{sec:ext} is to set up a rough exterior
calculus apparatus in which differential forms, such as the ones
in~\eqref{eq:zustI}, are well-defined. In fact, the elements $g^0$,
$\diff g^1, \dots, \diff g^d$ can be seen as fractional charges. We
construct more generally a wedge product between fractional charges:
if $\omega$ and $\eta$ are respectively $\alpha$- and
$\beta$-fractional charges (of any order), we make sense of $\omega
\wedge \eta$ provided that the Young-type condition $\alpha + \beta >
1$ is met. The resulting product $\omega \wedge \eta$ is an $(\alpha +
\beta - 1)$-fractional charge. Using this wedge product, $g^0 \wedge
\diff g^1 \wedge \cdots \wedge \diff g^d$ is then a well-defined
$(\alpha_0 + \cdots + \alpha_d - d)$-fractional charge, which can
therefore be integrated not only over $[0, 1]^d$, but over any
fractional $d$-current. More generally, the wedge product extends the
Züst integral in any codimension.

Recently, Chandra and Singh undertook a similar effort to develop a
rough geometric integration theory in~\cite{ChanSing}. Their approach
is probabilistic in nature, relying on a multidimensional sewing
lemma. Establishing a connection between their space of chains and
cochains and ours is an open question.

\subsection{Outline of the paper}
Section~\ref{sec:prelim} contains preliminaries, mostly on normal and
flat currents. Unlike the introduction, we reverse the order of
exposition: we begin in Section~\ref{sec:fractC} with the general
definition of fractional currents and their fundamental properties
before turning to the top-dimensional case in
Section~\ref{sec:codim0}, where we show the equivalence between
$\bF_d^\alpha(\R^d)$ and $W^{1-\alpha,
  1}_c(\R^d)$. Sections~\ref{subsec:firstexample}
and~\ref{subsec:zust} present examples of fractional currents; they
are not essential for the remainder of the paper.

In Section~\ref{sec:charges}, we study the space of $m$-charges
$\bCH^m(K)$ over a compact $K \subset \R^d$, whose elements are linear
functionals on $\bN_m(K)$ such that for any sequence $(T_n)$ in
$\bN_m(K)$,
\[
\omega(T_n) \to 0 \text{ if } \sup_{n \geq 0} \bN(T_n) < \infty \text{
  and } \bF(T_n) \to 0.
\]
This continuity condition is weaker than~\eqref{eq:introCF}. The
inclusion $\bCH^{m, \alpha}(K) \subset \bCH^m(K)$ can be informally
viewed as analogous to the inclusion $\rmLip^\alpha(K) \subset
C(K)$. This section contains essentially no new results. Fractional
charges, in duality with fractional currents, are then defined in
Section~\ref{sec:fracCharges}.

Our main results are in the subsequent sections: the compactness
theorem and the pushforward construction are found in
Sections~\ref{sec:compact} and~\ref{sec:holderpf}. Young geometric
integration is the object of Section~\ref{sec:ext}. Finally, we show
in Section~\ref{sec:metricCurrents} that fractional currents are
metric currents of a snowflaked version of Euclidean space.  The last
two sections do not depend on Sections~\ref{sec:compact}
and~\ref{sec:holderpf}.

We are grateful to the two anonymous referees for their careful reading and insightful suggestions, and to Philippe Mathieu for his careful proofreading of the manuscript.

\section{Preliminaries}
\label{sec:prelim}

\subsection{Notations}
\label{subsec:notations}
We will work in the Euclidean space $\R^d$ of dimension $d \geq 1$,
endowed with the standard Euclidean norm $| \cdot |$. For any compact
set $K \subset \R^d$ and $\varepsilon > 0$, the compact
$\varepsilon$-tubular neighborhood of $K$ is denoted
\[
B(K, \varepsilon) = \left\{ x \in \R^d: \operatorname{dist}(x, K) \leq
\varepsilon \right\},
\]
where $\operatorname{dist}(x, K)$ denotes the distance between $x$ and
$K$. The Lebesgue (outer) measure is denoted $\calL^d$.

Let $(X, d)$ be a metric space and $(E, \| \cdot \|)$ be a
finite-dimensional normed linear space.  For any $f \colon X \to E$
and $\alpha \in \mathopen{(} 0, 1 \mathclose{]}$, we denote
\[
\rmLip^\alpha f = \sup \left\{ \frac{\|f(x) - f(y)\|}{d(x, y)^\alpha}
: x, y \in X, x \neq y \right\}
\]
and by $\rmLip^\alpha(X; E)$ the space of maps $f \colon X \to E$ such
that $\rmLip^\alpha f < \infty$. We abbreviate $\rmLip^\alpha(X; \R)=
\rmLip^\alpha(X)$.

The space of continuous functions $X \to E$ is denoted $C(X; E)$. We
abbreviate $C(X; \R) = C(X)$. For any $f \in C(X; E)$ and any subset
$K \subset X$, we write
\[
\| f \|_\infty = \sup \left\{ \|f(x)\| : x \in X\right\} \text{ and }
\|f\|_{\infty, K} = \sup \left\{ \| f(x) \| : x \in K \right\}.
\]
Throughout the paper, multiplicative constants are denoted by $C$; their value may change from line to line, and even within the same line.

\subsection{Currents}
We will work within the setting of Federer-Fleming's currents. For an
in-depth exploration of this subject, we refer the reader to the
treatise~\cite{Fede}. In this preliminary part, our focus will be on
defining common notations, highlighting those that deviate from
\cite{Fede}, and revisiting a few definitions.

The spaces of $m$-vectors and $m$-covectors are $\wedge_m \R^d$ and
$\wedge^m \R^d$, and they are respectively given the mass and comass
norm described in \cite[1.8.1]{Fede}, both of which will be denoted
here by $\| \cdot \|$. We will use the bracket notation $\langle
\cdot, \cdot \rangle$ for the duality between $m$-covectors and
$m$-vectors.

All our currents will be compactly supported, defined on $\R^d$, and
their order will be typically denoted by the integer $m$. We recall
that a compactly supported $m$-current is a member of the topological
dual of $C^\infty(\R^d; \wedge^m \R^d)$, the space of smooth
$m$-forms, equipped with the locally convex topology described
in~\cite[4.1.1]{Fede}. The support $\spt T$ of a compactly supported
$m$-current $T$ is the smallest compact subset of $\R^d$ such that
$T(\omega) = 0$ for all smooth $m$-forms supported in $\R^d \setminus
\spt T$.

A sequence of compactly supported currents
$(T_n)$ is said to converge weakly to $T$ whenever $T_n(\omega) \to
T(\omega)$ for every compactly supported smooth $m$-form $\omega$.

For any compactly supported $m$-current $T$ and any smooth $k$-form
$\eta$, where $k \leq m$, we define the compactly supported
$(m-k)$-current $T \hel \eta$ by $(T \hel \eta)(\omega) = T(\eta
\wedge \omega)$.

The mass of a compactly supported current $T$ is
\[
\bM(T) = \sup \left\{ T(\omega) : \omega \in C^\infty(\R^d; \wedge^m
\R^d) \text{ and } \| \omega \|_\infty \leq 1\right\}.
\]
When $\bM(T) < \infty$, the $m$-current $T$ is representable by
integration: there is a unique Radon measure $\|T\|$, a
$\|T\|$-measurable unit $m$-vector field $\vec{T}$, unique
$\|T\|$-almost everywhere, such that for all $\omega \in
C^\infty(\R^d; \wedge^m \R^d)$,
\[
T(\omega) = \int_{\R^d} \langle \omega(x), \vec{T}(x) \rangle \diff \|T\|(x).
\]
Additionally, one has $\spt \|T\| = \spt T$.

In case $m \geq 1$, we define the boundary $\partial T$ to
be the compactly supported current $\omega \mapsto T(\diff
\omega)$. The normal mass of $T$ is
\[
\bN(T) = \bM(T) + \bM(\partial T)
\]
if $m \geq 1$ and $\bN(T) = \bM(T)$ if $m = 0$. A compactly supported
current $T$ is said to be normal whenever $\bN(T) < \infty$, and the
space of normal $m$-currents is denoted $\bN_m(\R^d)$.

Examples of normal $m$-currents are provided by compactly supported
smooth $m$-currents. Those are the currents of the form
\[
\omega \mapsto \int_{\R^d} \langle \omega(x), \xi(x) \rangle \diff \calL^d (x)
\]
where $\xi \colon \R^d \to \wedge_m \R^d$ is a compactly supported
smooth $m$-vector field. Such a current will be denoted $\calL^d \wedge
\xi$.

We define the flat norm of a normal current $T \in \bN_m(\R^d)$ in a
way which departs slightly from Federer's exposure:
\begin{align*}
  \bF(T) & = \sup \left\{ T(\omega) : \omega \in C^\infty(\R^d;
  \wedge^m \R^d) \text{ and } \max \left\{ \| \omega \|_\infty, \|
  \diff \omega \|_\infty \right\} \leq 1 \right\} \\ & = \inf
  \left\{ \bM(S) + \bM(T - \partial S) : S \in \bN_{m+1}(\R^d)
  \right\}.
\end{align*}
The proof of the above equality is similar to
\cite[4.1.12]{Fede}. From the first equality, it is clear that
convergence in flat norm implies weak convergence.

A compactly supported $m$-current $T$ is said to be flat whenever
there is a sequence $(T_n)$ of normal $m$-currents that is
$\bF$-Cauchy, weakly converges to $T$, and such that $\bigcup_n \spt
T_n$ is relatively compact. The flat norm of $T$ is then defined to be
$\bF(T) = \lim \bF(T_n)$. When $T$ is not flat, we set $\bF(T) =
\infty$. The space of flat $m$-currents is denoted by $\bF_m(\R^d)$.

It may be worth noting that, for every $T \in \bF_m(\R^d)$, one has
\begin{equation}
\label{eq:FMN}
\bF(T) \leq \bM(T) \leq \bN(T),
\end{equation}
a fact that will be used repeatedly.

The pushforward $f_\#T$ of a flat $m$-current $T \in \bF_m(\R^d)$ by a
locally Lipschitz map $f \colon \R^d \to \R^{d'}$ is the flat
$m$-current $f_\#T \in \bF_m(\R^{d'})$ defined as
in~\cite[4.1.14]{Fede}. It depends only on the values of $f$ on $\spt
T$, so we may write $f_\#T$ if the domain of $f$ contains $\spt T$ and
the restriction $f_{\mid \spt T}$ is Lipschitz continuous.

For any compact subset $K \subset \R^d$, we write
\begin{align*}
  \bN_m(K) & = \{T \in \bN_m(\R^d) : \spt T \subset K \} \\
  \bF_m(K) & = \{T \in \bF_m(\R^d) : \spt T \subset K \}
\end{align*}
The construction of charges ultimately relies on the Federer-Fleming's
compactness theorem of normal currents in flat norm. The following
version uses the flat norm $\bF$ (rather than $\bF_K$ as in
\cite[4.2.17(1)]{Fede}). It can be easily deduced from the original
version. Alternatively, it is possible to reproduce the arguments in
the proof of Federer-Fleming, as done in
\cite[Theorem~4.2]{DePaMoonPfef}.
\begin{thm}[Compactness]
  Let $K \subset \R^d$ be compact. For all $c \geq 0$, the ball $\{ T
  \in \bN_m(K) : \bN(T) \leq c\}$ is $\bF$-compact.
\end{thm}

In this article, convolutions will play an important role in
regularizing various chains or cochains. Here, we simply recall how
convolution works at the level of currents.  The convolution of a
compactly supported $m$-current $T$ with a function $\phi \in
C^\infty_c(\R^d)$ is the compactly supported $m$-current defined by
\[
(T * \phi)(\omega) = T(\check{\phi} * \omega) \qquad \text{for all }
\omega \in C^\infty(\R^d; \wedge^m\R^d)
\]
where $\check{\phi}(x) = \phi(-x)$ for all $x \in \R^d$. 
 
Throughout the article, we fix a $C^\infty$ function $\Phi \colon \R^d
\to \R$ with compact support in the unit ball of $\R^d$, that is
nonnegative and such that $\int_{\R^d} \Phi = 1$. For the sake of
simplicity, we will further assume that $\Phi$ is even.  For all
$\varepsilon > 0$, we define $\Phi_\varepsilon(x) = \varepsilon^{-d}
\Phi(\varepsilon^{-1} x)$.

For the reader's convenience, we collect here several well-known facts concerning normal and flat currents that will be used throughout the paper. 
\begin{prop}
  \label{prop:22}
  Let $K$ be a compact subset of $\R^d$, $f, g \colon K \to \R^{d'}$ be two
  Lipschitz maps and $T \in \bF_m(K)$.
  \begin{itemize}
  \item[(A)] If $m \geq 1$, one has
    \[
    \bN(f_\# T) \leq \max \{ (\rmLip f)^m, (\rmLip f)^{m-1} \} \bN(T).
    \]
    If $m = 0$ then $\bN(f_\#T) \leq \bN(T)$.
  \item[(B)] $\bF(f_\# T) \leq \max \{ (\rmLip f)^m, (\rmLip f)^{m+1}
    \} \bN(T)$.
  \item[(C)] $\spt f_\#T \subset f(\spt T)$.
  \item[(D)] If $m \geq 1$, one has
    \[
    \bF(f_\# T - g_\# T) \leq \|f - g\|_{\infty, K} \max\{ (\rmLip
    f)^m, (\rmLip f)^{m-1}, (\rmLip g)^m, (\rmLip g)^{m-1} \}\bN(T).
    \]
    If $m = 0$ then $\bF(f_\# T - g_\# T) \leq \| f - g \|_{\infty, K}
    \bN(T)$.
  \item[(E)] $\bM(T * \Phi_\varepsilon) \leq \bM(T)$, $\bF(T *
    \Phi_\varepsilon) \leq \bF(T)$ and $\bN(T * \Phi_\varepsilon) \leq
    \bN(T)$.
  \item[(F)] $\bN(T * \Phi_\varepsilon) \leq C \varepsilon^{-1} \bF(T)$ if $\varepsilon \leq 1$.
  \item[(G)] $\bF(T - T * \Phi_\varepsilon) \leq \varepsilon \bN(T)$.
  \item[(H)] $\spt (T * \Phi_\varepsilon) \subset B(\spt T,
    \varepsilon)$.
  \end{itemize}
\end{prop}
\noindent Specifically, we refer to \cite[4.1.13, 4.1.14]{Fede} for properties (A) to (D), while (E) to (G) follow mainly from \cite[4.1.2]{Fede}. In (F), the constant $C$ depends only on the dimension $d$ (through the choice of $\Phi$). Recall from~\eqref{eq:FMN} that the flat norm $\bF$ is weaker than the norm $\bN$. This provides a useful heuristic for the appearance of the coefficients: an $\varepsilon^{-1}$ factor in~(F), where one estimates a stronger quantity in terms of a weaker one, and an $\varepsilon$ factor in~(G), reflecting the converse situation.

\section{Fractional currents}
\label{sec:fractC}

\subsection{Definition and first properties}
\label{def:frac}

Let $\alpha \in [0, 1]$. We say that a flat $m$-current $T \in
\bF_m(\R^d)$ is \emph{$\alpha$-fractional} whenever one can find a
sequence $(T_k)_{k \geq 0}$ of normal currents satisfying the
following properties:
\begin{itemize}
\item[(A)] there is a compact set $K$ such that $\spt T_k \subset K$
  for all $k$;
\item[(B)] $\sum_{k=0}^\infty \bN(T_k)^{1 - \alpha} \bF(T_k)^\alpha <
  \infty$;
\item[(C)] $T = \sum_{k = 0}^\infty T_k$.
\end{itemize}
We recall that $\bF \leq \bM \leq \bN$, which ensures that
\[
\sum_{k=0}^\infty \bF(T_k) \leq \sum_{k=0}^\infty \bN(T_k)^{1 -
  \alpha} \bF(T_k)^\alpha < \infty
\]
by (B). This guarantees that the series in (C) is in fact
$\bF$-convergent, and it is in this sense that it should be
understood. A sequence $(T_k)$ that satisfies (A), (B) and (C) is
called an \emph{$\alpha$-fractional decomposition} of $T$. This
paragraph makes it clear that any permutation of the sequence $(T_k)$
is still an $\alpha$-fractional decomposition of $T$. In the remainder
of the paper, we may at times consider decompositions indexed over a
countable set other than $\{0, 1,2, \dots\}$.

We will denote the space of $\alpha$-fractional currents by
$\bF_m^\alpha(\R^d)$, and associate to each $T \in \bF_m^\alpha(\R^d)$
the \emph{$\alpha$-fractional norm}
\begin{equation}
  \label{eq:normFa}
  \bF^\alpha(T) = \inf \sum_{k=0}^\infty \bN(T_k)^{1 - \alpha}
  \bF(T_k)^\alpha
\end{equation}
where the infimum is taken over all $\alpha$-fractional decompositions
of $T$ (the fact that $\bF^\alpha$ is a norm will become apparent in
Proposition~\ref{prop:basic}). If $T \in \bF_m(\R^d)$ is not
$\alpha$-fractional, we set $\bF^\alpha(T) = \infty$.

Additionally, for compact sets $K \subset \R^d$, we will denote by
$\bF_m^\alpha(K)$ the space of $\alpha$-fractional currents $T \in
\bF_m^\alpha(\R^d)$ such that $\spt T \subset K$. In case $K$ is
compact convex and $T \in \bF^\alpha(K)$, the infimum
in~\eqref{eq:normFa} can be taken only over $\alpha$-fractional
decompositions $(T_k)$ of $T$ into normal currents that are all
supported in $K$ (as we can replace each $T_k$ by its pushforward by
the orthogonal projection onto $K$). This fact will be often used
implicitly.

The following comments are in order.
\begin{itemize}
\item The boundary cases where $\alpha = 0$ or $\alpha = 1$ correspond
  to well-established classes of currents. Specifically, it can be
  easily verified that $\bF_m^0(\R^d)$ is the space of normal currents
  (with $\bF^0$ being the normal mass $\bN$) while $\bF^1_m(\R^d)$ is
  the space of flat currents (with $\bF^1 = \bF$). As we will observe
  later, many properties of fractional currents hold exclusively for
  $0 < \alpha < 1$, and the extreme cases demand closer scrutiny.
\item The definition of fractional currents can be readily extended to
  the metric setting, with coefficients in a complete normed Abelian
  group, as defined in \cite{DePaHard}. In this context, the sequence
  of normal currents $(T_k)$ is to be replaced with rectifiable
  $G$-chains. However in this work, we focus exclusively on Euclidean
  fractional currents with real coefficients, the only case for which
  a compactness theorem is currently known.
\item We will aim to convince the reader that an $\alpha$-fractional
  current, while being an $m$-dimensional object (in the sense that it
  can be integrated against $m$-forms), is characterized by the
  property that its ``fractal dimension'' (a loosely defined notion)
  is at most $m+\alpha$, and that the fractal dimension of its
  boundary at most $m-1+\alpha$. Thus, the parameter $\alpha$ can be
  understood as representing an upper bound of the fractional part of
  the fractal dimension.
\end{itemize}
In the next proposition, we gather some immediate properties of
fractional currents.

\begin{prop}
  \label{prop:basic}
  Let $K \subset \R^d$ be a compact set. The following hold.
  \begin{itemize}
    \item[(A)] If $\beta$ is such that $\alpha \leq \beta \leq 1$,
      then $\bF_m^\alpha(\R^d) \subset \bF_m^\beta(\R^d)$, and
      $\bF^\beta \leq \bF^\alpha$.
    \item[(B)] For $m \geq 1$, the boundary operator $\partial$ takes
      $\bF_m^\alpha(\R^d)$ to $\bF_{m-1}^\alpha(\R^d)$ and
      $\bF^\alpha(\partial T) \leq \bF^\alpha(T)$ for all $T \in
      \bF_m^\alpha(\R^d)$.
    \item[(C)] $\bF_m^\alpha(\R^d)$ is a linear space, normed by
      $\bF^\alpha$.
    \item[(D)] $(\bF_m^\alpha(K), \bF^\alpha)$ is a Banach space.
    \item[(E)] In case $K$ is convex, $\bN_m(K)$ is dense in
      $\bF^\alpha_m(K)$.
    \item[(F)] If $T \in \bF^\alpha_m(\R^d)$ and $f \colon \spt T \to
      \R^{d'}$ is a Lipschitz map, then $f_{\#}T \in
      \bF^\alpha_m(\R^{d'})$ and
      \begin{equation}
        \label{eq:pushLip}
        \bF^\alpha(f_{\#} T) \leq \max \{ (\rmLip f)^{m-1+\alpha},
        (\rmLip f)^{m+\alpha} \} \bF^\alpha(T)
      \end{equation}
      if $m \geq 1$, and
      \[
      \bF^\alpha(f_\# T) \leq \max \{1, (\rmLip f)^{\alpha} \} \bF^\alpha(T)
      \]
      if $m = 0$.
  \end{itemize}
\end{prop}

\begin{proof}
  (A) and (B) are easy consequences of the inequalities
  \[
  \bF(T) \leq \bN(T), \qquad \bF(\partial T) \leq \bF(T), \qquad
  \bN(\partial T) = \bM(\partial T) \leq \bN(T)
  \]
  that hold for any normal current $T$.
  
  (C). As $\bF = \bF^1 \leq \bF^\alpha$ by (A) and $\bF$ is a norm, the
  equality $\bF^\alpha(T) = 0$ readily implies that $T = 0$.

  The set $\bF_m^\alpha(\R^d)$ is closed under addition, for if
  $(T_k)$ and $(S_k)$ are $\alpha$-fractional decompositions of $T$
  and $S$, we can construct the decomposition $(R_k)$ of $T+S$, with
  $R_{2k} = T_k$ and $R_{2k+1} = S_k$. Furthermore, one has
  \[
  \bF^\alpha(T + S) \leq \sum_{k=0}^\infty \bN(R_k)^{1 - \alpha}
  \bF(R_k)^{\alpha} = \sum_{k=0}^\infty \bN(T_k)^{1 - \alpha}
  \bF(T_k)^{\alpha} + \sum_{k=0}^\infty \bN(S_k)^{1 - \alpha}
  \bF(S_k)^{\alpha}
  \]
  Taking the infimum on the right-hand side over all decompositions
  $(T_k)$ and $(S_k)$ yields the triangle inequality for
  $\bF^\alpha$. 

  (D). To prove that $\bF_m^\alpha(K)$ is complete, it is sufficient
  to prove that an $\bF^\alpha$-absolutely convergent series $\sum_n
  T^{(n)}$ actually converges in $\bF^\alpha_m(K)$. The proof of this
  is fairly standard; we do it by considering for each $T^{(n)}$ an
  $\alpha$-fractional decomposition $(T_k^{(n)})_k$ with
  \[
  \sum_{k=0}^\infty \bN(T_k^{(n)})^{1 - \alpha} \bF(T_k^{(n)})^\alpha
  < \bF^\alpha(T^{(n)}) + \frac{1}{2^n}
  \]
  Without loss of generality, we can assume that all of $T_k^{(n)}$
  are supported in the same compact set ($K$ is contained in a closed
  ball, then replace if necessary $T_k^{(n)}$ by its projection onto
  this ball). Then define
  \[
  T = \sum_{n,k} T_k^{(n)}
  \]
  This series is well-defined (it is in fact absolutely convergent in
  flat norm). Clearly, $T$ is $\alpha$-fractional, as
  $(T_{k}^{(n)})_{k,n}$ is an $\alpha$-fractional decomposition, and
  \[
  \bF^\alpha\left(T - \sum_{n=0}^N T^{(n)} \right) \leq
  \sum_{n=N+1}^\infty \left( \bF^\alpha(T^{(n)}) + \frac{1}{2^n}
  \right) \to 0
  \]
  as $N \to \infty$.

  (E). Let $T \in \bF^\alpha_m(K)$ and $(T_k)$ be an
  $\alpha$-fractional decomposition of $T$ such that all $T_k$ are
  supported in $K$. It is clear that
  \[
  \bF^\alpha \left( T - \sum_{k=0}^N T_k \right) \leq
  \sum_{k=N+1}^\infty \bN(T_k)^{1 - \alpha} \bF(T_k)^\alpha \to 0
  \]
  as $N \to \infty$.
  
  (F). First extend $f$ to a Lipschitz map $\R^d \to \R^{d'}$, while
  preserving the Lipschitz constant $L = \rmLip f$. Let $(T_k)$ be an
  $\alpha$-fractional decomposition of $T$. Observe that $(f_\# T_k)$
  is an $\alpha$-fractional decomposition of $f_\# T$. Thus, $f_\#T \in
  \bF^\alpha_m(\R^{d'})$ and, if $m \geq 1$,
  \begin{align*}
    \bF^\alpha(f_\#T) & \leq \sum_{k=0}^\infty \bN(f_\# T_k)^{1 -
      \alpha} \bF(f_\# T_k)^{\alpha} \\ & \leq \sum_{k=0}^\infty \max
    \{ L^{(m-1)(1-\alpha)}, L^{m(1 - \alpha)} \} \bN(T_k)^{1 - \alpha}
    \max \{ L^{m \alpha}, L^{(m+1)\alpha} \} \bF(T_k)^\alpha \\ & \leq
    \max \{ L^{m-1+\alpha}, L^{m+\alpha} \} \sum_{k=0}^{\infty}
    \bN(T_k)^{1 - \alpha} \bF(T_k)^{\alpha}
  \end{align*}
  Taking the infimum over all $\alpha$-fractional decompositions of
  $T$ results in the inequality~\eqref{eq:pushLip}. The case $m = 0$
  is treated similarly.
\end{proof}

\subsection{A first example}
\label{subsec:firstexample}

Many fractal-like objects can be seen as fractional currents. Although
a complete characterization of fractional currents with codimension
$0$ will be provided in Section~\ref{sec:codim0}, we introduce a
preliminary example here. Recall that a \emph{dyadic $k$-cube} (in
short: $k$-cube) is a set of the form
\begin{equation*}
  %\label{eq:kcube}
Q = \prod_{i=1}^d \left[ \frac{\ell_i}{2^k}, \frac{\ell_i+1}{2^k}
  \right]
\end{equation*}
where $k, \ell_1, \dots, \ell_d$ are integers. Let $A \subset \R^d$ be
a bounded subset of $\R^d$. For any integer $k$, we let $N_A(k)$ be
the number of $k$-cubes that intersect $A$. Following~\cite{HarrNort}, we
say that $A$ is \emph{$m$-summable} (where $m \geq 0$ is real) whenever
\[
\sum_{k=0}^\infty N_A(k) 2^{-km} < \infty
\]
Summability is related to the \emph{upper box dimension} of $A$, which is defined by
\[
\dim_{\mathrm{box}} A = \limsup_{k \to \infty} \frac{\log N_A(k)}{k
  \log 2}
\]
Indeed, one can easily check that $A$ is $m$-summable for all $m >
\dim_{\mathrm{box}} A$.

\begin{prop}
  \label{prop:firstexample}
  Let $U \subset \R^d$ be a bounded open set of $\R^d$ such that
  $\partial U$ is $(d-1+\alpha)$-summable, for $\alpha \in [0,
    1]$. Then $\llbracket U \rrbracket$ is $\alpha$-fractional.
\end{prop}

\begin{proof}
  This proof is heavily influenced by \cite[Section~3]{HarrNort} and
  \cite[Lemma~4.4]{Zust}. We assume that $U$ is nonempty, otherwise
  there is nothing to prove. The main ingredient will be the Whitney
  decomposition of $U$, that we now describe.

  As $U$ is bounded and nonempty, there is a smallest integer $k_0
  \geq 0$ such that some $k_0$-cube and its $3^d-1$ neighbors are
  contained in $U$. Let $\calQ_{k_0}$ be the collection of such
  $k_0$-cubes. Then we define inductively the collections
  $\calQ_{k_0+1}, \calQ_{k_0+2}, \dots$, where for all $k \geq k_0 +
  1$, the collection $\calQ_k$ consists of all $k$-cubes $Q$ such that
  \begin{itemize}
  \item[(A)] $Q$ and all of its $3^d - 1$ neighbors are contained in $U$;
  \item[(B)] $Q$ is not contained in any cube of $\calQ_\ell$, for
    some $k_0 \leq \ell < k$.
  \end{itemize}
  Let $k \geq k_0 + 1$. By construction, the parent $(k-1)$-cube of
  any $Q \in \calQ_k$, or one of its neighbors, intersects $\partial
  U$. Therefore,
  \[
  \operatorname{Card} \calQ_k \leq 3^d \times 2^d N_{\partial U}(k-1)
  \leq C N_{\partial U}(k-1)
  \]
  Note that $\calQ_{k_0}$ is finite as well by boundedness of $U$. We
  define, for each $k \geq k_0$, the normal current
  \[
  T_k = \sum_{Q \in \calQ_k} \llbracket Q \rrbracket
  \]
  Moreover, one clearly has
  \begin{gather*}
    \bN(T_k) \leq \operatorname{Card} \calQ_k \left( 2^{-kd} + 2d 2^{-k(d-1)} \right) \leq C \left(\operatorname{Card} \calQ_k\right) 2^{-k(d-1)} \\
    \bF(T_k) \leq \bM(T_k) \leq \left(\operatorname{Card} \calQ_k\right) 2^{-kd}
  \end{gather*}
  Consequently,
  \[
  \sum_{k=k_0}^\infty \bN(T_k)^{1-\alpha} \bF(T_k)^{\alpha} \leq C
  \sum_{k=k_0}^\infty (\operatorname{Card} \calQ_k)
  2^{-k(d-1+\alpha)} < \infty
  \]
  Hence the sequence of normal currents $(T_k)$ satisfies
  condition~(B) of Definition~\ref{def:frac}.
\end{proof}

\subsection{Züst's fractal currents}
\label{subsec:zust}

In this subsection, we exhibit another class of currents that are
proven to be $\alpha$-fractional. It was introduced
in~\cite[Definition~4.1]{Zust}. The definition is adapted here to the
Euclidean setting, using Federer-style currents instead of metric
currents.

Let $m \leq \gamma < m+1$ and $m-1 \leq \delta < m$. An
$m$-dimensional flat current $T \in \bF_m(\R^d)$ is said to be
\emph{$(\gamma,\delta)$-fractal} in the sense of Züst whenever there
are sequences $(R_k)$ in $\bN_{m}(\R^d)$ and $(S_k)$ in
$\bN_{m+1}(\R^d)$, and parameters $\sigma, \rho > 1$ (depending on $T$) such that
\begin{itemize}
\item[(A)] there is a compact set $K \subset \R^d$ such that the
  supports of $R_k$, $S_k$ all lie in $K$;
\item[(B)] one has
  \begin{gather*}
  \sum_{k=0}^\infty \bM(S_k) \sigma^{k(m+1-\gamma)} < \infty, \qquad
  \sum_{k=0}^\infty \bM(\partial S_k) \sigma^{k(m-\gamma)} < \infty
  \\ \sum_{k=0}^\infty \bM(R_k) \rho^{k(m-\delta)} < \infty, \qquad
  \sum_{k=0}^\infty \bM(\partial R_k) \rho^{k(m-1-\delta)} < \infty
  \end{gather*}
\item[(C)] $T = \sum_{k=0}^\infty (R_k + \partial S_k)$ weakly.
\end{itemize}
A helpful way to think about $T$ is as having a fractal dimension less
than $\gamma$, while its boundary can be seen as having a fractal
dimension less than $\delta$. One limitation of Züst’s fractal currents, compared to fractional currents, is that they do not form a linear space. Fractional currents are more general, as we now show:
\begin{prop}
A $(\gamma, \delta)$-fractal current (in the sense of Züst) is $\alpha$-fractional for $\alpha = \max \{\gamma - m, \delta - (m-1) \}$.
\end{prop}

\begin{proof}
Let $T$ be a $(\gamma, \delta)$-fractal current. We adopt the notations introduced at the beginning of the subsection.
First we estimate
\begin{align*}
\bN(\partial S_k)^{1 - \alpha} \bF(\partial S_k)^{\alpha} & \leq
\bM(\partial S_k)^{1-\alpha} \bM(S_k)^{\alpha} \\
& \leq \left( \bM(\partial S_k) \sigma^{-k(m+1-\gamma)\alpha/(1-\alpha) } \right)^{1-\alpha}  \left( \bM(S_k) \sigma^{k(m+1-\gamma)} \right)^{\alpha} \\ & \leq (1 - \alpha)
\bM(\partial S_k) \sigma^{-k(m+1-\gamma)\alpha/(1-\alpha) } + \alpha \bM(S_k)
\sigma^{k(m+1-\gamma)}
\end{align*}
by Young's inequality. Using $\alpha \geq \gamma - m$ and $\sigma > 1$,
one further obtains
\begin{equation}
  \label{eq:estimDSk}
\bN(\partial S_k)^{1 - \alpha} \bF(\partial S_k)^{\alpha} \leq (1 - \alpha)
\bM(\partial S_k) \sigma^{k(m-\gamma) } + \alpha \bM(S_k)
\sigma^{k(m+1-\gamma)}
\end{equation}
Similarly,
\begin{align}
  \bN(R_k)^{1-\alpha} \bF(R_k)^\alpha & \leq (1-\alpha) \bN(R_k)
  \rho^{-k(m-\delta) \alpha/(1-\alpha)} + \alpha \bF(R_k) \rho^{k(m-
    \delta)} \notag \\ & \leq (1-\alpha) \bN(R_k) \rho^{k(m-1-\delta)}
  + \alpha \bF(R_k) \rho^{k(m- \delta)} \notag \\ & \leq \bM(R_k)
  \rho^{k(m-\delta)}\left( \alpha + \frac{1-\alpha}{\rho^k} \right) +
  (1-\alpha)\bM(\partial R_k) \rho^{k(m-1-\delta)} \label{eq:estimRk}
\end{align}
The inequalities~\eqref{eq:estimDSk}, \eqref{eq:estimRk} and (B) show
that $(T_k)$, where $T_{2k} = \partial S_k$ and $T_{2k+1} = R_k$, is
an $\alpha$-fractional decomposition of $T$. Hence $T \in
\bF^\alpha_m(\R^d)$. This shows, by the way, that the sum in~(C) is
$\bF$-convergent.
\end{proof}

\section{Fractional Sobolev regularity of zero-codimensional fractional currents}
\label{sec:codim0}

It is well-known that normal currents of top dimension have the form
$\llbracket u \rrbracket$, where $u$ is a compactly supported $BV$
function, that is, a $L^1_c$ function such that the De Giorgi
variation
\[
\|Du\|(\R^d) = \sup \left\{ \int_{\R^d} u \operatorname{div} v : v \in
C^1_c(\R^d; \R^d) \text{ and } \|v\|_\infty \leq 1 \right\}
\]
is finite, in which case $\|Du\|(\R^d) = \bM(\partial \llbracket u
\rrbracket)$, see \cite[Section~4.5]{Fede}. As for flat currents, they are the currents of the form
$\llbracket u \rrbracket$, where $u$ is a compactly supported
integrable function, see \cite[4.1.18]{Fede}. We will characterize in Theorem~\ref{thm:topdim}
the functions $u$ such that $\llbracket u \rrbracket \in
\bF^\alpha_d(\R^d)$ as members of a fractional Sobolev
space. \textbf{We will assume throughout this section that $0 < \alpha
  < 1$.}

\subsection{Fractional Sobolev space}
Here, we will simply recall the definition of fractional Sobolev
spaces, referring the reader to~\cite{DiNePalaVald} for a
comprehensive discussion on the subject.

The \emph{$W^{1-\alpha,1}$-norm} of an $L^1$ function $u \colon \R^d
\to \R$ is defined to be
\[
\| u \|_{W^{1-\alpha, 1}(\R^d)} = \int_{\R^d} \int_{\R^d} \frac{|u(x)
  - u(y)|}{|x - y|^{d+1-\alpha}} \diff x \diff y.
\]
This quantity defines a genuine norm (and not merely a seminorm) on
$L^1(\R^d)$, since the only constant function belonging to $L^1(\R^d)$
is the zero function.

In the case where $u$ is the indicator function of a Lebesgue
measurable set $A$, we call it the \emph{$(1-\alpha)$-fractional
  perimeter} of $A$:
\[
\operatorname{Per}_{1 - \alpha}(A) = \| \ind_A
\|_{W^{1-\alpha,1}(\R^d)} = 2 \int_A \int_{\R^d \setminus A}
\frac{\diff x \diff y}{| x - y |^{d + 1 - \alpha}}
\]
We let $W^{1-\alpha,1}_c(\R^d)$ be the space of all compactly
supported integrable functions $u$ such that $\|u
\|_{W^{1-\alpha,1}(\R^d)} < \infty$.

%We will also make use of the fractional Sobolev space
%$W^{1-\alpha,1}((0, 1)^d)$ over the open cube $(0, 1)^d$, which
%consists of the functions $u \in L^1((0, 1)^d)$ for which the
%semi-norm
%\[
%[u]_{W^{1-\alpha,1}((0, 1)^d)} = \int_{(0,1)^d} \int_{(0, 1)^d}
%\frac{|u(x) - u(y)|}{|x - y|^{d+1-\alpha}} \diff x \diff y
%\]
%is finite. The space $W^{1-\alpha,1}(\R^d)$ is normed by
%\[
%\|u\|_{W^{1-\alpha,1}((0,1)^d)} = \|u\|_1 + [u]_{W^{1-\alpha,1}((0, 1)^d)}
%\]
%
%For each integer $k \geq 0$, we let $\calP_k$ be the collection of the
%$2^k$ dyadic cubes of sidelength $2^{-k}$ and $E_k \colon L^1((0,
%1)^d) \to L^1((0, 1)^d)$ that maps $u$ be the function $E_k(u) \colon
%(0, 1)^d \to \R$ that is equal on each $Q \in \calP_k$ to the average
%of $u$ on $Q$
%\[
%E_k(u) = \sum_{Q \in \calP_k} \left( \dashint_Q u \right) \ind_{Q}
%\]
%(\textit{i.e} $E_k(u)$ is the conditional expectation of $u$ with
%respect to the $\sigma$-algebra generated by the dyadic cubes of
%sidelength $2^{-k}$). The following result is due to G. Bourdaud
%\cite{Bour}. It will be needed in the course of the proof of
%Theorem~\ref{thm:topdim}. It is a sort of discrete Littlewood-Paley
%characterization of $W^{1-\alpha, 1}((0, 1)^d)$. For a proof, we refer
%to the paper \cite[Theorem~A.1]{BourBrezMiro}.
%
%\begin{thm}
%  \label{thm:WLP}
%  There is a constant $C > 0$ such that, for all $u \in L^1((0, 1)^d)$, on%e has
%  \[
%  \frac{[u]_{W^{1-\alpha,1}((0, 1)^d)}}{C} \leq \sum_{k=1}^\infty 2^{(1
%      - \alpha)k} \| E_k(u) - E_{k-1}(u) \|_1 \leq C
%    [u]_{W^{1-\alpha,1}((0, 1)^d)}
%  \]
%\end{thm}

\subsection{A norm equivalent to $\bF^\alpha$}
We observe that the $W^{1-\alpha,1}$-norm behaves well under
rescaling. Indeed,
\[
\|u \circ \varphi_r^{-1}\|_{W^{1-\alpha,1}(\R^d)} = r^{d-1+\alpha} \|u
\|_{W^{1-\alpha,1}(\R^d)}
\]
for $r > 0$, where $\varphi_r(x) = r x$ for all $x \in \R^d$.

On the space $\bF_d^\alpha(\R^d)$ of top-dimensional $\alpha$-fractional currents, it will be more convenient to work with a
norm that has better homogeneity properties than $\bF^\alpha$. We
define
\[
\tilde{\bF}^\alpha(T) = \inf \sum_{k=0}^\infty \bM(\partial T_k)^{1 -
  \alpha} \bM(T_k)^\alpha
\]
where the infimum is taken over all $\alpha$-fractional decompositions
of $T$. One checks that $\tilde{\bF}^\alpha(\varphi_{r\#} T) =
r^{d-1+\alpha} \tilde{\bF}^\alpha(T)$.

The norms $\bF^\alpha$ and $\tilde{\bF}^\alpha$ are equivalent, in a
sense that we now describe. Of course, one has $\tilde{\bF}^\alpha
\leq \bF^\alpha$, as $\bF = \bM$ for top-dimensional currents. As for
the converse inequality, we will prove that
\begin{equation}
  \label{eq:equivFFtilde}
\bF^\alpha(T) \leq C(K) \tilde{\bF}^\alpha(T) \qquad
\text{for all } T \in \bF^\alpha_d(K)
\end{equation}
where $K \subset \R^d$ is compact and $C(K)$ is a constant depending
only on $K$.

We lose no generality in assuming that $K$ is convex, as we can
replace it with its closed convex hull. Consider an
$\alpha$-fractional decomposition $(T_k)$ of $T$, with $\spt T_k
\subset K$ for all $k$. Each $T_k = \llbracket u_k \rrbracket$ is
represented by a $BV$ function $u_k$, supported in $K$.  By the Hölder
and Sobolev inequality (see \cite[Theorem~5.10(i)]{EvanGari}), one has
\[
\bM(T_k) = \|u_k\|_1 \leq \|u_k\|_{d/(d-1)} |K|^{1/d} \leq C \|D u_k\|(\R^d)
|K|^{1/d}
\]
where $C$ denotes here the Sobolev constant. Hence
\[
\bN(T_k) \leq (1 + C|K|^{1/d}) \|D u_k\|(\R^d)
\]
and
\[
\bN(T_k)^{1-\alpha} \bF(T_k)^\alpha \leq \left( 1 + C |K|^{1/d}
\right)^{1-\alpha} \bM(\partial T_k)^{1-\alpha} \bM(T_k)^\alpha
\]
Summing over $k$ and then passing to the infimum
yields~\eqref{eq:equivFFtilde}.

\begin{thm}
  \label{thm:topdim}
  Let $u$ be a compactly supported integrable function. The current
  $\llbracket u \rrbracket$ is $\alpha$-fractional if and only if $u
  \in W^{1 - \alpha, 1}_c(\R^d)$, in which case
  \[
  \frac{1}{C} \tilde{\bF}^\alpha(\llbracket u \rrbracket)
  \leq\|u\|_{W^{1-\alpha, 1}(\R^d)} \leq C
  \tilde{\bF}^\alpha(\llbracket u \rrbracket)
  \]
  for some constant $C = C(d, \alpha) > 0$.
\end{thm}

\begin{proof}
  \textbf{Step 1. For $u \in C^1(\R^d) \cap L^1(\R^d)$, we have
    $\|u\|_{W^{1-\alpha, 1}(\R^d)} \leq C\|D u\|(\R^d)^{1 - \alpha}
    \|u\|_1^\alpha$, where $C = C(d, \alpha)$ is a constant.}

  To our knowledge, this step and the next one were first proven
  by Brasco, Lindgren and Parini in~\cite[Proposition~4.2]{BrasLindPari}. For the reader's convenience, we reproduce their proof.
  
  We can assume
  that $u \neq 0$ and $\|D u\|(\R^d)< \infty$. Let $\eta > 0$ be a real
  number, to be chosen adequately later. By a change of variables,
  \[
  \|u\|_{W^{1-\alpha,1}(\R^d)} = \int_{\R^d} \int_{\R^d} \frac{|u(x) -
    u(x - z)|}{|z|^{d+1-\alpha}} \diff x \diff z
  \]
  We then split
  \begin{multline*}
  \|u\|_{W^{1 - \alpha, 1}(\R^d)} = \int_{ \{|z| \leq \eta \} }
  \int_{\R^d} \frac{|u(x) - u(x - z)|}{|z|^{d+1-\alpha}} \diff x \diff
  z + \\ \int_{ \{|z| > \eta \} } \int_{\R^d} \frac{|u(x) - u(x -
    z)|}{|z|^{d+1-\alpha}} \diff x \diff z
  \end{multline*}
  The second term in the right-hand side is controlled by
  \[
  \int_{ \{|z| > \eta \} } \int_{\R^d} \frac{|u(x) - u(x -
    z)|}{|z|^{d+1-\alpha}} \diff x \diff z \leq \int_{ \{ |z| > \eta
    \} } \frac{2 \|u\|_1 }{|z|^{d+1 - \alpha}} \diff z =
  C \eta^{\alpha - 1} \|u\|_1
  \]
  As for the first term, we estimate
  \[
  |u(x) - u(x-z)| \leq |z| \int_0^1 | \nabla u(x - tz) | \diff t
  \]
  Thus, by Fubini's theorem,
  \begin{align*}
    \int_{ \{|z| \leq \eta \} } \int_{\R^d} \frac{|u(x) - u(x -
      z)|}{|z|^{d+1-\alpha}} \diff x \diff z & \leq \int_{ \{|z| \leq
      \eta \} } \int_{\R^d} \int_0^1 \frac{|\nabla u(x -
      tz)|}{|z|^{d-\alpha}} \diff t \diff x \diff z \\ & \leq \int_{
      \{|z| \leq \eta \} } \int_0^1 \int_{\R^d} \frac{|\nabla u(x -
      tz)|}{|z|^{d-\alpha}} \diff x \diff t \diff z \\ & = \int_{
      \{|z| \leq \eta \} } \frac{\|D u\|(\R^d)}{|z|^{d-\alpha}} \diff z \\ & =
    C \eta^\alpha \|D u\|(\R^d)
  \end{align*}
  Summing up, we find
  \[
  \|u\|_{W^{1-\alpha,1}(\R^d)} \leq C \left( \eta^\alpha \| D
  u\|(\R^d) + \eta^{\alpha - 1} \|u\|_1 \right).
  \]
  Choosing $\eta = \|u\|_1 / \|D u\|(\R^d)$, we finish the proof of
  the first step.

  \textbf{Step 2. For $u \in BV(\R^d)$, we have $\|u\|_{W^{1-\alpha,
      1}(\R^d)} \leq C\|D u\|(\R^d)^{1 - \alpha} \|u\|_1^\alpha$, where $C =
  C(d, \alpha)$ is a constant.}

  Indeed, by the approximation theorem \cite[Theorem~5.2.2]{EvanGari},
  there is a sequence $(u_n)$ in $BV(\R^d) \cap C^1(\R^d)$ such that
  $u_n \to u$ in $L^1(\R^d)$ and $\| Du_n\|(\R^d) \to \|D
  u\|(\R^d)$. We may also assume that $u_n \to u$ pointwise. We next
  apply Fatou lemma to infer that
  \[
  \|u\|_{W^{1-\alpha, 1}(\R^d)} \leq \liminf_{n \to \infty}
  \|u_n\|_{W^{1-\alpha, 1}(\R^d)} \leq C \|D u\|(\R^d)^{1-\alpha}
  \|u\|_1^\alpha.
  \]

  \textbf{Step 3. Let $u$ be an integrable compactly supported
  function. If $\llbracket u \rrbracket \in \bF^\alpha_d(\R^d)$, then
  $\|u\|_{W^{1-\alpha,1}(\R^d)} \leq C \tilde{\bF}^\alpha(\llbracket u
  \rrbracket)$.}

  Let $\left(\llbracket u_k \rrbracket\right)$ be an
  $\alpha$-fractional decomposition of $\llbracket u \rrbracket$. The
  $L^1$-convergence of the series $\sum u_k$ implies that some
  subsequence of the sequence of partial sums converges almost
  everywhere to $u$. It follows that $|u(x)-u(y)| \leq \sum_{k=0}^\infty
  |u_k(x) - u_k(y)|$ for almost all $x, y$. Thus,
  \[
  \|u\|_{W^{1-\alpha, 1}(\R^d)} \leq \sum_{k=0}^\infty
  \|u_k\|_{W^{1-\alpha,1}(\R^d)} \leq C \sum_{k=0}^\infty \|D
  u_k\|(\R^d)^{1-\alpha} \|u_k\|_1^\alpha.
  \]
  We conclude by taking the infimum in the right-hand side over all
  possible decompositions of $\llbracket u \rrbracket$.

  \textbf{Step 4. If $u \in W^{1-\alpha,1}_c(\R^d)$, then $\llbracket
  u \rrbracket \in \bF^\alpha_d(\R^d)$ and
  $\tilde{\bF}^\alpha(\llbracket u \rrbracket) \leq C
  \|u\|_{W^{1-\alpha, 1}(\R^d)}$.}

  Some elements here are inspired by the proof
  of~\cite[Theorem~A.1]{BourBrezMiro}.  As the norms $\| \cdot
  \|_{W^{1-\alpha, 1}(\R^d)}$ and $\tilde{\bF}^\alpha$ behave
  identically under rescaling and translation, we may assume without
  loss of generality that $\spt u \subset [0, 1]^d$.  For each integer
  $k \geq 0$, we let $\calP_k$ be the collection of dyadic $k$-cubes
  in $[0, 1]^d$. We define
  \[
  v_k = \sum_{Q \in \calP_k} \left( \frac{1}{|Q|} \int_Q u \right) \ind_Q
  \]
  It is well-known that $v_k \to u$ in $L^1(\R^d)$. We define the
  following $BV$ functions
  \[
  u_0 = v_0 \text{ and } u_k = v_k - v_{k-1} \text{ for all } k \geq 1
  \]
  all of them being compactly supported in $[0, 1]^d$.

  We wish to prove that the decomposition $\left(\llbracket u
  \rrbracket = \sum_{k=0}^\infty \llbracket u_k \rrbracket \right)$ is
  an $\alpha$-fractional decomposition of $\llbracket U \rrbracket$.

  We deal with the function $u_0 = \left(\int_{[0, 1]^d} u \right)
  \ind_{[0, 1]^d}$ separately. We have
  \[
  \| D u_0 \|(\R^d) \leq \left| \int_{[0, 1]^d} u \right| \| D \ind_{[0, 1]^d}\|(\R^d) \leq \|u\|_1 \cdot 2d.
  \]  
  Therefore
  \[
  \|D u_0\|(\R^d)^{1-\alpha} \| u_0 \|_1^\alpha \leq (2d)^{1-\alpha}
  \|u\|_1.
  \]
  Moreover, recalling that $\spt u \subset [0, 1]^d$, one has
  \begin{align*}
  \|u \|_{W^{1-\alpha, 1}(\R^d)} 
   & = \int_{\R^d} \int_{\R^d} \frac{|u(x) - u(y)|}{|x-y|^{d+1-\alpha}} \diff x \diff y \\ 
   & \geq \int_{[0, 1]^d} \left( \int_{\R^d \setminus [-1, 2]^d} \frac{|u(x) - u(y)|}{|x-y|^{d+1-\alpha}} \diff x \right) \diff y \\
   & \geq \int_{[0,1]^d} |u(y)| \left( \int_{\R^d \setminus [-1, 2]^d} \frac{1}{|y-x|^{d+1-\alpha}} \diff x\right) \diff y \\
   & \geq C \|u\|_1.
  \end{align*}
  Therefore,
  \begin{equation}
    \label{eq:Du0u0}
    \|D u_0\|(\R^d)^{1-\alpha} \| u_0 \|_1^\alpha \leq C
    \|u\|_{W^{1-\alpha, 1}(\R^d)}.
  \end{equation}
  Now consider any $k \geq 1$. As the function $u_k$ has a
  decomposition
  \[
  u_k = \sum_{Q \in \calP_k} \alpha_Q \ind_Q
  \]
  we infer that
  \[
  \| D u_k \|(\R^d) \leq \sum_{Q \in \calP_k} |\alpha_Q| \| D \ind_Q
  \|(\R^d) = 2d (2^{-k})^{d-1} \sum_{Q \in \calP_k} |\alpha_Q| = C
  2^{k} \|u_k\|_1
  \]
  This leads to
  \begin{equation}
    \label{eq:sumkDukuk}
  \sum_{k=1}^\infty \|D u_k\|(\R^d)^{1-\alpha} \|u_k\|_1^\alpha \leq C
  \sum_{k=1}^\infty 2^{(1-\alpha)k} \|u_k\|_1
  \end{equation}
  Next we intend to estimate the right-hand side
  of~\eqref{eq:sumkDukuk}. For any $Q \in \calP_k$, we let $\hat{Q}
  \in \calP_{k-1}$ be the parent cube of $Q$. Almost everywhere on a
  $k$-cube $Q \in \calP_k$, the function $u_k$ is constant and takes the value
  \begin{align*}
  \alpha_Q & = \frac{1}{\calL^d(Q)} \int_Q u(x) \diff x -
  \frac{1}{\calL^d(\hat{Q})} \int_{\hat{Q}} u(y) \diff y \\ & =
  \frac{1}{\calL^d(Q) \calL^d(\hat{Q})} \int_Q \int_{\hat{Q}} \left(
  u(x) - u(y) \right) \diff x \diff y
  \end{align*}
  Consequently,
  \begin{align*}
  \|u_k\|_1 & \leq \sum_{Q \in \calP_k} \frac{1}{\calL^d(\hat{Q})}
  \int_Q \int_{\hat{Q}} |u(x) - u(y)| \diff x \diff y \\ & \leq
  2^{(k-1)d} \int_{\R^{2d}} |u(x) - u(y)| \sum_{Q \in \calP_k}
  \ind_Q(x) \ind_{\hat{Q}}(y) \diff x \diff y
  \end{align*}
  This implies that
  \begin{equation}
    \label{eq:sumkDukuk2}
  \sum_{k=1}^\infty 2^{(1 - \alpha)k} \|u_k\|_1 \leq \int_{\R^{2d}}
  \frac{|u(x) - u(y)|}{|x - y|^{d+1-\alpha}} a(x, y) \diff x \diff y
  \end{equation}
  where
  \[
  a(x, y) \leq C |x - y|^{d + 1 - \alpha} \sum_{k=1}^\infty 2^{k(d+1 -
    \alpha)} \sum_{Q \in \calP_k} \ind_Q(x) \ind_{\hat{Q}}(y)
  \]
  It is clear that, for all $k \geq 1$ and almost all $(x, y)$,
  \[
  \sum_{Q \in \calP_k} \ind_Q(x) \ind_{\hat{Q}}(y) \leq 1
  \]
  and that this sum is zero if $|x - y|$ is greater than the diameter
  $\sqrt{d} 2^{-(k-1)}$ of a $(k-1)$-cube. Thus, we can bound $a(x,
  y)$ by
  \[
  C|x - y|^{d+1 - \alpha} \sum_{k=1}^{k_0} 2^{k(d+1 - \alpha)}
  \]
  where $k_0$ is the greatest integer such that $\sqrt{d} 2^{-(k_0 -
    1)} \geq |x - y|$. This proves that the function $a$ is bounded by
  a constant depending only on $d$ and
  $\alpha$. Finally,~\eqref{eq:Du0u0},~\eqref{eq:sumkDukuk}
  and~\eqref{eq:sumkDukuk2} show that
  \[
  \sum_{k=1}^\infty \| Du_k \|(\R^d)^{1 - \alpha} \| u_k \|_1^\alpha
  \leq C \| u \|_{W^{1-\alpha}(\R^d)}
  \]
  This proves that $\llbracket u \rrbracket \in \bF^\alpha_d(\R^d)$
  and $\tilde{\bF}^\alpha(\llbracket u \rrbracket) \leq C
  \|u\|_{W^{1-\alpha,1}(\R^d)}$.
\end{proof}

In light of Subsection~\ref{subsec:firstexample} and the previous
theorem, we can assert that if $U$ is a bounded open set in $\R^d$
then,
\[
d - 1 + \inf \{ \alpha \in [0, 1] : \llbracket U \rrbracket \in
\bF^\alpha_d(\R^d) \} \leq \dim_{\mathrm{box}}(\partial U)
\]
or, equivalently,
\[
d - \dim_{\mathrm{box}}(\partial U) \leq \sup \{ s \in (0, 1) :
\ind_{U} \in W^{s, 1}_c(\R^d) \}
\]
(where, in the right-hand side, the supremum is zero if there are no
$s \in (0, 1)$ such that $\ind_U \in W^{s, 1}_c(\R^d)$).  The idea of
using the fractional Sobolev exponent $s$ as a way to measure the
fractional codimension of the boundary of $U$ has been first
introduced in~\cite{Visi}, where the above formula is proven to hold
if $U$ is any bounded Lebesgue measurable subset of $\R^d$, and the
topological boundary $\partial U$ is replaced by a measure-theoretic
analog.

\section{Charges in middle dimension}
\label{sec:charges}

\subsection{Definition}
\label{e:charges}
In this section, we present a self-contained introduction to charges
in middle dimension and establish their fundamental properties. This
concept was initially developed by De Pauw, Moonens, and Pfeffer
in~\cite{DePaMoonPfef}, where charges were defined over arbitrary
subsets of Euclidean space. The theory was later extended to the more
general context of metric spaces in~\cite{DePaHardPfef}. In this work,
however, we restrict our attention to charges defined on compact
subsets of $\mathbb{R}^d$, later on $\mathbb{R}^d$ itself in
Subsection~\ref{subsec:chargesRd}. This narrower focus allows us to
simplify many of the underlying arguments.

An \emph{$m$-charge} over a compact set $K \subset \R^d$ is a linear
map $\omega \colon \bN_m(K) \to \R$ that satisfies one of the
following equivalent continuity properties:
\begin{enumerate}
\item[(A)] $\omega(T_n) \to 0$ for any bounded sequence $(T_n)$ in
  $\bN_m(K)$ that converges in flat norm to $0$;
\item[(B)] the restriction of $\omega$ to the unit ball of
  $\bN_m(K)$ is $\bF$-continuous;
\item[(C)] for all $\varepsilon > 0$, there exists $\theta \geq 0$
  such that
  \[
  |\omega(T)| \leq \varepsilon \bN(T) + \theta \bF(T)
  \]
  holds for any normal current $T \in \bN_m(K)$.
\end{enumerate}
\begin{proof}[Proof of the equivalences]
  One clearly has (A) $\iff$ (B) and (C) $\implies$ (A). The only non
  trivial implication (A) $\implies$ (C) can be derived as a short
  consequence of the compactness theorem. Indeed, suppose by
  contradiction that (A) holds and (C) is false. In this case, there
  is $\varepsilon > 0$ and a sequence $(T_n)$ of normal currents
  supported in $K$, with normal masses $\bN(T_n) = 1$ such that
  \begin{equation}
    \label{eq:1}
    |\omega(T_n)| > n \bF(T_n) + \varepsilon
  \end{equation}
  for all $n$. Some subsequence $(T_{n_k})$ converges to a normal
  current $T \in \bN_m(K)$ in flat norm. Property (A) then implies
  that $\omega(T_{n_k}) \to \omega(T)$. Consequently, $\bF(T_{n_k})
  \leq n_k^{-1} |\omega(T_{n_k})|$ tends to $0$ as $k \to \infty$,
  which implies that $T = 0$ and $\omega(T_{n_k}) \to 0$. This is in
  contradiction with~\eqref{eq:1}.
\end{proof}
The space of $m$-charges over $K$ is denoted $\bCH^m(K)$. As $\bF \leq
\bN$, the continuity property (C) above implies that charges are
$\bN$-continuous, \textit{i.e.} $\bCH^m(K)$ is a subspace of the dual
$\bN_m(K)^*$. The space $\bCH^m(K)$ is equipped with the operator norm
\begin{equation}
  \label{eq:normCHmK}
  \| \omega \|_{\bCH^m(K)} \defeq \sup \left\{ \omega(T) : T \in
  \bN_m(K) \text{ and } \bN(T) \leq 1 \right\}.
\end{equation}
In fact, $\bCH^m(K)$ is a closed subspace of $\bN_m(K)^*$, for if
$(\omega_n)$ is a sequence in $\bCH^m(K)$ converging towards $\omega
\in \bN_m(K)^*$, then $\omega_n \to \omega$ uniformly on the unit ball
of $\bN_m(K)$. We conclude by (B) that $\omega$ is a charge. As a
result, $\bCH^m(K)$ is a Banach space.

In addition, there is also a notion of \emph{weak convergence of
  charges}: we say that $\omega_n \to \omega$ weakly whenever
$\omega_n(T) \to \omega(T)$ for all $T \in \bN_m(K)$. %Actually, this
%terminology is consistent with fact, stated here without further use,
%that $\bN_m(K)$ is the dual of $\bCH^m(K)$, for the duality bracket
%$\langle \omega, T \rangle = \omega(T)$, where $\omega \in \bCH^m(K)$
%and $T \in \bN_m(K)$.

\subsection{Continuous differential forms}
\label{subsec:cdf}
Any continuous $m$-form on $K$ defines an $m$-charge, via the map
\[
\Lambda \colon C(K; \wedge^m \R^d) \to \bCH^m (K)
\]
defined by
\begin{equation}
  \label{eq:defLambda}
  \Lambda(\omega)(T) = \int_{K} \langle \omega(x), \vec{T}(x)
  \rangle \diff \|T\|(x)
\end{equation}
for $\omega \in C(K; \wedge^m \R^d)$ and $T \in \bN_m(K)$. Note that,
in the special case where $\omega$ has a smooth extension to $\R^d$,
still denoted $\omega$, then $\Lambda(\omega)(T) = T(\omega)$.

Let us check that $\Lambda(\omega)$, defined by~\eqref{eq:defLambda}
for $\omega \in C(K; \wedge^m \R^d)$, is indeed an $m$-charge.

To this end, let us fix $\varepsilon > 0$ and choose a compactly
supported smooth form $\phi \colon \R^d \to \wedge^m \R^d$ such that
$|\omega(x) - \phi(x)| \leq \varepsilon$ for all $x \in K$. Then, for
all $T \in \bN_m(K)$,
\begin{equation}
  \label{eq:lambdaomega1}
  |\Lambda(\omega)(T) - \Lambda(\phi_{\mid K})(T)|  \leq \left| \int_{K} \langle \omega(x) -
  \phi(x), \vec{T}(x) \rangle \diff \|T\|(x)\right| 
  \leq \varepsilon \bM(T) \leq \varepsilon \bN(T) %+ \max \{ \|\phi\|_{\infty}, \| \diff \phi
  %\|_{\infty} \} \bF(T) \\ & \leq \varepsilon \bN(T) + \theta \bF(T)
\end{equation}
whereas
\begin{equation}
  \label{eq:lambdaomega2}
|\Lambda(\phi_{\mid K})(T)| \leq \theta \bF(T)
\end{equation}
for $\theta = \max \{ \|\phi\|_{\infty}, \| \diff
\phi\|_{\infty}\}$. The map $\Lambda(\phi_{\mid K}) \colon T \mapsto
T(\phi)$ is clearly linear. As $\varepsilon > 0$ is arbitrary,
inequality~\eqref{eq:lambdaomega1} shows the linearity of
$\Lambda(\omega)$. As for continuity, it follows
from~\eqref{eq:lambdaomega1}, \eqref{eq:lambdaomega2}
and~\ref{e:charges}(C).

We warn the reader that $\Lambda$ is, in general, not injective. Indeed, flat $m$-chains cannot be excessively concentrated. More precisely, it is known that a nonzero flat chain $T \in \bF_m(\R^d)$ must have a support of Hausdorff measure $\calH^m(\spt T) > 0$. We refer to \cite[4.1.20]{Fede} for a more general statement involving the integralgeometric measure. Consequently, if $K$ is nonempty and $\calH^m(K) = 0$, then $\bN_m(K) \subset \bF_m(K) = \{0\}$, which in turn implies $\bCH_m(K) = \{0\}$. In this situation, $\Lambda$ cannot be injective. For this reason, we will not identify a continuous $m$-form with its associated $m$-charge.

Only in the simplest case $m = 0$ do charges correspond to continuous functions, as we prove now.

\begin{prop}
  \label{prop:CH0}
  The map $\Lambda \colon C(K) \to \bCH^0(K)$ is a Banach space
  isomorphism.
\end{prop}

\begin{proof}
  We define the map $\Gamma \colon \bCH^0(K) \to C(K)$ by
  $\Gamma(\omega)(x) = \omega(\llbracket x \rrbracket)$ for all
  $\omega \in \bCH^0(K)$ and $x \in K$. The continuity of
  $\Gamma(\omega)$ follows from~\ref{e:charges}(A). It is clear that
  $\Gamma \circ \Lambda$ is the identity operator. Therefore, it
  suffices to show that $\Gamma$ is a Banach space isomorphism.
  
  First, we check that $\Gamma$ is a continuous. For all $x \in K$, we
  have
  \[
  |\Gamma(\omega)(x)| \leq \|\omega\|_{\bCH^0(K)} \bM(\llbracket x
  \rrbracket) = \|\omega\|_{\bCH^0(K)}.
  \]
  Thus, $\|\Gamma(\omega)\|_\infty \leq \|\omega\|_{\bCH^0(K)}$.

  Next, we claim that $\Gamma$ is injective. Let us call $\bP_0(K)$
  the space of polyhedral $0$-currents supported in $K$, \textit{i.e.}
  the linear space spanned by the $\llbracket x \rrbracket$, for $x
  \in K$. By an easy corollary of the deformation
  theorem~\cite[4.2.9]{Fede}, every $T \in \bN_0(K)$ is the
  $\bF$-limit of a sequence $(T_n)$ in $\bP_0(K)$ such that $\bM(T_n)
  \leq \bM(T)$. Observe that if $\omega \in \ker \Gamma$, then
  $\omega$ vanishes on $\bP_0(K)$. By the preceding result and the
  continuity property of charges, $\omega = 0$. This proves that
  $\Gamma$ is injective.

  Next we prove the surjectivity of $\Gamma$. Let $g \in C(K)$. We
  define the function $\omega$, on polyhedral $0$-currents, by
  \[
  \omega \left( \sum_{k=1}^n a_k \llbracket x_k \rrbracket
  \right) = \sum_{k=1}^n a_k g(x_k).
  \]
  Let $\varepsilon > 0$. There is a smooth compactly supported
  function $f \in C^\infty_c(\R^d)$ such that $\|f - g\|_{\infty, K}
  \leq \varepsilon$. Setting $\theta = \max\{\|f\|_{\infty}, \| \diff
  f \|_{\infty}\}$, we have
  \begin{align}
  |\omega(T)| & \leq \left| \sum_{k=1}^n a_k f(x_k) \right|+ \left|
  \sum_{k=1}^n a_k (f - g)(x_k) \right| \notag \\ & \leq \theta \bF(T)
  + \varepsilon \bM(T) \label{eq:2}
  \end{align}
  for every polyhedral $0$-current $T$. We extend $\omega$ to
  $\bM_0(K)$ with
  \begin{equation}
    \label{eq:extensionCharge}
  \omega(T) = \lim_{n \to \infty} \omega(T_n)
  \end{equation}
  where $T \in \bM_0(K)$ and $(T_n)$ is any sequence of polyhedral
  $0$-currents that $\bF$-converges to $T$, with $\bM(T_n) \leq
  \bM(T)$. By~\eqref{eq:2}, the sequence $(\omega(T_n))$ is Cauchy,
  which ensures that the limit in~\eqref{eq:extensionCharge} exists and
  does not depend on the choice of an approximating sequence. It is
  also straightforward that $\omega$ is linear and~\eqref{eq:2} holds
  now for any $T \in \bM_0(K)$. Hence $\omega \in \bCH^0(K)$ and
  $\Gamma(\omega) = g$. This proves that $\Gamma$ is onto.

  Finally, $\Gamma^{-1} = \Lambda$ is continuous by the open mapping
  theorem.
\end{proof}

\subsection{Operations on charges}
Operations on normal currents, such as pushforwards by Lipschitz maps,
taking the boundary, have a counterpart in term of charges, defined by
duality. We first define the \emph{exterior derivative} $\diff \omega
\in \bCH^{m+1}(K)$ of a charge $\omega \in \bCH^m(K)$, by setting
\[
\diff \omega(T) \defeq \omega(\partial T)
\]
for all $T \in \bN_m(K)$. That $\diff \omega$ is continuous, in the
sense of charges, is a consequence of the identities
\begin{equation}
  \label{eq:FNboundary}
  \bN(\partial T) \leq \bN(T), \qquad \bF(\partial T) \leq \bF(T)
\end{equation}
that furthermore implies that the operator $\mathrm{d} \colon
\bCH^{m}(K) \to \bCH^{m+1}(K)$ is bounded with norm less than or equal to $1$.

This allows us to introduce other examples of charges, namely the
differentials of continuous forms.
%{\color{red}In fact, this is the most general
%we can do, as a representation theorem, due to De Pauw, Moonens, and
%Pfeffer \cite[Theorem~6.1]{DePaMoonPfef} --which we shall not use--
%states that any charge can be decomposed as
%\[
%\omega = \Lambda(\omega_1) + \diff \Lambda(\omega_2)
%\]
%where $\omega_1$ are continuous $m$- and $(m-1)$-forms. (Is this
%true?)}

Next, if $f \colon K \to L$ is a Lipschitz map between two compact
subsets $K$ and $L$ of $\R^d$ and $\R^{d'}$, respectively, we define
the \emph{pullback}
\[
f^{\#} \colon \bCH^m(L) \to \bCH^m(K)
\]
by
\[
f^{\#} \omega(T) = \omega(f_\# T)
\]
for all $\omega \in \bCH^m(L)$ and $T \in \bN_m(K)$.

Most results from differential calculus extend to charges, by
duality arguments. We state some of them for the sake of completeness.
\begin{itemize}
\item[(A)] $\diff \circ \diff = 0$.
\item[(B)] The exterior derivative commutes with Lipschitz pullbacks.
\item[(C)] $(f \circ g)^{\#} = g^{\#} \circ f^{\#}$ for Lipschitz maps
  $f,g$ with compatible domains and codomains.
\end{itemize}

\subsection{Relative compactness in $\bCH^m(K)$}
We now state a criterion for relative compactness in $\bCH^m(K)$. It
will prove useful in the next section for establishing the basic
properties of the space of fractional charges. This result seems to
be new.

\begin{prop}
  \label{prop:compactCH}
  Let $\Omega \subset \bCH^m(K)$. The following are equivalent:
  \begin{itemize}
  \item[(A)] $\Omega$ is relatively compact;
  \item[(B)] for all $\varepsilon > 0$, there is $\theta \geq 0$ such that
    \[
    |\omega(T)| \leq \varepsilon \bN(T) + \theta \bF(T)
    \]
    holds for all $\omega \in \Omega$. We stress that $\theta$ does not depend on $\omega$. 
  \end{itemize}
\end{prop}

\begin{proof}
  (A) $\implies$ (B). We prove this implication by
  contradiction. Suppose there are $\varepsilon > 0$ and two sequences
  $(\omega_n)$ in $\Omega$ and $(T_n)$ in $\bN_m(K)$ such that
  \[
  |\omega_n(T_n)| > \varepsilon \bN(T_n) + n \bF(T_n) 
  \]
  for all integers $n$. We can also suppose $\bN(T_n) = 1$ for all
  $n$. As $\Omega$ is relatively compact, it is bounded, consequently
  \[
  n \bF(T_n) < \sup_{\omega \in \Omega} \|\omega\|_{\bCH^m(K)} <
  \infty
  \]
  which implies that $(T_n)$ converges to $0$ in flat norm. On the
  other side, there is a subsequence $(\omega_{n_k})$ that converges
  to $\omega \in \bCH^m(K)$. Hence
  \[
  |\omega_{n_k}(T_{n_k})| \leq |\omega(T_{n_k})| + \|\omega -
  \omega_{n_k}\|_{\bCH^m(K)} \to 0
  \]
  which contradicts that $|\omega_{n_k}(T_{n_k})| > \varepsilon$.

  (B) $\implies$ (A). Denote by $B_{\bN_m(K)}$ the unit ball of
  $\bN_m(K)$, metrized by $\bF$, and let $\iota \colon \bCH^m(K) \to
  C(B_{\bN_m(K)})$ be the linear map that sends a charge to its
  restriction to $B_{\bN_m(K)}$. Here, $C(B_{\bN_m(K)})$ is given the
  supremum norm. Since $\bCH^m(K)$ is complete and $\iota$ is an
  isometric embedding, we only need to show that $\iota(\Omega)$ is
  relatively compact in $C(B_{\bN_m(K)})$.

  First, the inequality in (B) (for $\varepsilon = 1$) entails that
  $\iota(\Omega)$ is pointwise bounded. Now, for an arbitrary
  $\varepsilon > 0$ there is $\theta \geq 0$ as in (B). If $T, S \in
  B_{\bN_m(K)}$ satisfy $\bF(T - S) \leq \varepsilon / \theta$, then
  for any $\omega \in \Omega$, one has
  \[
  |\iota(\omega)(T) - \iota(\omega)(S)| \leq \theta \bF(T - S) +
  \varepsilon \bN(T - S) \leq 3 \varepsilon.
  \]
  This proves that $\iota(\Omega)$ is equicontinuous, thus relatively
  compact by the Arzelà-Ascoli theorem. The proof is then finished.
\end{proof}

\subsection{Charges over $\R^d$}
\label{subsec:chargesRd}
An \emph{$m$-charge} over $\R^d$ is a linear map $\omega \colon
\bN_m(\R^d) \to \R$ that satisfies the continuity condition
\[
\omega(T_n) \to 0 \quad \text{whenever } \bF(T_n) \to 0, \, \, \sup_{n \geq
  0} \bN(T_n) < \infty \text{ and } \bigcup_{n=0}^\infty \spt T_n
\text{ is bounded}
\]
Alternatively, $\omega$ is an $m$-charge over $\R^d$ whenever its
restriction to $\bN_m(K)$ for every compact set $K$ is an $m$-charge over
$K$. We denote by $\bCH^m(\R^d)$ the space of $m$-charges over
$\R^d$.

For each compact set, we define
\[
\| \omega \|_{\bCH^m, K} = \sup \left\{ \omega(T) : T \in \bN_m(K) \right\}
\]
This definition closely resembles~\eqref{eq:normCHmK}, with the
difference that $\| \cdot \|_{\bCH^{m}, K}$ is now only a seminorm on
$\bCH^m(\R^d)$. The collection of these seminorms, as $K$ ranges over
all compact subsets of $\R^d$, defines a vector space topology on
$\bCH^m(\R^d)$. Clearly, only a countable subset of these seminorms is
required to determine this topology (for example, by letting $K$ range
over the closed ball centered at the origin with positive integer
radius). Furthermore, it follows directly from the fact that the
spaces $\bCH^m(K)$ are Banach spaces and that $\bCH^m(\R^d)$ is a
Fréchet space.

Many definitions and properties carry over easily to $\R^d$.
\begin{itemize}
\item[(A)] \emph{Weak convergence}: a sequence $(\omega_n)$ converges
  to $\omega$ weakly in $\bCH^m(\R^d)$ whenever $\omega_n(T) \to
  \omega(T)$ for any $T \in \bN_m(\R^d)$;
\item[(B)] \emph{Exterior derivative}: the operator $\diff \colon
  \bCH^m(\R^d) \to \bCH^{m+1}(\R^d)$ defined by $\diff \omega(T) =
  \omega(\partial T)$ is linear and continuous;
\item[(C)] \emph{Pullback by locally Lipschitz maps}: for any locally
  Lipschitz map $f \colon \R^d \to \R^{d'}$ and $\omega \in
  \bCH^{m}(\R^{d'})$, we define the pullback $f^\# \omega \in
  \bCH^{m}(\R^d)$ by $f^\# \omega(T) = \omega( f_\# T)$. This
  definition has a natural adaptation in the domain or the codomain of
  $f$ is a compact subset of the Euclidean space.

  In particular, we will frequently use pullbacks as a tool to extend
  charges, in the following context: if $K \subset \R^d$ is compact
  convex and $p \colon \R^d \to K$ is the orthogonal projection, then
  $p^\# \omega \in \bCH^{m}(\R^d)$ agrees with $\omega$ on $\bN_m(K)$. This is because $p_\# T = \operatorname{id}_\# T = T$ for all $T \in \bN_m(K)$. Recall that the pushforward of a flat chain $T$ by a map $f$ depends only on the restriction of $f$ to the support $\spt T$.
\item[(D)] \emph{Continuous differential $m$-forms over $\R^d$}: as before, a
  continuous $m$-form can be naturally regarded as an $m$-charge,
  through the map $\Lambda \colon C(\R^d; \wedge^m \R^d) \to
  \bCH^m(\R^d)$, defined by
  \[
  \Lambda(\omega)(T) = \int_{\R^d} \langle \omega(x), \vec{T}(x)
  \rangle \diff \| T \|(x)
  \]
  for all $T \in \bN_m(\R^d)$.
  
  This map is injective. Indeed, suppose that $\Lambda(\omega) = 0$. For every compactly supported smooth $m$-current $\calL^d \wedge \xi$, one has
  \[
  \Lambda(\omega)(\calL^d \wedge \xi) = \int_{\R^d} \langle \omega(x), \xi(x) \rangle \diff x = 0.
  \]
  Since this holds for every $\xi \in C^\infty_c(\R^d; \wedge_m \R^d)$, one infers that $\omega = 0$ by the fundamental lemma of the calculus of variations.
  
  Consequently, we will identify continuous $m$-forms with $m$-charges over $\R^d$,
  allowing us to consider that
  \[
  C(\R^d; \wedge^m \R^d) \subset \bCH^{m}(\R^d).
  \]
  In particular, in the case $\omega$ is a smooth $m$-form over $\R^d$
  and $T$ is a normal $m$-current, both expressions $\omega(T)$ and
  $T(\omega)$ are well-defined and interchangeable.

  When $\omega \in C^1(\R^d; \wedge^m \R^d)$, the meaning of $\diff
  \omega$ is unambiguous, whether we think of $\omega$ as an $m$-form
  or an $m$-charge. This follows from the classical Stokes' formula.
\item[(E)] \emph{Continuity of $0$-charges}: the map $\Lambda \colon
  C(\R^d) \to \bCH^0(\R^d)$ is a Fréchet space isomorphism, whose
  inverse is the map $\Gamma$ defined by $\Gamma(\omega)(x) =
  \omega(\llbracket x \rrbracket)$ for $\omega \in \bCH^0(\R^d)$ and
  $x \in \R^d$. Here, $C(\R^d)$ is given the Fréchet topology induced
  by the family of seminorms $\| \cdot \|_{\infty, K}$, where $K$ can
  be any compact subset of $\R^d$.
\item[(F)] \emph{Relative compactness in $\bCH^m(\R^d)$}: a subset
  $\Omega \subset \bCH^m(\R^d)$ is relatively compact in
  $\bCH^m(\R^d)$ whenever for all compact subsets $K \subset \R^d$ and
  $\varepsilon > 0$, there exists $\theta = \theta(K, \varepsilon)
  \geq 0$ such that $|\omega(T)| \leq \varepsilon \bN(T) + \theta
  \bF(T)$ for all $T \in \bN_m(K)$.
\end{itemize}
%This time, the formula~\eqref{eq:seminormCHm} still makes
%sense for $\omega \in \bCH^m(\R^d)$, however $\| \cdot \|_{\bCH^m(K)}$
%is only a seminorm on $\bCH^m(\R^d)$. The collection of all such seminorms
%induces a Fréchet topology on $\bCH^m(\R^d)$.

\subsection{Regularization of charges}

Operating on the entire space $\R^d$ enables the regularization of
charges through convolution. The \emph{convolution} of a charge $
\omega \in \bCH^m(\R^d)$ with a smooth compactly supported function
$\phi$ is the linear map $\bN_m(\R^d) \to \R$ that sends $T$ to
$\omega(T * \check{\phi})$. It satisfies the continuity property of
charges because the map $T \mapsto T * \check{\phi}$ is both $\bN$-
and $\bF$-continuous.

This provides an example of weak convergence of charges. Indeed, for
$T \in \bN_m(\R^d)$, one has $T * \Phi_\varepsilon \to 0$ in flat norm
as $\varepsilon \to 0$, whereas the normal masses of the $T *
\Phi_\varepsilon$ are bounded by $\bN(T)$, by
Proposition~\ref{prop:22}. (Recall that we fixed a regularization
kernel $\Phi_\varepsilon$ in Section~\ref{sec:prelim}). This entails
that $\omega * \Phi_\varepsilon \to \omega$ weakly.

\begin{prop}
  \label{prop:smooth}
  Let $\omega \in \bCH^m(\R^d)$ and $\phi \in C^\infty_c(\R^d)$. Then
  $\omega * \phi \in C^\infty(\R^d ; \wedge^m \R^d)$.
\end{prop}

\begin{proof}
  We will actually prove the explicit formula
  \begin{equation}
    \label{eq:formulaConvOmega}
    (\omega * \phi)(z) = \sum_{I \in \Lambda(d, m)} \omega \left(
    \calL^d \wedge \phi( z - \cdot ) \boldsymbol{e}_I \right) \diff
    x_I \text{ for all } z \in \R^d,
  \end{equation}
  where $\Lambda(d, m)$ is the set of increasing subfamilies of $\{1,
  \dots, d\}$ of cardinal $m$, and for $I = (i_1, \dots, i_m)$,
  \[
  \boldsymbol{e}_I = e_{i_1} \wedge \cdots \wedge e_{i_m} \text{ and }
  \diff x_I = \diff x_{i_1} \wedge \cdots \wedge \diff x_{i_m}
  \]
  and $e_1, \dots, e_d$ is the canonical basis of $\R^d$.

  Call $\tilde{\omega}(z)$ the right-hand side
  in~\eqref{eq:formulaConvOmega}. First we check that $\tilde{\omega}$
  is a smooth $m$-form. This is done by ensuring that, for all $1 \leq
  i \leq d$ and for any sequence $(h_n)$ of nonzero real numbers
  tending to 0,
  \[
  \frac{ \calL^d \wedge \phi( z + h_n \boldsymbol{e}_i - \cdot )
    \boldsymbol{e}_I - \calL^d \wedge \phi( z - \cdot )
    \boldsymbol{e}_I }{h_n} \to \calL^d \wedge \frac{\partial
    \phi}{\partial x_i}(z - \cdot) \boldsymbol{e}_I
  \]
  in flat norm with uniformly bounded normal masses. An argument by
  induction finishes the proof that the component functions $z \mapsto
  \omega\left(\calL^d \wedge \phi(z - \cdot) \boldsymbol{e}_I \right)$
  are smooth.

  Next, in order to prove that the charges $\omega * \phi$ and
  $\tilde{\omega}$ coincide, we need only do so on currents of the
  form $\calL^d \wedge \xi$, where $\xi = \sum_{I \in \Lambda(d, m)}
  \xi_I \boldsymbol{e}_I$ is a compactly supported smooth
  $m$-vector field. This is because, for all $T \in \bN_m(\R^d)$,
  \[
  (\omega * \phi - \tilde{\omega})(T) = \lim_{\varepsilon \to 0}
  (\omega * \phi - \tilde{\omega})(T * \Phi_\varepsilon)
  \]
  using the continuity property of the charge $\omega * \phi -
  \tilde{\omega}$ and Proposition~\ref{prop:22}(E) and (G). As $T *
  \Phi_\varepsilon$ has the form $\calL^d \wedge \xi$ by
  \cite[4.1.2]{Fede}, the claim follows.

  We begin by evaluating
  \[
  (\calL^d \wedge \xi)(\tilde{\omega}) = \sum_{I \in \Lambda(d,m)}
  \int_{\R^d} \omega(\calL^d \wedge \phi( z - \cdot )
  \boldsymbol{e}_I) \xi_I(z) \diff z.
  \]
  On the other hand, one has
  \[
  \omega(T * \check{\phi}) = \sum_{I \in \Lambda(d,m)} \omega \left(
  \calL^d \wedge \xi_I * \check{\phi} \boldsymbol{e}_I \right).
  \]
  Following~\cite[4.1.2]{Fede}, we introduce, for every $n \geq 1$, a
  partition $A_{n,1}, \dots, A_{n,p_n}$ of $\spt \xi$ into Borel sets
  of diameter less than $n^{-1}$ and choose points $z_{n, k} \in A_{n,
    k}$ for $1 \leq k \leq p_n$. Then
  \[
  \sum_{k=1}^{p_n} \xi_I(z_{n,k}) \left(\calL^d \wedge \phi(z_{n,k} - \cdot)
  \boldsymbol{e}_I\right) \calL^d(A_{n,k}) \to \calL^d \wedge \xi_I *
  \check{\phi} \boldsymbol{e}_I
  \]
  in flat norm with uniformly bounded normal masses. Thus,
  \begin{align*}
    \omega(T * \check{\phi}) & = \lim_{n \to \infty} \sum_{I \in
      \Lambda(d,m)} \sum_{k=1}^{p_n} \xi_I(z_{n,k}) \omega(\calL^d \wedge
    \phi(z_{n,k} - \cdot) \boldsymbol{e_I}) \calL^d(A_{n,k}) \\ & =
    \sum_{I \in \Lambda(d,m)} \int_{\R^d} \xi_I(z) \omega(\calL^d
    \wedge \phi(z - \cdot) \boldsymbol{e_I}) \diff z \\ & = (\calL^d
    \wedge \xi)(\tilde{\omega}). \qedhere
  \end{align*}
\end{proof}

%\subsection{Operations on charges}
%Operations on normal currents, such as pushforwards under Lipschitz
%maps and taking boundaries, have corresponding counterparts for
%charges, defined through duality. Here we define:
%\begin{itemize}
%\item the \emph{exterior derivative} of a charge $\omega \in
%  \bCH^m(K)$ is the charge $\diff \omega \in \bCH^{m+1}(K)$ defined by
%  $\diff \omega(T) = \omega(\partial T)$.  This construction works
%  over $\R^d$ as well. %Of course $\diff \circ \diff = 0$.
%\item the \emph{pullback} of a charge $\omega \in \bCH^m(L)$ by a map
%  $f \colon K \to L$ (where $K \subset \R^d$ is either a compact
%  subset or $K = \R^d$, and similarly $L \subset \R^{d'}$ is a compact
%  subset or $L = \R^{d'}$) is the charge $f^{\#} \omega \in \bCH^m(K)$
%  defined by $f^{\#}\omega(T) = \omega(f_\#T)$.
%\item the \emph{convolution by $\Phi_\varepsilon$} of an $m$-charge
%  $\omega \in \bCH^m(\R^d)$ is the charge $\omega * \Phi_\varepsilon
%  \in \bCH^{m}(\R^d)$ defined by $\omega * \Phi_\varepsilon(T) =
%  \omega(T * \Phi_\varepsilon)$. It is a smooth $m$-form, see
%  \cite[Proposition~3.9]{Boua1}.
%\end{itemize}

\section{Fractional charges and duality with fractional currents}
\label{sec:fracCharges}

\subsection{Definition}
Let $\alpha \in \mathopen{(} 0, 1 \mathclose{]}$. An
  \emph{$\alpha$-fractional charge} over a compact set $K \subset
  \R^d$ is a linear functional $\omega \colon \bN_m(K) \to \R$ for
  which there is a constant $C \geq 0$ such that
  \[
  |\omega(T)| \leq C \bN(T)^{1 - \alpha} \bF(T)^\alpha \text{ for all
  } T \in \bN_m(K)
  \]
It is clear that the above requirement is stronger than the continuity
condition of Subsection~\ref{e:charges}. We adopt the notation
$\bCH^{m, \alpha}(K)$ to represent the space of $\alpha$-fractional
$m$-charges, normed by
\begin{equation*}
  \|\omega\|_{\bCH^{m, \alpha}(K)} = \inf \left\{ C \geq 0 :
  |\omega(T)| \leq C \bN(T)^{1-\alpha} \bF(T)^{\alpha} \text{ for all
  } T \in \bN_m(K) \right\}.
\end{equation*}
We also define $\| \omega \|_{\bCH^{m, \alpha}(K)} = \infty$ if
$\omega \in \bCH^m(K) \setminus \bCH^{m, \alpha}(K)$.
  
The parameter $\alpha$ represents regularity. One clearly has
inclusions
\[
\bCH^{m,\beta}(K) \subset \bCH^{m, \alpha}(K) \subset \bCH^m(K)
\]
(that are continuous) whenever $\beta \geq \alpha$. In addition, the
reader may use the continuity of the second embedding and the lower
semicontinuity of $\|\cdot\|_{\bCH^{m, \alpha}(K)}$ with respect to
weak convergence to check that $\bCH^{m, \alpha}(K)$ is a Banach
space.

When $\alpha = 1$ and $K$ is convex, we encounter a well-known
object. Indeed, in this case, a $1$-fractional charge $\omega$ is
$\bF$-continuous and $\bN_m(K)$ is $\bF$-dense in $\bF_m(K)$. As such,
$\omega$ can be uniquely extended so as to become an element of
$\bF_m(K)^*$, the space of \emph{flat $m$-cochains over $K$},
introduced by H. Whitney. This result will be generalized in
Theorem~\ref{thm:duality}.

More generally, we can think of $\alpha$-fractionality as a regularity
that is intermediate between that of mere charges and that of flat
cochains.

We observe that, as a consequence of~\eqref{eq:FNboundary}, the
exterior derivative of an $\alpha$-fractional $m$-charge is again
$\alpha$-fractional (and the map $\diff \colon \bCH^{m, \alpha}(K) \to
\bCH^{m+1, \alpha}(K)$ is continuous). As for pullbacks, if $m \geq 1$, $f \colon
K \to L$ denotes a Lipschitz map and $\omega \in \bCH^{m, \alpha}(L)$,
then for all $T \in \bN_m(K)$,
\begin{align*}
  |f^\# \omega(T)| & = |\omega ( f_\# T )| \\ & \leq \| \omega
  \|_{\bCH^{m, \alpha}(L)} \bN(f_\# T)^{1 - \alpha} \bF(f_\# T)^\alpha
  \\ & \leq \max \{(\rmLip f)^{m-1 + \alpha}, (\rmLip f)^{m + \alpha}
  \} \| \omega \|_{\bCH^{m,\alpha}(L)} \bN(T)^{1 - \alpha} \bF(T)^\alpha.
\end{align*}
Accordingly, $f^\# \omega$ is $\alpha$-fractional and
\[
\| f^\# \omega \|_{\bCH^{m,\alpha}(K)} \leq \max \{(\rmLip f)^{m-1 +
  \alpha}, (\rmLip f)^{m + \alpha} \} \| \omega
\|_{\bCH^{m,\alpha}(L)}.
\]
In case $m = 0$, one proves similarly that $f^\# \omega$ is $\alpha$-fractional and
\[
\| f^\# \omega \|_{\bCH^{0, \alpha}(K)} \leq (\rmLip f)^\alpha \| \omega\|_{\bCH^{0, \alpha}(L)}
\]

\subsection{Hölder differential forms}
In this subsection, we claim that $\alpha$-Hölder continuous $m$-forms
are $\alpha$-fractional charges. More precisely, the map $\Lambda$
defined in Subsection~\ref{subsec:cdf} restricts to a continuous
linear map
\begin{equation}
  \label{eq:Lambdaalpha}
  \Lambda \colon \rmLip^\alpha(K; \wedge^m \R^d) \to \bCH^{m,
    \alpha}(K)
\end{equation}
Indeed, let $\omega \in \rmLip^\alpha(K; \wedge^m \R^d)$. It is
possible to extend $\omega$ to an $\alpha$-Hölder continuous $m$-form $\tilde{\omega}$
on $\R^d$ such that
\[
\rmLip^\alpha(\tilde{\omega}) \leq C \rmLip^\alpha(\omega) \text{ and
} \|\tilde{\omega}\|_\infty = \| \omega\|_{\infty}.
\]
Such an extension can be obtained, for instance, via the McShane extension theorem (see \cite[Theorem~1.33]{Weav}).
For any $\varepsilon \in \mathopen{(}0, 1\mathclose{]}$, the smooth
  $m$-form $\omega_\varepsilon = \tilde{\omega} * \Phi_\varepsilon$
  satisfies
  \[
  \| \omega_\varepsilon\|_\infty \leq \| \omega \|_\infty, \qquad
\| \tilde{\omega} - \omega_\varepsilon \|_\infty \leq C\rmLip^\alpha(\omega)
\varepsilon^\alpha \text{ and } \| \diff \omega_\varepsilon \|_\infty
\leq \frac{C}{\varepsilon^{1 - \alpha}} \rmLip^\alpha(\omega)
\]
For an arbitrary $T \in \bN_m(K)$, we have
\[
\Lambda(\omega)(T) = \int_K \langle \omega(x) - \omega_\varepsilon(x), \vec{T}(x) \rangle \diff \|T\|(x) + T(\omega_\varepsilon)
\]
From the above inequalities, we can control the first term
\[
\left| \int_K \langle \omega(x) - \omega_\varepsilon(x), \vec{T}(x)
\rangle \diff \|T\|(x) \right| \leq C\rmLip^\alpha(\omega)
\varepsilon^\alpha \bN(T)
\]
whereas,
\begin{align*}
  |T(\omega_\varepsilon)| & \leq \max \left\{
  \|\omega_\varepsilon\|_\infty, \| \diff \omega_\varepsilon \|_\infty
  \right\} \bF(T) \\ & \leq \frac{C}{\varepsilon^{1 - \alpha}} \max
  \left\{ \| \omega \|_{\infty}, \rmLip^\alpha(\omega) \right\} \bF(T)
\end{align*}
Combining the two preceding inequalities yields
\begin{equation}
  \label{eq:combining}
|\Lambda(\omega)(T)| \leq C \max \left\{ \| \omega \|_{\infty},
\rmLip^\alpha(\omega) \right\} \left( \varepsilon^\alpha \bN(T) +
\frac{\bF(T)}{\varepsilon^{1-\alpha}}\right)
\end{equation}
In case $\alpha \neq 1$, we choose $\varepsilon = \bF(T) / \bN(T)$
(which is indeed less than or equal to $1$), so that
\begin{equation}
  \label{eq:LambdaC0bis}
|\Lambda(\omega)(T)| \leq C \max \left\{ \| \omega \|_\infty,
\rmLip^\alpha(\omega) \right\} \bN(T)^{1 - \alpha} \bF(T)^\alpha
\end{equation}
thereby showing that $\Lambda(\omega)$ is $\alpha$-fractional and
$\Lambda$ is continuous, with
\begin{equation}
  \label{eq:LambdaC0}
  \|\Lambda(\omega)\|_{\bCH^{m,\alpha}(K)} \leq C \max \{ \| \omega
  \|_\infty, \rmLip^\alpha(\omega) \}.
\end{equation}
Inequalities~\eqref{eq:LambdaC0bis} and~\eqref{eq:LambdaC0} are
obtained in case $\alpha = 1$ by letting $\varepsilon$ tend to $0$
in~\eqref{eq:combining}.

\begin{prop}
  \label{prop:CH0alpha}
  Let $\alpha \in \mathopen{(} 0, 1 \mathclose{]}$. The map $\Lambda
    \colon \rmLip^\alpha(K) \to \bCH^{0, \alpha}(K)$ is a Banach space
    isomorphism.
\end{prop}

\begin{proof}
  Recall the map $\Gamma \colon \bCH^0(K) \to C(K)$ from the proof of
  Proposition~\ref{prop:CH0}
  \[
  \Gamma(\omega)(x) = \omega(\llbracket x \rrbracket) \text{ for all }
  \omega \in \bCH^{0}(K) \text{ and } x \in K.
  \]
  It was proved that it is inverse to the map $\Lambda \colon C(K) \to
  \bCH^0(K)$. It suffices to prove that $\Gamma$ restricts to a Banach
  space isomorphism $\bCH^{0, \alpha}(K) \to \rmLip^\alpha(K)$.

  Let $\omega \in \bCH^{0, \alpha}(K)$. The function $\Gamma(\omega)$
  is $\alpha$-Hölder continuous, as for all $x, y \in K$, one has
  \begin{align*}
  \left| \Gamma(\omega)(x) - \Gamma(\omega)(y) \right| & =
  \left|\omega(\llbracket x \rrbracket - \llbracket y
  \rrbracket)\right| \\ & \leq \| \omega \|_{\bCH^{0, \alpha}(K)}
  \bN(\llbracket x \rrbracket - \llbracket y \rrbracket)^{1 -
    \alpha}\bF(\llbracket x \rrbracket - \llbracket y
  \rrbracket)^\alpha \\ & \leq C \| \omega \|_{\bCH^{0, \alpha}(K)} |x
  - y|^\alpha.
  \end{align*}
  Furthermore, for any $x \in K$,
  \[
  |\Gamma(\omega)(x)| = |\omega(\llbracket x \rrbracket)| \leq \|
  \omega \|_{\bCH^{0, \alpha}(K)}.
  \]
  This shows that $\max \{ \| \Gamma(\omega) \|_\infty,
  \rmLip^\alpha(\Gamma(\omega)) \} \leq C \| \omega \|_{\bCH^{0,
      \alpha}(K)}$, giving the continuity of $\Gamma$.

  Let $f \in \rmLip^\alpha(K)$. In particular, $f \in C(K)$ so there
  is a charge $\omega \in \bCH^0(K)$ such that $\omega =
  \Lambda(f)$. Let us prove that $\omega \in \bCH^{0,
    \alpha}(K)$. First extend $f$ to $\R^d$ while preserving the
  Hölder constant. As usual, we consider $\varepsilon \in \mathopen{(}
  0, 1\mathclose{]}$ and define $f_\varepsilon = f *
    \Phi_\varepsilon$. We have
  \[
  \| f_\varepsilon \|_\infty \leq \|f \|_{\infty, K}, \| f -
  f_\varepsilon \|_\infty \leq \varepsilon^\alpha \rmLip^\alpha(f)
  \text{ and } \rmLip f_\varepsilon = \| \diff f_\varepsilon \|_\infty
  \leq C \varepsilon^{\alpha - 1} \rmLip^\alpha f.
  \]
  Let $T = \sum_{k=0}^n a_k \llbracket x_k \rrbracket$ be a
  $0$-polyhedral current, supported in $K$.
  \[
  \omega(T) = \sum_{k=0}^n a_k f_\varepsilon(x_k) + \sum_{k=0}^n a_k
  \left( f(x_k) - f_\varepsilon(x) \right)
  \]
  Therefore,
  \begin{align*}
  |\omega(T)| & \leq |T(f_\varepsilon)| + \varepsilon^\alpha
  \rmLip^\alpha(f) \bN(T) \\ & \leq C \max \{ \|f \|_{\infty, K},
  \rmLip^\alpha f\} \left( \varepsilon^{\alpha - 1} \bF(T) +
  \varepsilon^\alpha \bN(T) \right).
  \end{align*}
  Choosing $\varepsilon = \bF(T) / \bN(T)$ yields
  \[
  |\omega(T)| \leq C \max \{ \|f \|_{\infty, K}, \rmLip^\alpha(f) \}
  \bN(T)^{1 - \alpha} \bF(T)^\alpha.
  \]
  The preceding inequality holds as well for an arbitrary $T \in
  \bN_0(K)$, as it is the $\bF$-limit of as sequence of $0$-polyhedral
  currents $(T_n)$ with $\bM(T_n) \leq \bM(T)$. This proves that
  $\omega$ is $\alpha$-fractional. Since $\Gamma(\omega) = f$, this
  concludes that the map $\Gamma$, restricted to $\bCH^{0, \alpha}(K)
  \to \rmLip^\alpha(K)$, is surjective. By the open mapping theorem,
  it is a Banach space isomorphism, and so is $\Lambda = \Gamma^{-1}
  \colon \rmLip^\alpha(K) \to \bCH^{0, \alpha}(K)$.
\end{proof}

\subsection{Fractional charges over $\R^d$}
An \emph{$\alpha$-fractional $m$-charge} over $\R^d$ is a linear functional
$\omega \colon \bN(\R^d) \to \R$ such that, for all compact $K \subset
\R^d$, there is a constant $C_K \geq 0$ such that
\[
|\omega(T)| \leq C_K \bN(T)^{1-\alpha} \bF(T)^\alpha \text{ for all }
T \in \bN_m(K).
\]
In other words, an $\alpha$-fractional charge is an $m$-charge whose
restriction to each $\bN_m(K)$ belongs to $\bCH^{m,\alpha}(K)$. We
denote by $\bCH^{m,\alpha}(\R^d)$ the space of $\alpha$-fractional
charges over $\R^d$.

We define, for each compact $K \subset \R^d$, the following seminorm
on $\bCH^{m,\alpha}(\R^d)$,
\[
\|\omega\|_{\bCH^{m, \alpha}, K} = \inf \left\{ C \geq 0 : |\omega(T)|
\leq C \bN(T)^{1-\alpha} \bF(T)^{\alpha} \text{ for all } T \in
\bN_m(K) \right\}.
\]
The family of such seminorms induces a Fréchet topology on
$\bCH^{m,\alpha}(\R^d)$. We summarize the following facts, each of
which is a straightforward adaptation of previous arguments to $\R^d$.
\begin{itemize}
\item[(A)] \emph{Exterior derivative}: the operator $\mathrm{d} \colon
  \bCH^{m}(\R^d) \to \bCH^{m+1}(\R^d)$ restricts to a linear
  continuous operator $\bCH^{m, \alpha}(\R^d) \to \bCH^{m+1,
    \alpha}(\R^d)$.
\item[(B)] \emph{Pullback by locally Lipschitz maps}: let $f \colon
  \operatorname{dom} f \to \operatorname{codom} f$ be a locally
  Lipschitz map, where the domain is either $\R^d$ a compact subset
  thereof, and likewise, the codomain is either $\R^{d'}$ a compact
  subset thereof. Then the pullback $f^\#$ is a linear continuous map
  $\bCH^{m, \alpha}(\operatorname{codom} f) \to
  \bCH^{m,\alpha}(\operatorname{dom} f)$.
\item[(C)] \emph{Hölder continuous differential forms}: the map $\Lambda
  \colon C(\R^d; \wedge^m \R^d) \to \bCH^m(\R^d)$ restricts to a
  continuous linear map $\rmLip^\alpha_{\rmloc}(\R^d; \wedge^m \R^d)
  \to \bCH^{m, \alpha}(\R^d)$, that is still injective. As such, we
  consider that
  \[
  \rmLip^\alpha_{\rmloc}(\R^d; \wedge^m \R^d) \subset \bCH^{m,
    \alpha}(\R^d).
  \]
\item[(D)] \emph{Hölder continuity of fractional $0$-charges}: the map
  $\Lambda \colon \rmLip^\alpha_{\rmloc}(\R^d) \to \bCH^{0,
  \alpha}(\R^d)$ is a Fréchet space isomorphism, whose inverse is
  given by the map $\Gamma$ defined by $\Gamma(\omega)(x) =
  \omega(\llbracket x \rrbracket)$.
\end{itemize}

We end this subsection with two technical results that give the
$1$-fractional seminorm of a smooth form. They will be used in the
proof of Theorem~\ref{thm:main} and their reading may be postponed
until that point.

\begin{prop}
  \label{prop:CH1smooth}
  Suppose $\omega \in C^\infty(\R^d; \bigwedge^m \R^d)$ and $K \subset
  \R^d$ is a compact convex set with positive Lebesgue measure.  Then
  $\| \omega \|_{\bCH^{m,1}, K} = \max \{ \| \omega \|_{\infty, K}, \|
  \diff \omega \|_{\infty, K} \}$.
\end{prop}

\begin{proof}
  The convexity of $K$ guarantees that for all $T \in \bN_m(K)$, one has
  \begin{align*}
  \bF(T) & = \inf \{ \bM(A) + \bM(B) : T = A + \partial B \text{ with } A \in \bN_m(\R^d) \text { and } B \in \bN_{m+1}(\R^d) \} \\
  & =  \inf \{ \bM(A) + \bM(B) : T = A + \partial B \text{ with } A \in \bN_m(K) \text { and } B \in \bN_{m+1}(K) \}
  \end{align*}
  Indeed, one can always replace $A$ and $B$ by their pushforwards $p_\#$ and $p_\# B$ by the orthogonal projection $p$ onto $K$. Since $p$ is $1$-Lipschitz, this operation does not increase the mass and the identity $T = p_\# A + \partial (p_\# B)$ is valid.
  
  Therefore, we consider $A \in \bN_m(K)$ and $B \in
  \bN_{m+1}(K)$ such that $T = A + \partial B$ and compute
  \[
  |\omega(T)| = |A(\omega)| + |B(\diff \omega)| \leq \left( \bM(A) +
  \bM(B) \right) \max \{ \| \omega \|_{\infty, K}, \| \diff \omega
  \|_{\infty, K} \}.
  \]
  Taking the infimum over $A, B$, one obtains that $|\omega(T)| \leq
  \bF(T) \max \{ \| \omega \|_{\infty, K}, \| \diff \omega \|_{\infty,
    K} \}$. This means that $\| \omega \|_{\bCH^{m,1}, K} \leq \max \{
  \| \omega \|_{\infty, K}, \| \diff \omega \|_{\infty, K} \}$.

  As the Lebesgue measure of the convex set $K$ is nonzero, one can
  replace the $\| \cdot \|_{\infty, K}$ seminorms with essential
  suprema. One has then
  \begin{gather*}
  \| \omega \|_{\infty, K} = \sup_\zeta \int \langle \omega(x),
  \zeta(x) \rangle \diff x = \sup_\zeta \omega(\calL^d \wedge \zeta)
  \\ \| \diff \omega \|_{\infty, K} = \sup_\xi \int \langle (\diff
  \omega)(x), \xi(x) \rangle \diff x = \sup_{\xi}
  \omega\left(\partial(\calL^d \wedge \xi) \right)
  \end{gather*}
  where $\zeta$ (resp. $\xi$) ranges over the summable
  $m$-vector fields (resp. $(m+1)$-vector fields) supported in $K$ of
  $L^1$-norm 1. As
  \[
  \bF(\calL^d \wedge \zeta) \leq \bM(\calL^d \wedge \zeta) \leq 1
  \text{ and } \bF\left(\partial(\calL^d \wedge \xi) \right) \leq
  \bM\left(\partial(\calL^d \wedge \xi) \right) \leq 1,
  \]
  one finally proves the desired inequality.
\end{proof}

\begin{coro}
  \label{cor:411}
  If $\omega \in C^\infty(\R^d, \bigwedge^m \R^d)$, $\eta \in
  C^\infty(\R^d, \bigwedge^{m'} \R^d)$ and $K$ is a compact set that
  satisfies (A) and (B), then $\| \omega \wedge \eta
  \|_{\bCH^{m+m',1}, K} \leq C \| \omega \|_{\bCH^{m, 1}, K} \| \eta
  \|_{\bCH^{m', 1}, K}$, for some constant $C$.
\end{coro}

\begin{proof}
  It is a consequence of the identity $\diff (\omega \wedge \eta) =
  \diff \omega \wedge \eta + (-1)^m \omega \wedge \diff \eta$. The
  reader interested in estimating the constant may consult
  \cite[1.8.1]{Fede}.
\end{proof}

\subsection{Duality and a generalized Gauss-Green formula for fractal boundaries}
\begin{thm}
  \label{thm:duality}
  Let $K$ be a compact convex subset of $\R^d$. Then
  $\bCH^{m,\alpha}(K)$ is (isometrically isomorphic to) the dual space
  of $\bF_m^\alpha(K)$. The corresponding duality bracket $\langle
  \cdot, \cdot \rangle \colon \bCH^{m,\alpha}(K) \times
  \bF_m^\alpha(K) \to \R$ satisfies
  \begin{equation}
    \label{eq:dualB}
    \langle \omega, T \rangle = \omega(T)
  \end{equation}
  for all $\omega \in \bCH^{m,\alpha}(K)$ and $T \in \bN_m(K)$.
\end{thm}

\begin{proof}
  First we define the duality bracket. For $\omega \in \bCH^{m,
    \alpha}(K)$ and $T \in \bF_m^\alpha(K)$ we set
  \begin{equation}
    \label{eq:defbracket}
    \langle \omega, T \rangle = \sum_{k=0}^\infty \omega(T_k)
  \end{equation}
  where $(T_k)$ is a decomposition of $T$ as in the
  definition~\ref{def:frac} and all the normal currents $T_k$ are
  supported in $K$. The right-hand side series is convergent, as
  \[
  \sum_{k=0}^\infty |\omega(T_k)| \leq \| \omega
  \|_{\bCH^{m,\alpha}(K)} \sum_{k=0}^\infty \bN(T_k)^{1 - \alpha}
  \bF(T_k)^\alpha < \infty
  \]
  We need however to show that the right-hand side
  of~\eqref{eq:defbracket} does not depend on the choice of the
  decomposition $(T_k)$. We reduce to the case where $T = 0$. First we
  define $\hat{\omega} = p^{\#} \omega \in \bCH^{m,\alpha}(\R^d)$,
  where $p \colon \R^d \to K$ is the orthogonal projection onto
  $K$. For $\varepsilon > 0$, we know that the convolution
  $\hat{\omega} * \Phi_\varepsilon$ is a smooth $m$-form. As $\sum_k
  T_k = 0$ in flat norm (and therefore weakly), one has
  \[
  \sum_{k=0}^\infty \hat{\omega} * \Phi_\varepsilon (T_k) = \lim_{n
    \to \infty} \hat{\omega} * \Phi_\varepsilon \left( \sum_{k=0}^n
  T_k \right) = 0
  \]
  In addition,
  \begin{align*}
    |\hat{\omega} * \Phi_\varepsilon (T_k)| & \leq |\omega(p_\# (T *
    \Phi_\varepsilon)) | \\ & \leq \| \omega \|_{\bCH^{m,\alpha}(K)}
    \bN( p_\#(T_k * \Phi_\varepsilon))^{1 - \alpha} \bF(p_\#(T_k *
    \Phi_\varepsilon))^\alpha \\ & \leq \| \omega
    \|_{\bCH^{m,\alpha}(K)} \bN(T_k)^{1 - \alpha} \bF(T_k)^\alpha
  \end{align*}
  and for all $k \geq 0$, one has
  \[
  \lim_{\varepsilon \to 0} \hat{\omega} * \Phi_\varepsilon(T_k) =
  \hat{\omega}(T_k) = \omega(T_k)
  \]
  This follows from the continuity of charges, as the normal currents
  $(T_k * \Phi_\varepsilon)_{0 < \varepsilon \leq 1}$ have uniformly
  bounded normal masses, are supported in $B(\spt T_k, 1)$ and $T_k *
  \Phi_\varepsilon \to T_k$ in flat norm by Proposition~\ref{prop:22}.
  By the Lebesgue dominated convergence theorem,
  \[
  \sum_{k=0}^\infty \omega(T_k) = \lim_{\varepsilon \to 0}
  \sum_{k=0}^\infty \hat{\omega}* \Phi_\varepsilon (T_k) = 0
  \]
  as desired.

  The map $T \mapsto \langle \omega, T \rangle$ is
  $\bF^\alpha$-continuous. Indeed, by~\eqref{eq:defbracket}, one has
  \[
  |\langle \omega, T \rangle| \leq \|\omega \|_{\bCH^{m,\alpha}(K)}
  \sum_{k=0}^\infty \bN(T_k)^{1 - \alpha} \bF(T_k)^\alpha
  \]
  and by passing to the infimum $|\langle \omega, T \rangle| \leq \|
  \omega\|_{\bCH^{m,\alpha}(K)} \bF^\alpha(T)$. Hence, the map
  \[
  \Upsilon \colon \bCH^{m,\alpha}(K) \to \bF^\alpha_m(K)^* \colon
  \omega \mapsto \langle \omega, \cdot \rangle
  \]
  is well-defined and $\| \Upsilon \| \leq 1$.

  Finally, $\Upsilon$ is surjective. Indeed, consider a linear
  continuous functional $\varphi \in \bF^\alpha_m(K)^*$ and denote by
  $\omega$ its restriction to $\bN_m(K)$. Then for all $T \in
  \bN_m(K)$, one has
  \[
  |\omega(T)| = |\varphi(T)| \leq \| \varphi \| \bF^\alpha(T) \leq \|
  \varphi \| \bN(T)^{1 - \alpha} \bF(T)^\alpha
  \]
  Accordingly, $\omega \in \bCH^{m,\alpha}(K)$ and it is
  straightforward that $\varphi = \Upsilon(\omega)$ by density of
  $\bN_m(K)$ in $\bF^\alpha_m(K)$ and $\|\omega\|_{\bCH^{m,\alpha}(K)}
  \leq \| \Upsilon(\omega) \|$. We conclude that $\Upsilon$ is a
  surjective isometry.
\end{proof}
Observe that if \( K \) is convex, Stokes' formula  
\[
\langle \diff \omega, T \rangle = \langle \omega, \partial T \rangle,
\]
which is valid by definition for \( \omega \in \bCH^{m-1, \alpha}(K)
\) and \( T \in \bN_m(K) \), extends by density to all \( T \in
\bF^\alpha_m(K) \).  Setting \( m = d \), \( T = \llbracket U
\rrbracket \), where \( U \) is a bounded open set with finite
\((1-\alpha)\)-perimeter, and choosing \( \omega \) as the
\(\alpha\)-fractional charge associated with an \(\alpha\)-Hölder
continuous \((d-1)\)-form, one arrives at a generalized Gauss-Green
theorem.  Notably, this framework encompasses the case where \( U \)
is \((d-1+\alpha)\)-summable, because of
Proposition~\ref{prop:firstexample}, thereby extending the generalized
Gauss-Green formula established by J. Harrison and A. Norton in
\cite{HarrNort}.

\subsection{Weak* convergence of fractional charges}
We say that a sequence $(\omega_n)$ in $\bCH^{m,\alpha}(K)$
\emph{converges weakly-*} to $\omega \in \bCH^{m, \alpha}(K)$ whenever
\begin{itemize}
\item[(A)] the sequence $(\omega_n)$ is bounded in $\bCH^{m, \alpha}(K)$;
\item[(B)] $\omega_n \to \omega$ in $\bCH^{m}(K)$.
\end{itemize}
One could prove that this notion of weak* convergence corresponds to the duality formulation given in Theorem~\ref{thm:duality} whenever $K$ is convex; however, we shall not rely on this more precise statement.

The next proposition, which establishes a compactness property
of $\bCH^{m, \alpha}(K)$, implies that condition (B) can be replaced with
\begin{itemize}
\item[(B')] $\omega_n \to \omega$ weakly in $\bCH^m(K)$, that is,
  $\omega_n(T) \to \omega(T)$ for all $T \in \bN_m(K)$.
\end{itemize}
\begin{prop}
  Any bounded sequence in $\bCH^{m, \alpha}(K)$ has a subsequence that
  converges in $\bCH^m(K)$ to an $\alpha$-fractional charge.
\end{prop}

\begin{proof}
  First observe that $\| \cdot \|_{\bCH^{m, \alpha}(K)}$ (defined on
  $\bCH^m(K)$ with values in $[0, \infty]$) is lower semi-continuous
  with respect to weak convergence. Therefore, we only need to check
  that a sequence $(\omega_n)$ that satisfies
  \[
  M \defeq \sup_n \| \omega_n \|_{\bCH^{m, \alpha}(K)} < \infty
  \]
  has a convergent subsequence in $\bCH^m(K)$. This is an easy
  consequence of the compactness criterion
  (Proposition~\ref{prop:compactCH}), as for any $T \in \bN_m(K)$, any
  integer $n$ and $\varepsilon > 0$, one has, by Young's inequality,
  \begin{align*}
  |\omega_n(T)| 
  & \leq M \bN(T)^{1 - \alpha} \bF(T)^\alpha \\
  & \leq \left( \varepsilon \bN(T) \right)^{1 - \alpha} \left( \frac{M^{1/\alpha}}{\varepsilon^{(1-\alpha)/\alpha}} \bF(T) \right)^\alpha \\
  &  \leq (1 - \alpha)\varepsilon \bN(T) +
  \frac{\alpha M^{1/\alpha}}{\varepsilon^{(1-\alpha)/\alpha}}
  \bF(T) \\
  & \leq \varepsilon \bN(T) +
   \frac{\alpha M^{1/\alpha}}{\varepsilon^{(1-\alpha)/\alpha}}
  \bF(T). \qedhere
  \end{align*}
\end{proof}

The definition of weak* convergence is adapted for charges over
$\R^d$ as follows. We say that $\omega_n \to \omega$ weakly-* in
$\bCH^{m, \alpha}(\R^d)$ whenever
\begin{itemize}
\item[(A)] $(\omega_n)$ is bounded in $\bCH^{m, \alpha}(\R^d)$. As the
  topology of $\bCH^{m, \alpha}(\R^d)$ is induced by the family of
  seminorms $\| \cdot \|_{\bCH^{m, \alpha}, K}$, this means for
  \[
  \sup_n \| \omega_n \|_{\bCH^{m,\alpha}, K} < \infty \text{ for any
    compact } K \subset \R^d;
  \]
\item[(B)] $\omega_n \to \omega$ in $\bCH^{m}(\R^d)$.
\end{itemize}
As before, one proves that (B) can be replaced with the weaker
\begin{itemize}
\item[(B')] $\omega_n \to \omega$ weakly in $\bCH^{m}(\R^d)$.
\end{itemize}
The smoothing of charges provides an example of weak*
convergence. Precise estimates are given in the next proposition. We recall that $B(K, \varepsilon)$ denotes the closed tubular $\varepsilon$-neighborhood of a compact set $K \subset \R^d$, see Subsection~\ref{subsec:notations}.
\begin{prop}
  \label{prop:estimatesSmoothing}
  Let $\omega \in \bCH^{m, \alpha}(\R^d)$, let $K \subset \R^d$ be
  compact and $\varepsilon \in \mathopen{(} 0, 1 \mathclose{]}$. We have
  \begin{itemize}
  \item[(A)] $\| \omega * \Phi_\varepsilon \|_{\bCH^{m, \alpha}, K}
    \leq \| \omega \|_{\bCH^{m, \alpha}, B(K, \varepsilon)}$;
  \item[(B)] $\| \omega * \Phi_\varepsilon \|_{\bCH^{m, 1}, K} \leq C
    \varepsilon^{\alpha - 1} \| \omega \|_{\bCH^{m, \alpha}, B(K,
      \varepsilon)}$;
  \item[(C)] $\| \omega - \omega * \Phi_\varepsilon \|_{\bCH^m, K}
    \leq C \varepsilon^\alpha \| \omega \|_{\bCH^{m, \alpha}, B(K,
      \varepsilon)}$.
  \end{itemize}
\end{prop}
\begin{proof}
  Let $T \in \bN_m(K)$ be arbitrary. Regarding (A), we have
  \begin{equation}
    \label{eq:proofS}
    |\omega * \Phi_\varepsilon(T)| = |\omega(T * \Phi_\varepsilon)|
    \leq \| \omega \|_{\bCH^{m,\alpha}(K_\varepsilon)} \bN(T *
    \Phi_\varepsilon)^{1-\alpha} \bF(T * \Phi_\varepsilon)^\alpha
  \end{equation}
  because $\spt (T * \Phi_\varepsilon) \subset B(K, \varepsilon)$. By
  Proposition~\ref{prop:22}(E),
  \[
  |\omega * \Phi_\varepsilon(T)| \leq \| \omega \|_{\bCH^{m,\alpha},
    B(K, \varepsilon)} \bN(T)^{1 - \alpha} \bF(T)^\alpha.
  \]
  Therefore $\|\omega * \Phi_\varepsilon\|_{\bCH^{m, \alpha},K} \leq
  \|\omega \|_{\bCH^{m, \alpha}, B(K,\varepsilon)}$.

  (B) is obtained by combining~\eqref{eq:proofS} with
  Proposition~\ref{prop:22}(F).

  (C). This time,
  \begin{align*}
    | (\omega - \omega * \Phi_\varepsilon)(T) | & = |\omega(T - T *
    \Phi_\varepsilon)| \\ & \leq \| \omega \|_{\bCH^{m, \alpha}, B(K,
      \varepsilon)} \bF(T - T * \Phi_\varepsilon)^\alpha \bN(T - T *
    \Phi_\varepsilon)^{1 - \alpha} \\ & \leq C \| \omega \|_{\bCH^{m,
        \alpha}, B(K, \varepsilon)} \varepsilon^\alpha \bN(T)^\alpha
    \left( \bN(T) + \bN(T * \Phi_\varepsilon) \right)^{1 - \alpha} &
    \text{Prop.~\ref{prop:22}(F)} \\ & \leq C \| \omega
    \|_{\bCH^{m,\alpha}, B(K, \varepsilon)} \varepsilon^\alpha \bN(T)
    & \text{Prop.~\ref{prop:22}(E)}
  \end{align*}
  We conclude with the arbitrariness of $T$.
\end{proof}
Another example of weak* convergence is given by the following proposition.
\begin{prop}
  Let $(\omega_n)$ be a sequence in $\rmLip^\alpha_{\rmloc}(\R^d;
  \wedge^m \R^d)$ such that, for all compact subsets $K \subset \R^d$,
  one has
  \[
  \sup_n \rmLip^\alpha(\omega_{n\mid K}) < \infty \text{ and }
  \omega_n \to \omega \text{ uniformly on } K.
  \]
  Then $\omega_n \to \omega$ weakly-* in $\bCH^{m,\alpha}(\R^d)$.
\end{prop}
\begin{proof}
  Let $T \in \bN_m(\R^d)$. From
  \[
  \omega_n(T) - \omega(T) = \int_{\spt T} \langle \omega_n(x) - \omega(x),  \vec{T}(x) \rangle \diff \|T\|(x)
  \]
  we infer $|\omega_n(T) - \omega(T)| \leq \|\omega_n -
  \omega\|_{\infty, \spt T} \bM(T) \to 0$. This proves that $\omega_n
  \to \omega$ weakly in $\bCH^m(\R^d)$.

  For any compact subset $K \subset \R^d$, one has
  \[
  \| \omega_n \|_{\bCH^{m, \alpha}, K} \leq C \max \{ \| \omega_n
  \|_{\infty, K}, \rmLip^\alpha(\omega_{n \mid K}) \}
  \]
  by~\eqref{eq:LambdaC0}. This guarantees that the sequence
  $(\omega_n)$ is bounded in $\bCH^{m, \alpha}(\R^d)$.
\end{proof}

\subsection{A Littlewood-Paley type lemma}
The following interpolation lemma, though simple, will prove crucial
in the remainder of this paper. Indeed, it will be an important
ingredient in the proof of our three main results and constructions:
the compactness theorem, pushforwards by Hölder maps and the wedge
product of fractional charges.
\begin{lemma}
  \label{lemma:LPcharge}
  Let $K$ be a compact subset of $\R^d$, let $0 < \alpha < \beta \leq
  1$ and $\kappa \geq 0$. Suppose $(\omega_n)$ is a sequence in
  $\bCH^{m,\beta}(K)$ such that, for all $n \geq 0$, one has
  \begin{equation}
    \label{eq:compo}
  \| \omega_n \|_{\bCH^{m, \beta}(K)} \leq 2^{-n(\alpha - \beta)} \kappa
  \qquad
  \text{and}
  \qquad
  \| \omega_n \|_{\bCH^{m}(K)} \leq 2^{-n\alpha} \kappa
  \end{equation}
  Then $\sum_{n=0}^\infty \omega_n \colon T \mapsto \sum_{n=0}^\infty
  \omega_n(T)$ is well-defined, and is a $\alpha$-fractional charge
  such that
  \[
  \left\| \sum_{n=0}^\infty \omega_n \right\|_{\bCH^{m,\alpha}(K)}
  \leq C \kappa
  \]
  where $C = C(\beta, \alpha)$ is a constant.
\end{lemma}
This lemma calls for several remarks:
\begin{itemize}
\item First, the convergence of the series $\sum_{n=0}^\infty
  \omega_n$ is only weak* in $\bCH^{m,\alpha}(K)$. Indeed, by applying
  the same lemma to the truncated sequence $(\omega_0, \dots,
  \omega_N, 0, \dots)$ for all $N \geq 0$, one obtains
  \[
  \sup_{N} \left\| \sum_{n=0}^N \omega_n \right\|_{\bCH^{m,\alpha}(K)}
  \leq C\kappa
  \]
  therefore ensuring that the sequence of partial sums is bounded in
  $\bCH^{m, \alpha}(K)$. Moreover, the sequence of partials sums
  converges weakly in $\bCH^m(K)$, as already asserted in
  Lemma~\ref{lemma:LPcharge}.
\item Let us consider the case $m = 0$ and $\beta = 1$. According to
  Propositions~\ref{prop:CH0} and~\ref{prop:CH0alpha}, charges,
  $\alpha$-fractional charges and $1$-fractional charges can be
  identified with continuous, $\alpha$-Hölder continuous, and
  Lipschitz functions, respectively. In this setting,
  Lemma~\ref{lemma:LPcharge} establishes the $\alpha$-Hölder
  continuity of a uniformly convergent series $f = \sum_{n=0}^\infty
  f_n$. The estimates in~\eqref{eq:compo} translate to $\rmLip f_n
  \leq C 2^{-n(\alpha - 1)} \kappa$ and $\|f_n\|_{\infty, K} \leq C
  2^{-n\alpha} \kappa$, which aligns with the expected behavior of
  the Littlewood-Paley components of $f$. Moreover, the weak*
  convergence of the sum $\sum_{n=0}^\infty \omega_n$ echoes with the
  weak orthogonality properties observed in the Littlewood-Paley
  decomposition in harmonic analysis. In the context of metric spaces,
  and $m = 0$, the existence of such a decomposition is proved
  in~\cite[Appendix~B,~2.6]{Grom}.
\end{itemize}

\begin{proof}[Proof of Lemma~\ref{lemma:LPcharge}]
  Let $T \in \bN_m(K)$. Using~\eqref{eq:compo}, we can estimate
  $|\omega_n(T)|$ in two different ways
  \[
  |\omega_n(T)| \leq 2^{-n(\alpha - \beta)} \kappa \bN(T)^{1 - \beta}
  \bF(T)^{\beta} \qquad \text{and} \qquad |\omega_n(T)| \leq
  2^{-n\alpha} \kappa \bN(T)
  \]
  Let $N$ be a nonnegative integer, to be determined later. We have
  \begin{align*}
  \left| \sum_{k=0}^N \omega_n(T) \right| & \leq \kappa \bN(T)^{1 -
    \beta} \bF(T)^\beta \sum_{k=0}^N 2^{-n(\alpha-\beta)} \\ & \leq C
  \kappa \bN(T)^{1 - \beta} \bF(T)^\beta 2^{-N(\alpha - \beta)}
  \end{align*}
  and
  \[
  \left| \sum_{k=N+1}^\infty \omega_n(T) \right| \leq \kappa \bN(T)
  \sum_{k=N+1}^\infty 2^{-n\alpha} \leq C \kappa \bN(T) 2^{-N\alpha}
  \]
  Since $\bF(T) \leq \bN(T)$, we choose $N$ to be a nonnegative
  integer such that
  \[
  2^{-(N+1)} \leq \frac{\bF(T)}{\bN(T)} \leq 2^{-N}
  \]
  Using the preceding inequalities, we obtain
  \begin{align*}
  \left| \sum_{k=0}^\infty \omega_n(T) \right| & \leq C \kappa \left(
  \bN(T)^{1 - \beta} \bF(T)^\beta \left(
  \frac{\bF(T)}{\bN(T)}\right)^{\alpha - \beta} + \bN(T) \left(
  \frac{\bF(T)}{\bN(T)}\right)^\alpha \right) \\
  & \leq C \kappa \bN(T)^{1 - \alpha} \bF(T)^\alpha \qedhere
  \end{align*}
\end{proof}

When adapted to charges over $\R^d$, the preceding proposition takes
the following form. We state it only for the case $\beta = 1$, which
will be used.

\begin{coro}
  \label{cor:LP}
  Suppose $0 < \alpha < 1$. Let $(\omega_n)$ be a sequence in
  $\bCH^{m,1}(\R^d)$ such that, for each compact $K \subset \R^d$,
  there is $\kappa(K) \geq 0$ such that
  \[
  \|\omega_n\|_{\bCH^{m,1}, K} \leq \kappa(K) 2^{n(1 - \alpha)} \text{ and }
  \|\omega_n\|_{\bCH^m, K} \leq \frac{\kappa(K)}{2^{n\alpha}}
  \]
  for all $n$. Then $\sum_{n=0}^\infty \omega_n$ converges weakly-* to
  a charge in $\bCH^{m,\alpha}(\R^d)$. In addition, for all compact $K
  \subset \R^d$,
  \begin{equation}
    \label{eq:tobeused}
    \left\| \sum_{n=0}^\infty \omega_n
    \right\|_{\bCH^{m, \alpha}, K} \leq C \kappa(K)
  \end{equation}
\end{coro}

Of course, it is enough to check the hypothesis of
Corollary~\ref{cor:LP} when $K$ ranges over (non degenerate) closed
balls. A Littlewood-Paley type decomposition of a fractional charge
$\omega \in \bCH^{m, \alpha}(\R^d)$ will be obtained by convolution
\[
\omega = \omega * \Phi_{1} + \sum_{n=0}^\infty \left( \omega *
\Phi_{2^{-(n+1)}} - \omega * \Phi_{2^{-n}} \right).
\]
This decomposition will play a pivotal role in the forthcoming proof
of Theorem~\ref{thm:main}.

\section{Compactness theorem}
\label{sec:compact}

The purpose of this section is to prove the following compactness
theorem
\begin{thm}
  \label{thm:compactness}
  Let $K \subset \R^d$ be a compact subset, $0 \leq \alpha < 1$ and $c
  \geq 0$. Then \[\{T \in \bF^\alpha_m(K) : \bF^\alpha(T) \leq c\}\]
  is $\bF^\beta$-compact, for all $\beta \in \mathopen{(} \alpha, 1
  \mathclose{]}$.
\end{thm}

%The proof of Theorem~\ref{thm:compactness} can be reduced to the case
%where $K$ is compact convex. We will make this assumption throughout
%this section.

\subsection{Little fractional charges}
Our first step towards the compactness theorem is to prove that
$\bF^\alpha_m(K)$ is itself a dual space, specifically the space of
so-called little fractional charges. This will help us show that the
ball $\{T \in \bF^\alpha_m(K) : \bF^\alpha(T) \leq c\}$ is closed in
flat norm. 

Let $\alpha \in (0, 1)$. We define the map
\[
\iota_K \colon C^\infty(\R^d; \wedge^m \R^d) \to \bCH^{m,\alpha}(K)
\colon \omega \mapsto (T \mapsto T(\omega))
\]
Note that $\iota_K(\omega) = \Lambda(\omega_{\mid K})$ for all smooth
forms $\omega$, where $\Lambda$ is the map
from~\eqref{eq:Lambdaalpha}. For all $S \in \bN_{m+1}(K)$, one has
\begin{align*}
  |\iota_K(\omega)(T)| & \leq |(T - \partial S)(\omega)| + |S(\diff
  \omega)| \\ & \leq \left( \bM(T - \partial S) + \bM(S) \right) \max
  \{ \|\omega\|_{\infty, K}, \| \diff \omega \|_{\infty, K} \}
\end{align*}
In case $K$ is convex, we take the infimum over all $S \in
\bN_{m+1}(K)$ and deduce that
\[
\|\iota_K(\omega)\|_{\bCH^{m,1}(K)} \leq \max \{ \|\omega\|_{\infty,
  K}, \| \diff \omega \|_{\infty, K} \}
\]
and
\begin{equation}
  \label{eq:iotaKC0}
  \| \iota_K(\omega) \|_{\bCH^{m, \alpha}(K)} \leq
    \|\iota_K(\omega)\|_{\bCH^{m,1}(K)} \leq \max \{
    \|\omega\|_{\infty, K}, \| \diff \omega \|_{\infty, K} \}.
\end{equation}

The space $\bch^{m,\alpha}(K)$ of \emph{little $\alpha$-fractional
  charges} is defined to be the closure in $\bCH^{m, \alpha}(K)$ of
$\operatorname{im} \iota_K$. It is given the norm inherited from
$\bCH^{m, \alpha}(K)$.

%The consideration of the case $m=0$ provides a natural motivation for
%introducing little fractional charges. Specifically, it is known that
%$\bCH^{0,\alpha}(K)$ is isomorphic to $\rmLip^\alpha(K)$, the space of
%$\alpha$-Hölder continuous functions on $K$, see
%\cite[Proposition~4.4]{Boua1}. This space is the double dual of
%$\operatorname{lip}^\alpha(K)$, the closure of $\rmLip(K)$ in
%$\rmLip^\alpha(K)$. In the literature, $\rmlip^\alpha(K)$ is often
%referred to as the space of little Hölder functions. Drawing on this
%analogy, we have adopted the term ``little fractional charge''.

\begin{prop}
  \label{prop:dualch}
  Let $K$ be a compact convex subset of $\R^d$ and $\alpha \in (0,
  1)$. Then $\bF^\alpha_m(K)$ is isomorphic to the dual of $\bch^{m,
    \alpha}(K)$. The corresponding duality bracket is the restriction
  of~\eqref{eq:dualB} to $\bch^{m,\alpha}(K) \times \bF^\alpha_m(K)$.
\end{prop}

\begin{proof}
  Let $\Upsilon \colon \bF^\alpha_m(K) \to \bch^{m,\alpha}(K)^*$ be
  the map $T \mapsto \langle \cdot , T \rangle$. It is clear that
  $\Upsilon$ is linear and continuous, as the bracket satisfies the
  continuity condition
  \[
  | \langle \omega, T \rangle | \leq \| \omega \|_{\bCH^{m,
      \alpha}(K)} \bF^\alpha(T) = \| \omega \|_{\bch^{m,\alpha}(K)}
  \bF^\alpha(T)
  \]
  for all $\omega \in \bch^{m,\alpha}(K)$ and $T \in \bF^\alpha_m(K)$.

  Moreover, $\Upsilon$ is injective, for if $\Upsilon(T) = 0$, then
  for all compactly supported smooth $m$-form $\omega$, one has
  \[
  T(\omega) = \langle \iota_K(\omega), T \rangle = \Upsilon(T)
  (\iota_K(\omega)) = 0
  \]
  Next we turn to surjectivity of $\Upsilon$. Let $\varphi \in
  \bch^{m,\alpha}(K)^*$. We define $T \colon C^\infty(\R^d; \wedge^m
  \R^d) \to \R$ by
  \[
  T(\omega) = \varphi(\iota_K(\omega))
  \]
  It is clear that $T$ is linear. By the continuity of $\varphi$
  and~\eqref{eq:iotaKC0}, one infers
  \[
  |T(\omega)| \leq \| \varphi \| \max \{ \| \omega\|_{\infty, K},
  \|\diff \omega\|_{\infty, K} \}
  \]
  This implies in particular that $T$ is an $m$-current, compactly
  supported in $K$. But at this point, it is not even clear that $T$ is a flat chain (for this claim, we need to prove that $T$ is a limit of normal currents under the flat norm).

  For each integer $n \geq 1$, we define $T_n = T * \Phi_{2^{-n}}$
  which is a smooth $m$-current. Let us denote $L = B(K, 1) = \{x \in \R^d : \operatorname{dist}(x, K) \leq 1\}$. We
  remark that all the currents $T_n$ belong to $\bN_m(L)$. Observe
  that $L$ is still a compact convex subset of $\R^d$.  Let us denote
  by $p \colon \R^d \to L$ the orthogonal projection onto $L$.
  \begin{Claim*}
    For all $n$, there is a smooth $m$-form $\theta_n$ such that
    \begin{equation}
      \label{eq:theta_n}
    (T_{n+1} - T_n)(\theta_n) \geq \frac{\bF^\alpha(T_{n+1} - T_n)}{2}
    \quad \text{and} \quad \| \iota_L(\theta_n) \|_{\bCH^{m,
        \alpha}(L)} \leq 1
    \end{equation}
  \end{Claim*}
  To make notations shorter, let us abbreviate $R = T_{n+1} - T_n$
  during the proof of this claim.  By Theorem~\ref{thm:duality}, there
  exists a charge $\theta \in \bCH^{m,\alpha}(L)$ such that
  \[
  \theta(R) \geq \frac{2}{3}\bF^\alpha(R)
  \text{ and } \| \theta \|_{\bCH^{m,\alpha}(L)} = 1
  \]
  Let us define the smooth form $\theta_n = p^\# \theta *
  \Phi_\varepsilon$, where $\varepsilon$ is to be determined
  shortly. We have
  \begin{align*}
    R(\theta_n) & = \theta\left( p_\#(R *
    \Phi_\varepsilon) \right) \\ & = \theta \left( R *
    \Phi_\varepsilon \right)
  \end{align*}
  if $\varepsilon + 2^{-n} \leq 1$ (this ensures that $\spt (R * \Phi_\varepsilon) \subset L$). By Proposition~\ref{prop:22}(E)
  and (G), one has
  \begin{align*}
    \left| \theta(R) - \theta \left(R * \Phi_\varepsilon \right) \right|&
    \leq \bN(R - R * \Phi_\varepsilon)^{1 - \alpha} \bF(R - R*
    \Phi_\varepsilon)^\alpha \\ & \leq C \varepsilon^\alpha \bN(R)
  \end{align*}
  so that
  \begin{align*}
  (T_{n+1} - T_n)(\theta_n) & = R(\theta_n) \\ & = \theta(R * \Phi_\varepsilon) \\ & \geq \theta(R) - | \theta(R) - \theta(R * \Phi_\varepsilon) | \\ & \geq \frac{2}{3} \bF^\alpha(T_{n+1} - T_n) - C \varepsilon^\alpha \bN(R) 
  \end{align*}  
  and the first inequality in~\eqref{eq:theta_n} is achieved if
  $\varepsilon$ (depending on $R$) is small enough. As for the second one, one computes,
  for $S \in \bN_m(L)$,
  \begin{align*}
    |\iota_L(\theta_n)(S)| & = |S(\theta_n)| \\
    & = |\theta(p_\#(S * \Phi_\varepsilon))| \\
    & \leq \bN(p_\#(S * \Phi_\varepsilon))^{1-\alpha} \bF(p_\#(S * \Phi_\varepsilon))^\alpha \\
    & \leq \bN(S)^{1 - \alpha} \bF(S)^\alpha
  \end{align*}
  which guarantees that $\|\iota_L(\theta_n)\|_{\bCH^{m,\alpha}(L)}
  \leq 1$.
  \begin{Claim*}
    For every smooth $m$-form $\omega$ and $0 < \varepsilon < 1$, one
    has
    \[
    (T * \Phi_\varepsilon)(\omega) = T(p^\# \iota_L(\omega) *
    \Phi_\varepsilon)
    \]
  \end{Claim*}
  Indeed, let $S$ be a normal $m$-current supported in the set $U
  = \{ x \in \R^d : \operatorname{dist}(K, x) < 1 -
  \varepsilon\}$. The charge $p^\# \iota_L(\omega) * \Phi_\varepsilon$
  appearing in the right-hand side satisfies
  \begin{align*}
  (p^\# \iota_L(\omega) * \Phi_\varepsilon) (S) & = \iota_L(\omega)(
    p_\#(S * \Phi_\varepsilon) ) \\ & = \iota_L(\omega)( S *
    \Phi_\varepsilon) \\ & = S * \Phi_\varepsilon(\omega) \\ & =
    S(\omega * \Phi_\varepsilon)
  \end{align*}
  It follows that the smooth forms $\omega * \Phi_\varepsilon$ and
  $p^\# \iota_L(\omega) * \Phi_\varepsilon$ agree on $U$ (here we use that $U$ is open). As $T$ is
  supported in $K \subset U$, the claim follows.

  It follows from the two preceding claims that, setting
  \[
  \eta_n = p^\#\iota_L(\theta_n) * \Phi_{2^{-(n+1)}} -
  p^\#\iota_L(\theta_n) * \Phi_{2^{-n}}
  \]
  one has
  \begin{equation}
    \label{eq:varphiiotaKetan}
  \varphi(\iota_K(\eta_n)) = T(\eta_n) = (T_{n+1} - T_n)(\theta_n)
  \geq \frac{\bF^\alpha(T_{n+1} - T_n)}{2}
  \end{equation}
  Our last goal is to apply Lemma~\ref{lemma:LPcharge} with $\beta =
  1$, and for this, we need to estimate the norms $\| \iota_K(\eta_n)
  \|_{\bCH^{m, 1}(K)}$ and $\|\iota_L(\eta_n)\|_{\bCH^m(K)}$.

  Let $S \in \bN_m(K)$. Then
  \begin{align*}
    |\iota_K(\eta_n)(S)| & = |S(\eta_n)| \\ & = \left|
    \iota_L(\theta_n) \left( p_\#(S * \Phi_{2^{-(n+1)}} - S*
    \Phi_{2^{-n}})\right) \right| \\ & = \left| \iota_L(\theta_n)
    \left(S * \Phi_{2^{-(n+1)}} - S* \Phi_{2^{-n}}\right) \right| \\ &
    \leq \bN \left( S * \Phi_{2^{-(n+1)}} - S*
    \Phi_{2^{-n}}\right)^{1-\alpha} \bF \left(S * \Phi_{2^{-(n+1)}} -
    S* \Phi_{2^{-n}} \right)^\alpha
  \end{align*}
  as $\| \iota_L(\theta_n) \|_{\bCH^{m,\alpha}(L)} \leq 1$. Using
  Proposition~\ref{prop:22}(E) and (F), one first estimates
  \[
  |\iota_K(\eta_n)(S)| \leq C 2^{-n(\alpha - 1)} \bF(S)
  \]
  which entails that $\|\iota_K(\eta_n)\|_{\bCH^{m, 1}(K)} \leq
  2^{-n(1 - \alpha)}$. On the other hand, Proposition~\ref{prop:22}(G)
  implies that
  \[
  \bF \left(S * \Phi_{2^{-(n+1)}} - S* \Phi_{2^{-n}} \right) \leq
  \bF(S - S* \Phi_{2^{-(n+1)}}) + \bF(S - S * \Phi_{2^{-n}}) \leq C
  2^{n} \bN(S)
  \]
  thus $|\iota_K(\eta_n)(S)| \leq C 2^{n\alpha} \bN(S)$. Hence $\|
  \iota_K(\eta_n) \|_{\bCH^{m}(K)} \leq C 2^{n\alpha}$.

  Finally, by applying Lemma~\ref{lemma:LPcharge} to the sequences of
  charges
  \[(\iota_K(\eta_1), \iota_K(\eta_2), \dots, \iota_K(\eta_N),
  0, 0, \dots)
  \]
  where $N$ is an arbitrary integer, we obtain that
  \[
  M = \sup_{N \geq 1} \left\|\sum_{n=1}^N \iota_K(\eta_n)
  \right\|_{\bCH^{m,\alpha}(K)} < \infty
  \]
  Hence, by~\eqref{eq:varphiiotaKetan}
  \[
   \sum_{n=1}^N \bF^\alpha(T_{n+1} - T_n) \leq 2 \varphi \left(
   \sum_{n=1}^N \iota_K(\eta_n) \right) \leq 2M \| \varphi
   \|_{\bch^{m,\alpha}(K)^*}
  \]
  As $N$ is arbitrary, it follows that $(T_n)$ is a
  $\bF^\alpha$-Cauchy sequence in the Banach space
  $\bF^\alpha_m(L)$. Additionally, $(T_n)$ converges weakly to $T$,
  and thus $T \in \bF^\alpha_m(K)$.

  Finally, the linear continuous forms $\Upsilon(T)$ and $\varphi$
  coincide on all little fractional charges of the form
  $\iota_K(\omega)$, for $\omega \in C^\infty(\R^d; \wedge^m
  \R^d)$. By density, $\varphi = \Upsilon(T)$, which concludes the
  proof of the surjectivity of $\Upsilon$.
\end{proof}

\begin{coro}
  \label{coro:closureThm}
  The ball $\{T \in \bF^\alpha_m(K) : \bF^\alpha(T) \leq c\}$ is
  $\bF^\beta$-closed in $\bF^\beta_m(K)$ for all $\beta \in
  \mathopen{(} \alpha, 1 \mathclose{]}$.
\end{coro}

\begin{proof}
  There is no loss in generality in assuming that $K$ is compact
  convex. Let us show that $\bch^{m,\alpha}(K)$ is separable. The
  inequality~\eqref{eq:iotaKC0} guarantees that $\iota_K$ factorizes
  through
  \[
  \begin{tikzcd}
    C^\infty(\R^d; \wedge^m \R^d) \arrow[rr, "\varphi"] && C(K;
    \wedge^m \R^d) \times C(K; \wedge^{m+1} \R^d) \arrow[rr,
      "\tilde{\iota}_K"] && \R
  \end{tikzcd}
  \]
  where $\varphi$ is the map $\omega \mapsto (\omega_{\mid K}, (\diff
  \omega)_{\mid K})$ and $\tilde{\iota}_K$ is continuous. By the
  separability of the middle space, we infer that $\operatorname{im}
  \iota_K$ is separable, and so is $\bch^{m,\alpha}(K)$.

  By the Banach-Alao\u{g}lu theorem, the closed ball $B := \{T \in
  \bF^\alpha_m(K) : \bF^\alpha(T) \leq c\}$ is weakly*
  compact. Furthermore, the induced weak* topology on $B$ is
  metrizable, since $\bch^{m,\alpha}(K)$ is separable,
  by~\cite[2.6.23]{Megg}.

  Let $(T_n)$ be a sequence in $B$ that $\bF^\beta$-converges to some
  $T \in \bF^\beta_m(K)$. By the weak* sequential compactness of $B$,
  we suppose, up to extracting a subsequence, that $(T_n)$ weak*
  converges to $S \in B$. In particular, for each $\omega \in
  C^\infty_c(\R^d; \wedge^m \R^d)$, one has
  \[
  T_n(\omega) = \iota_K(\omega)(T_n) \to \iota_K(\omega)(S) =
  S(\omega)
  \]
  hence the convergence $T_n \to S$ is also weak (in the sense of
  currents). As $\bF^\beta$-convergence implies convergence in flat
  norm, that in turn implies weak convergence, we conclude that $T = S
  \in B$ by uniqueness of the weak limit. This finishes the proof.
\end{proof}

\subsection{Deformation theorem for fractional currents}
The following result shows how well fractional currents can be
approximated by polyhedral currents. It is the version of the
deformation theorem (see \cite[Chapter~6,
    Theorem~5.3]{Simo}) for $\alpha$-fractional currents. It gives the
``totally boundedness'' part of Theorem~\ref{thm:compactness}. 

Let $\varepsilon > 0$. For every subset $I \subset \{1, \dots, d\}$ of cardinal $m$ and $a \in \Z^d$, we define the map
\[
h_{I, a} \colon [0, 1]^m \to \R^d \colon x \mapsto \varepsilon \left( a + \sum_{k=1}^m x_i e_{i_k} \right)
\]
where $i_1 < \cdots < i_m$ are the distinct elements of $I$. Here, $e_1, \dots, e_d$ denotes the canonical basis of $\R^d$. We say that a current $P$ is a \emph{polyhedral $m$-current on the standard $\varepsilon$-grid} whenever it is a finite linear combination of currents of the form $h_{I, a\#} \left\llbracket [0, 1]^m \right\rrbracket$. 

\begin{thm}
  \label{thm:deformation}
  Let $K$ be a compact convex subset of $\R^d$ and $0 \leq \alpha <
  \beta \leq 1$. There is a constant $C = C(\alpha, \beta, K) \geq 0$
  such that for all $T \in \bF_m^\alpha(K)$ and $\varepsilon > 0$, one
  can find a polyhedral $m$-current $P$ on the standard $\varepsilon$-grid
  of $\R^d$ such that
  \[
  \bF^\beta(T - P) \leq C \varepsilon^{\beta - \alpha} \bF^\alpha(T)
  \text{ and } \spt P \subset B(K, C \varepsilon).
  \]
\end{thm}

\begin{proof}
  Consider a decomposition $T = \sum_{k} T_k$ into normal currents
  such that
  \begin{equation}
    \label{eq:decT}
    \sum_{k=0}^\infty \bN(T_k)^{1 - \alpha} \bF(T_k)^\alpha \leq 2
    \bF^\alpha(T) \text{ and } \spt T_k \subset K.
  \end{equation}
  For each $k$, introduce two normal currents $A_k \in \bN_m(K)$ and
  $B_k \in \bN_{m+1}(K)$ such that $T_k = A_k + \partial B_k$ and
  $\bM(A_k) + \bM(B_k) \leq 2 \bF(T_k)$. In particular, $\bM(A_k) \leq
  2 \bM(T_k)$. We apply the deformation theorem \cite[Chapter~6,
    Theorem~5.3]{Simo} at scale $\varepsilon$ to $A_k$ and $B_k$,
  which provides decompositions
  \[
  A_k = Q_k + \partial R_k + S_k \text{ and } B_k = \tilde{Q}_k +
  \partial \tilde{R}_k + \tilde{S}_k
  \]
  where $Q_k$ and $\tilde{Q}_k$ are polyhedral currents on the standard $\varepsilon$-grid, and $R_k, S_k, \tilde{R}_k, \tilde{S}_k$ are normal currents 
  with the estimates
  \begin{gather}
    \bM(R_k) \leq C \varepsilon \bM(A_k) \leq C \varepsilon \bM(T_k) \label{eq:mRk} \\
    \bM(S_k) \leq C \varepsilon \bM(\partial A_k) = C \varepsilon
    \bM(\partial T_k) \label{eq:mSk} \\
    \bM(Q_k) \leq C \bM(A_k) \leq C \bF(T_k) \label{eq:mQk} \\
    \bM(\partial Q_k) \leq C \bM(\partial A_k) = C \bM(\partial T_k) \label{eq:mpartialQk} \\
%
%    \bM(\tilde{R}_k) \leq C \varepsilon \bM(B_k) \leq C \varepsilon
%    \bF(T_k) \label{eq:mRtildek} \\
%
    \bM(\tilde{S}_k) \leq C \varepsilon \bM(\partial B_k) \leq C
    \varepsilon (\bM(T_k) + \bM(A_k)) \leq C \varepsilon
    \bM(T_k) \label{eq:mStildek} \\
    \bM(\tilde{Q}_k) \leq C \bM(B_k) \leq C
    \bF(T_k) \label{eq:mQtildek} \\
    \bM(\partial \tilde{Q}_k) \leq C \bM(\partial B_k) \leq C \left( \bM(T_k) + \bM(A_k) \right) \leq C \bM(T_k) \label{eq:mpartialQtildek}
  \end{gather}
  and with $\spt Q_k, \spt \tilde{Q}_k \subset B(K, C \varepsilon)$.
  Let $n$ be such that
  \begin{equation}
    \label{eq:azeqsd}
    \sum_{k=n+1}^\infty \bN(T_k)^{1 - \beta} \bF(T_k)^\beta \leq
  \sum_{k=n+1}^\infty \bN(T_k)^{1 - \alpha} \bF(T_k)^\alpha \leq
  \varepsilon^{\beta - \alpha} \bF^\alpha(T)
  \end{equation}
  and define
  \[
  P = \sum_{k=0}^n (Q_k + \partial \tilde{Q}_k)
  \]
  which is clearly a polyhedral current on the $\varepsilon$-grid, with
  support in $B(K,C\varepsilon)$. We have
  \[
  \bF(Q_k + \partial \tilde{Q}_k) \leq \bM(Q_k) + \bM(\tilde{Q}_k)
  \leq C \bF(T_k)
  \]
  by~\eqref{eq:mQk} and~\eqref{eq:mQtildek}. This clearly implies that
  \begin{equation}
    \label{eq:fTQk1}
    \bF(T_k - Q_k - \partial \tilde{Q}_k) \leq C \bF(T_k)
  \end{equation}
  On the other hand,
  \begin{align}
  \bF(T_k - Q_k - \partial \tilde{Q}_k) & = \bF(S_k + \partial (R_k +
  \tilde{S}_k)) \notag \\ & \leq \bM(S_k) + \bM(R_k) +
  \bM(\tilde{S}_k) \notag \\ & \leq C
  \varepsilon\bN(T_k) \label{eq:fTQk2}
  \end{align}
  by~\eqref{eq:mSk}, \eqref{eq:mRk} and \eqref{eq:mStildek}. Now,
  by~\eqref{eq:fTQk1} and~\eqref{eq:fTQk2},
  \begin{equation}
    \label{eq:fTQk3}
    \bF(T_k - Q_k - \partial \tilde{Q}_k) \leq C \varepsilon^{1 -
      \alpha/\beta} \bN(T_k)^{1 - \alpha/\beta} \bF(T_k)^{\alpha/\beta}.
  \end{equation}
  Furthermore, one has
  \begin{equation}
    \label{eq:nTQk3}
  \bN(T_k - Q_k - \partial \tilde{Q}_k) \leq \bN(T_k) + \bM(Q_k) +
  \bM(\partial Q_k) + \bM(\partial \tilde{Q}_k) \leq C \bN(T_k)
  \end{equation}
  by~\eqref{eq:mQk}, \eqref{eq:mpartialQk} and
  \eqref{eq:mpartialQtildek}.  Therefore,
  \begin{align*}
    \bF^\beta(T - P) & \leq \sum_{k=0}^n \bN(T_k - Q_k - \partial
    \tilde{Q}_k)^{1 - \beta} \bF(T_k - Q_k - \partial
    \tilde{Q}_k)^\beta + \sum_{k=n+1}^\infty \bF^\beta(T_k) \\ & \leq C
    \varepsilon^{\beta - \alpha} \sum_{k=0}^n \bN(T_k)^{1 - \alpha}
    \bF(T_k)^\alpha + \sum_{k=n+1}^\infty \bN(T_k)^{1 - \beta}
    \bF(T_k)^\beta \\ & \leq C \varepsilon^{\beta - \alpha} \bF^\alpha(T)
  \end{align*}
  by~\eqref{eq:fTQk3}, \eqref{eq:nTQk3}, \eqref{eq:azeqsd} and
  \eqref{eq:decT}.
\end{proof}

\begin{proof}[Proof of the compactness theorem~\ref{thm:compactness}]
  There is no loss in generality in assuming that $K$ is convex. In
  that case, the ball $B = \{T \in \bF^\alpha_m(K) : \bF^\alpha(T)
  \leq c\}$ is $\bF^\beta$-closed by
  Corollary~\ref{coro:closureThm}. It remains to prove that it is
  $\bF^\beta$-totally bounded. Let $\varepsilon \in \mathopen{(} 0, 1
  \mathclose{]}$ and $p \colon \R^d\to K$ be the orthogonal projection
    onto $K$. By the deformation theorem~\ref{thm:deformation}, there
    exists, for any $T \in B$, a polyhedral $m$-current $P$ with $\spt
    P \subset B(K, C\varepsilon)$ and
    \begin{equation}
      \label{eq:distTpP}
  \bF^\beta(T - p_\#P) \leq \bF^\beta(T - P) \leq C \varepsilon^{\beta - \alpha} \bF^\alpha(T)
  \leq Cc \varepsilon^{\beta - \alpha}
  \end{equation}
  Besides,
  \begin{align*}
    \bF^\beta(p_\#P) & \leq \bF^\beta(T) + \bF^\beta(T - p_\# P) \\
    & \leq \bF^\alpha(T) + Cc \varepsilon^{\beta - \alpha} \\
    & \leq Cc
  \end{align*}
  The set
  \[
  X(\varepsilon) = \left\{p_\# P : P \text{ is a polyhedral
  }m\text{-current with } \spt P \subset B(K,C\varepsilon) \text{ and
  } \bF^\beta(p_\#P) \leq Cc \right\}
  \]
  is $\bF^\beta$-compact, being a closed and bounded subset of a
  finite-dimensional space. By~\eqref{eq:distTpP}, every element of
  $B$ is at a distance at most $Cc \varepsilon^{\beta - \alpha}$ from
  an element of $X(\varepsilon)$. By the arbitrariness of
  $\varepsilon$, we conclude that $B$ is $\bF^\beta$-compact.
\end{proof}

The following proposition clarifies several notions of convergences in $\bF^\alpha_m(K)$.

\begin{prop}
  Let $(T_n)$ be a bounded sequence in $\bF^\alpha(K)$, where $0 \leq
  \alpha < 1$. The following are equivalent:
  \begin{itemize}
  \item[(A)] $T_n \to T$ weakly;
  \item[(B)] for all $\beta \in \mathopen{(} \alpha, 1 \mathclose{]}$, one has $\bF^\beta(T_n - T) \to 0$;
  \item[(C)] $\bF(T_n - T) \to 0$.
  \end{itemize}
\end{prop}

\begin{proof}
  Clearly, (B) $\implies$ (C) (by taking $\beta = 1$) and (C)
  $\implies$ (A). Suppose (A). Any subsequence of $(T_{n_k}))$ has an
  $\bF^\beta$-convergent subsequence $(T_{n_{k_\ell}})$ by
  Theorem~\ref{thm:compactness}. Since $\bF^\beta$-convergence is
  stronger than weak convergence, we conclude that the $\bF^\beta$-limit
  of $(T_{n_{k_\ell}})$ is $T$. This implies that $T_n \to T$ in
  $\bF^\beta(K)$.
\end{proof}

\section{Pushforward by Hölder maps}
\label{sec:holderpf}

% gamma(m + alpha) = m + beta
% beta > alpha
% 
% On considere T in F^alpha_m
% (f_n) suite |f - f_n|infty <= C 2^(-n gamma)
%             Lip f_n        <= C 2^(n (1 - gamma))
%
% On prend omega in CH^(m, beta)
% 
% Il faut estimer   (f_(n+1)^# - f_n^#) omega en norme CH^m et CH^(m, beta)
%
% En déduire que ( f_n# T) est bornée dans F^beta_m
% Existence d'une sous-suite convergente ?

%% Van Koch snowflake
%% dimension = ln 4/ln 3 sim 1.262
%% alpha = 0.262
%%
\subsection{Main result}
\label{subsec:mainH}
The following theorem shows that it is possible to make sense of
pushforwards of fractional currents by Hölder maps. In the subsequent
subsection, we will extend this construction in the top-dimensional
case $m = d'$.
\begin{thm}
  \label{thm:holderpush}
  Let $K \subset \R^d$ be a compact set. Let $0 \leq \alpha < \beta <
  1$ and $0 < \gamma < 1$ such that
  \[
  \frac{m + \alpha}{\gamma} = m + \beta.
  \]
  For all $\gamma$-Hölder continuous maps $f \colon K \to \R^{d'}$,
  there is a unique linear map
  \[
  f_\# \colon \bF_m^\alpha(K) \to \bF_m^\beta(\R^{d'})
  \]
  such that
  \begin{itemize}
  \item[(A)] in case $f$ is a Lipschitz map, $f_\#$ is the usual
    pushforward operator;
  \item[(B)] for all $T \in \bF_m^\alpha(K)$, one has
    \begin{gather*}
    \bF^\beta(f_\# T) \leq C \max \left\{ \left(\rmLip^\gamma f
    \right)^{m - 1 + \beta}, \left(\rmLip^\gamma f \right)^{m + \beta}
    \right\} \bF^\alpha(T) \qquad \text{if } m \geq 1 \\
    \bF^\beta(f_\# T) \leq C \max \left\{ 1, \left(\rmLip^\gamma f \right)^{\beta} \right\} \bF^\alpha(T) \qquad \text{if } m = 0,
    \end{gather*}
    where the constant $C = C(d, \alpha, \beta)$;
%  \item[(C)] if $f, g \in \rmLip^\alpha(K; \R^{d'})$ and $\delta \in
%    \mathopen{(} \beta, 1 \mathclose{]}$, then
%      \[
%      \bF^\delta(f_\# T - g_\# T) \leq C \| f - g \|_\infty \max \{
%      \rmLip^\gamma f, \rmLip^\gamma g \} \bF^\alpha(T).
%      \]
  \item[(C)] if $(f_n)$ is a sequence in $\rmLip^\gamma(K; \R^{d'})$
    that converges pointwise to $f$ with $\sup_n \rmLip^\gamma(f_n) <
    \infty$, then $f_{n\#}T \to f_\# T$ in flat norm, and therefore
    weakly.
  \end{itemize}
\end{thm}

This theorem is of course also true in the case $\gamma = 1$ where $f$
is Lipschitz continuous.  In fact, this has already been treated in
Proposition~\ref{prop:basic}(F). In this special case, $\alpha =
\beta$ can be equal to $1$.

The following example shows, however, that the restriction $\beta < 1$
is natural in the purely Hölder case $\gamma < 1$.  Suppose $m=1$,
$\alpha = 0$, and $\gamma \in (1/2, 1)$. Suppose $f_1,f_2$ are two
$\gamma$-Hölder continuous functions $[0, 1] \to \R$, and let
$f=(f_1,f_2)$. By the preceding theorem, the pushforward $f_\#
\llbracket 0,1\rrbracket$ defines a $\left( \frac{1}{\gamma} - 1
\right)$-fractional current.

In particular, we can evaluate this current against the smooth
differential form $x \diff y$, which yields the Young integral
\[
\int_0^1 f_1 \diff f_2 =: f_\#\llbracket 0, 1 \rrbracket (x \diff y).
\]
However, in the critical case $\gamma = 1/2$ (corresponding to $\beta = 1$),
Young's theory breaks down and the above expression cannot, in general,
be given a meaningful interpretation. In stochastic settings, most notably
when $f_2$ is a sample path of Brownian motion and $f_1$ belongs to a suitable
class of processes, the integral can instead be interpreted in the sense of
Itô. This construction, however, is intrinsically probabilistic and is not
determined solely by the sample paths. We refer to \cite{friz-hairer} for a
detailed discussion of how the integral can be defined pathwise in certain cases by enhancing 
$f_2$ to a rough path.
%Let us mention in passing that the results of
%Section~\ref{sec:ext} will allow to recover the Young integral, even
%in the case where \( f_1 \) and \( f_2 \) have different Hölder
%exponents, \(\gamma_1\) and \(\gamma_2\), under the sharp condition
%\(\gamma_1 + \gamma_2 > 1\).

\begin{proof}[Proof of Theorem~\ref{thm:holderpush}]
  Throughout the proof, we will suppose $m \geq 1$. The case $m = 0$
  requires to simplify some of the estimates below. This work is left
  to the reader.
  
  \textbf{Step 1. There exists a sequence $(f_n) \in
    \rmLip^\gamma(\R^d; \R^{d'})$ such that, for all $n$,
    \begin{gather}
      \| f_n - f \|_{\infty, K} \leq C 2^{-n\gamma} \rmLip^\gamma f \label{eq:a1}\\
      \| f_{n+1} - f_n \|_{\infty} \leq C 2^{-n\gamma} \rmLip^\gamma f \label{eq:a2} \\
      \rmLip f_n \leq C 2^{n(1 - \gamma)} \rmLip^\gamma f \label{eq:a3}
    \end{gather}
    where $C = C(d, d')$.  } First extend $f$ to a function $\hat{f}$
  defined on $\R^d$, such that $\|\hat{f}\|_\infty = \|f \|_{\infty,
    K}$ and $\rmLip^\gamma(\hat{f}) \leq C \rmLip^\gamma(f)$ (it is
  possible to choose $C = \sqrt{d'}$). Then we define, for each $n$,
  the function $f_n = \hat{f} * \Phi_{2^{-n}}$. It is clear that
  \[
  \|f_n - \hat{f}\|_\infty \leq 2^{-n\gamma} \rmLip^\gamma( \hat{f})
  \leq C 2^{-n\gamma} \rmLip^\gamma(f).
  \]
  Since $\|f_n - f\|_{\infty, K} \leq \|f_n - \hat{f}\|_{\infty}$,
  this yields~\eqref{eq:a1}. Inequality~\eqref{eq:a2} follows since
  $\|f_{n+1} - f_n\|_{\infty} \leq \| f_{n+1} - \hat{f}\|_\infty +
  \|\hat{f} - f_{n}\|_\infty$. Finally, $f_n$ is smooth and for all $x
  \in \R^d$,
  \begin{align*}
    \nabla f_n(x) & = 2^{-nd} \cdot 2^n \int_{\R^d} \hat{f}(y) \nabla
    \Phi\left(2^n(x - y)\right) \diff y \\ & = 2^{-nd} \cdot 2^n \int_{B(x, 2^{-n})}
    (\hat{f}(y) - \hat{f}(x)) \nabla \Phi\left(2^n(x - y)\right) \diff
    y
  \end{align*}
  This implies that
  \[
  | \nabla f_n(x) | \leq 2^{-nd} 2^n \calL^d(B(x, 2^{-n})) 2^{-n\gamma} \rmLip^\alpha f \| \nabla \Phi \|_\infty
  \]
  and from that one easily infers~\eqref{eq:a3}.

  \textbf{Step 2. Uniqueness.} There exists a sequence of $(f_n)$ in
  $\rmLip^\gamma(K; \R^{d'})$ that converges to $f$ in the sense of
  (C). For instance, we can consider the restrictions to $K$ of the functions
  constructed in Step 1.

  If a pushforward operator $f_\#$ satisfies (A) and (C), then necessarily
  \[
  f_\# T = \lim_{n \to \infty} f_{n\#}T
  \]
  in flat norm, for all $T \in \bF_m^\alpha(K)$. This shows that $f_\#$ is uniquely determined.

  The rest of the proof, steps 3 to 6, is devoted to showing the
  existence of the pushforward operator $f_\#$.

  \textbf{Step 3. For any sequence $(f_n)$ in $\rmLip^\gamma(\R^d;
    \R^{d'})$ such that \eqref{eq:a1}, \eqref{eq:a2} and \eqref{eq:a3}
    hold for all~$n$, the sequence $(f_{n\#} T)$ converges in
    $\bF^\delta$ norm, for every $\beta < \delta \leq 1$, to a current
    denoted $f_\# T$.}  Before we prove this claim, let us make two
  remarks.
  \begin{itemize}
  \item In fact, we will prove that the stronger result that series
    $\sum_n \bF^\delta((f_{n+1})_\# T - f_{n\#} T)$ converges
    $\bF^\delta$-absolutely. Since all the currents $f_{n\#}T$ are
    supported in some common compact set, bigger than $f(K)$, this
    will conclude.
  \item By the standard interlacing argument, the limit $f_\# T$ does
    not depend on the choice of the approximating sequence $(f_n)$.
  \end{itemize}

  Let $(T_k)$ be an $\alpha$-fractional decomposition of $T$. Set $U_n
  = (f_{n+1})_\# T - f_{n\#} T$ and $U_{n,k} = (f_{n+1})_\# T_k -
  f_{n\#} T_k$. We first estimate
  \begin{align*}
    \bF(U_{n,k}) & \leq 2 \max \{(\rmLip f_n)^m, (\rmLip f_n)^{m+1},
    (\rmLip f_{n+1})^m, (\rmLip f_{n+1})^{m+1}\} \bF(T_k) \\ & \leq C
    2^{n(1-\gamma)(m+1)} \bF(T_k).
  \end{align*}
  Note that in this step (only), the constant $C$ depends on $f$ (via
  $\rmLip^\gamma f$), but not on $n$.  We can control the flat norm of
  $U_{n,k}$ in a different manner, using Proposition~\ref{prop:22}(D).
  \begin{align*}
  \bF(U_{n,k}) & \leq \|f_{n+1} - f_n\|_\infty \max \{(\rmLip f_n)^m,
  (\rmLip f_n)^{m-1}, (\rmLip f_{n+1})^m, (\rmLip f_{n+1})^{m-1} \}
  \bN(T_k) \\ & \leq C 2^{-n \gamma} 2^{n(1-\gamma)m} \bN(T_k).
  \end{align*}
  Finally, the estimate for the normal mass of $U_{n,k}$ is
  \begin{align*}
    \bN(U_{n,k}) & \leq 2 \max \{(\rmLip f_n)^m, (\rmLip f_n)^{m-1},
    (\rmLip f_{n+1})^m, (\rmLip f_{n+1})^{m-1}\} \bN(T_k) \\ & \leq C
    2^{n(1 - \gamma)m} \bN(T_k).
  \end{align*}
  Those three estimates imply that
  \begin{align*}
    \bF(U_{n,k})^\delta \bN(U_{n,k})^{1 - \delta} & \leq
    \bF(U_{n,k})^\alpha \bF(U_{n,k})^{\delta - \alpha} \bN(U_{n,k})^{1
      - \delta} \\ & \leq C 2^{n(1-\gamma) \left((m+1) \alpha +
      m(\delta - \alpha) + m (1-\delta) \right)} 2^{-n\gamma(\delta - \alpha)}
    \bF(T_k)^\alpha \bN(T_k)^{1-\alpha} \\
    & \leq C 2^{-n\gamma(\delta - \beta)} \bF(T_k)^\alpha \bN(T_k)^{1-\alpha}.
  \end{align*}
  After summing over $k$, one obtains
  \[
  \bF^\delta(U_n) \leq C 2^{-n\gamma(\delta - \beta)}
  \sum_{k=0}^\infty \bF(T_k)^\alpha \bN(T_k)^{1-\alpha}.
  \]
  We finally obtain that $\sum_{n=0}^\infty \bF^\delta(U_n) < \infty$,
  as the geometric series $\sum_{n=0}^\infty 2^{-n\gamma(\delta -
    \beta)}$ converges.

  \textbf{Step 4. Proof of (A).} It suffices to extend $f$ to a Hölder
  map $\R^d \to \R^{d'}$ and then choose the approximating sequence
  $f_n = f$.

  \textbf{Step 5. Proof of (B).} Let $(f_n)$ be a sequence that
  satisfies~\eqref{eq:a1}, \eqref{eq:a2}, \eqref{eq:a3}. Let $K'$ be
  the compact convex hull of $K$ and $L$ be a compact set that
  contains $\bigcup_{n=0}^\infty f_n(K')$. We will show that $(f_{n\#}
  T)$ is bounded in $\bF^\beta(L)$. Because of the compactness theorem~\ref{thm:compactness}
  (recall $\beta < 1$) and step~3, we will then deduce that $f_{\#} T$
  is in $\bF^\beta(L)$ and $\bF^\beta(f_\# T) \leq \sup_n
  \bF^\beta(f_{n\#} T)$.

  For technical reasons, we will first suppose that $\rmLip^\gamma f
  \geq 1$.  Let $\omega \in \bCH^{m, \beta}(L)$ such that $\| \omega
  \|_{\bCH^{m,\beta}(L)} = 1$. Considering each Lipschitz function
  $f_n$ restricted to $K' \to L$, we can define the sequence of
  charges $\omega_n = f_{n+1}^\# \omega - f_{n}^\# \omega \in
  \bCH^{m,\beta}(K')$ for all $n \geq 0$.  Our goal is to apply
  Lemma~\ref{lemma:LPcharge}.

  For any $n \geq 0$ and $S \in \bN_m(K')$, one has
  \begin{equation}
    \label{eq:omeganS}
  |\omega_n(S)| \leq \bN((f_{n+1})_\#S - f_{n\#} S)^{1 - \beta}
  \bF((f_{n+1})_\#S - f_{n\#} S)^{\beta}
  \end{equation}
  As in step 3, we use three sorts of estimates. Specifically,
  \begin{align}
    \bN((f_{n+1})_\#S - f_{n\#} S) & \leq 2 \max \{(\rmLip f_n)^m,
    (\rmLip f_n)^{m-1}, (\rmLip f_{n+1})^m, (\rmLip f_{n+1})^{m-1}\}
    \bN(S) \notag \\ &\leq C 2^{n(1 - \gamma)m} (\rmLip^\gamma f)^{m} \bN(S) \label{eq:Hstep3.1}
  \end{align}
  (we used that $\rmLip^\gamma f \geq 1$, so that~\eqref{eq:Hstep3.1}
  holds for all $n \geq 0$) and
  \begin{align}
  \bF((f_{n+1})_\#S - f_{n\#} S) & \leq 2 \max \{(\rmLip f_n)^m,
  (\rmLip f_n)^{m+1}, (\rmLip f_{n+1})^m, (\rmLip f_{n+1})^{m+1}\}
  \bF(S) \notag \\ &\leq C 2^{n(1 - \gamma)(m+1)} (\rmLip^\gamma
  f)^{m+1} \bF(S) \label{eq:Hstep3.2}
  \end{align}
  and by Proposition~\ref{prop:22}(D),
  \begin{align}
    \bF((f_{n+1})_\#S - f_{n\#} S) & \leq \| f_{n+1} - f_n\|_\infty
    \max \{(\rmLip f_n)^m, (\rmLip f_n)^{m-1}, (\rmLip f_{n+1})^m,
    (\rmLip f_{n+1})^{m-1}\} \bN(S) \notag \\ & \leq C 2^{-n\gamma}
    2^{n(1 - \gamma)m} (\rmLip^\gamma f)^{m+1}
    \bN(S) \label{eq:Hstep3.3}
  \end{align}
  We deduce from~\eqref{eq:omeganS}, \eqref{eq:Hstep3.1} and
  \eqref{eq:Hstep3.3} that
  \begin{align*}
  |\omega_n(S)| & \leq C 2^{-n \gamma \beta} 2^{n (1 - \gamma) m} (\rmLip^\gamma f)^{m+\beta}
  \bN(S) \\
  & = C 2^{-n \alpha} (\rmLip^\gamma f)^{m+\beta}
  \bN(S)
  \end{align*}
  Therefore, $\|\omega_n\|_{\bCH^m(K')} \leq C2^{-n\alpha}
  (\rmLip^\gamma f)^{m+\beta}$. In addition, we infer
  from~\eqref{eq:Hstep3.1} and \eqref{eq:Hstep3.2} that
  \begin{align*}
    |\omega_n(S)| & \leq C 2^{n (1 - \gamma)(m+\beta)} (\rmLip^\gamma
    f)^{m+\beta} \bN(S)^{1 - \beta} \bF(S)^\beta \\ & = C 2^{n(\beta -
      \alpha) } (\rmLip^\gamma f)^{m+\beta} \bN(S)^{1 - \beta}
    \bF(S)^\beta,
  \end{align*}
  thereby showing that $\| \omega_n \|_{\bCH^{m, \beta}(K')} \leq C
  2^{n(\beta - \alpha) } (\rmLip^\gamma f)^{m+\beta}$.  By
  Lemma~\ref{lemma:LPcharge}, applied for each $n \geq 1$ to the
  truncated sequence of charges $(\omega_0, \omega_1, \dots,
  \omega_{n-1}, 0, \dots)$, we have $\sup_n \| f_n^\# \omega - f_0^\#
  \omega \|_{\bCH^{m, \alpha}(K')} \leq C (\rmLip^\gamma f)^{m+\beta}
  $.

  Besides, the formula
  \[
  \langle f_n^\# \omega - f_0^\#
  \omega, S \rangle = \langle \omega, f_{n\#} S - f_{0\#} S \rangle
  \]
  that holds for each $S \in \bN_m(K')$, extends by density to all $S
  \in \bF^\beta(K')$ (see Proposition~\ref{prop:basic}(E)(F), which
  makes use of the convexity of $K'$). Applied to $T$ (that belongs to
  $\bF_m^\alpha(K) \subset \bF_m^\beta(K')$), we obtain
  \[
  |\langle \omega, f_{n\#} T - f_{0\#}T \rangle | = |\langle f_n^\#
  \omega - f_0^\# \omega, T \rangle| \leq C (\rmLip^\gamma
  f)^{m+\beta} \bF^\alpha(T).
  \]
  Taking the supremum over all $\omega$ in the closed unit ball of
  $\bCH^{m, \beta}(K')$ results in
  \[
  \bF^\beta( f_{n\#} T - f_{0\#}T ) \leq C (\rmLip^\gamma f)^{m+\beta}
  \bF^\alpha(T)
  \]
  according to Theorem~\ref{thm:duality}.  Moreover,
  \begin{align*}
  \bF^\beta( f_{0\#} T) & \leq \max \{ (\rmLip f_0)^{m-1+\beta}, (\rmLip f)^{m+\beta} \}
  \bF^\beta(T) \\
  & \leq C (\rmLip^\gamma f)^{m+\beta} \bF^\beta(T) \\
  & \leq C (\rmLip^\gamma f)^{m+\beta} \bF^\alpha(T)
  \end{align*}
  by Proposition~\ref{prop:22} and~\eqref{eq:a3}. This proves that
  \[
  \bF^\beta(f_{n \#} T) \leq C (\rmLip^\gamma f)^{m+\beta}
  \bF^\alpha(T)
  \]
  and finally
  \begin{equation}
    \label{eq:finally}
    \bF^\beta(f_{\#} T) \leq C (\rmLip^\gamma f)^{m+\beta}
    \bF^\alpha(T).
  \end{equation}

  Next we treat the case where $0 < \rmLip^\gamma f \leq 1$. We let $r
  := (\rmLip^\gamma f)^{-1}$ and $\varphi_r \colon \R^{d'} \to
  \R^{d'}$ be the map $x \mapsto rx$. It is easy to prove (from the
  definition given in Step 3) that $(\varphi_r \circ f)_\# T =
  \varphi_{r\#} f_\# T$. We then make use of the inequalities
  \[
  \bF^\beta(f_\#T) \leq \frac{1}{r^{m-1+\beta}} \bF^\beta\left(
  (\varphi_r \circ f)_\# T \right)
  \]
  (which comes from Proposition~\ref{prop:basic}(F) applied to
  $\varphi_r^{-1}$) and
  \[
  \bF^\beta\left(
  (\varphi_r \circ f)_\# T \right) \leq C \bF^\alpha(T) 
  \]
  by~\eqref{eq:finally}. Thus
  \[
  \bF^\beta(f_{\#} T) \leq C (\rmLip^\gamma f)^{m-1+\beta}
  \bF^\alpha(T).
  \]

  \textbf{Step 6. Proof of (C).} Let $R \geq 0$ and $f, g \in
  \rmLip^\gamma(K; \R^{d'})$ be two Hölder maps with Hölder constants
  $\max\{ \rmLip^\gamma f, \rmLip^\gamma g\} \leq R$. Extend $f, g$ to
  $\R^d$ such that
  \[
  \max\{ \rmLip^\gamma(f; \R^d), \rmLip^\gamma(g; \R^d)\} \leq CR
  \text{ and } \|f - g\|_\infty = \|f - g\|_{\infty, K}
  \]
  (To force the second condition, one can first find Hölder extensions
  $\hat{f}, \hat{g}$ to $\R^d$ and then, if necessary, compose
  $(\hat{f}, \hat{g})$ with the orthogonal projection onto $\{(x, y)
  \in \R^{2d'} : |x - y| \leq \|f - g\|_{\infty, K}\}$). Our goal is
  to prove~\eqref{eq:goal}, from which (C) is an easy consequence.

  Call $f_n = f
  * \Phi_{2^{-n}}$ and $g_n = g * \Phi_{2^{-n}}$, so that
  \begin{gather}
    \|f_n - g_n\|_\infty \leq \|f - g\|_{\infty, K} \label{eq:b1} \\ \max\{\|f_{n+1}
    - f_n\|_{\infty}, \|g_{n+1} - g_n\|_\infty \} \leq L 2^{-n\gamma}
    \label{eq:b2} \\ \max\{\rmLip f_n, \rmLip g_n \} \leq L 2^{n(1-\gamma)} \label{eq:b3}
  \end{gather}
  with $L = \max\{1, CR\}$.

  Let $(T_k)$ be an $\alpha$-fractional decomposition of $T$. Set, for
  all possible $n$ and $k$,
  \[
  U_{n,k} = (f_{n+1})_\#T_k - (g_{n+1})_\#T_k - (f_{n\#}T_k - g_{n\#}T_k)
  \]
  There are three possible ways of estimating the flat norm of
  $U_{n,k}$. The first one follows from an application of
  Proposition~\ref{prop:22}(D).
  \begin{align*}
    \bF(U_{n,k}) & \leq \bF((f_{n+1})_\#T_k - (g_{n+1})_\#T_k) + \bF(f_{n\#}T_k - g_{n\#}T_k) \\
    & \leq CL^m \| f - g\|_{\infty, K} 2^{n(1-\gamma)m} \bN(T_k) 
  \end{align*}
  A second use of Proposition~\ref{prop:22}(D) yields
  \begin{align*}
    \bF(U_{n,k}) & \leq \bF((f_{n+1})_\#T_k - f_{n\#} T_k) +
    \bF((g_{n+1})_\# T_k - g_{n\#}T_k) \\ & \leq C L^{m+1}
    2^{-n\gamma} 2^{n(1-\gamma)m} \bN(T_k)
  \end{align*}
  and finally
  \begin{align*}
    \bF(U_{n,k}) & \leq \bF((f_{n+1})_\#T_k) + \bF((g_{n+1})_\#T_k) +
    \bF(f_{n\#}T_k) + \bF(g_{n\#}T_k) \\ & \leq C L^{m+1}
    2^{n(1-\gamma)(m+1)} \bF(T_k)
  \end{align*}
  This leads to
  \begin{align*}
    \bF(U_{n,k}) & = \bF(U_{n,k})^{(1-\beta)/2}
    \bF(U_{n,k})^{(1+\beta)/2 - \alpha} \bF(U_{n,k})^{\alpha} \\ &
    \leq C L^{m+(1+\beta)/2} \|f - g\|_{\infty, K}^{(1 - \beta)/2}
    2^{-n \gamma(1 - \beta)/2} \bN(T_k)^{1 - \alpha} \bF(T_k)^\alpha
  \end{align*}
  By summing over $k$, one obtains
  \[
  \bF( (f_{n+1})_\#T - (g_{n+1})_\#T - (f_{n\#}T - g_{n\#}T) )
  \leq C L^{m+(1+\beta)/2} \|f - g\|_{\infty, K}^{(1 - \beta)/2} 2^{-n
    \gamma(1 - \beta)/2} \bF^\alpha(T)
  \]
  Summing over $n$ (and recalling that $f_{n\#}T$ and $g_{n\#}T$
  converge to $f_\#T$ and $g_\#T$ in flat norm by Step 3) yields
  \begin{equation}
    \label{eq:almostthere}
  \bF(f_\# T - g_\# T - (f_{0\#}T - g_{0\#}T)) \leq C
  L^{m+(1+\beta)/2} \|f - g\|_{\infty, K}^{(1 - \beta)/2}
  \bF^\alpha(T).
  \end{equation}
  It remains to estimate $\bF(f_{0\#}T - g_{0\#}T)$. Again, we use
  Proposition~\ref{prop:22}(D) and~\eqref{eq:b1}, \eqref{eq:b3} to
  obtain that
  \[
  \bF(f_{0\#}T_k - g_{0\#}T_k) \leq L^m \|f - g\|_{\infty, K} \bN(T_k). 
  \]
  Also,
  \[
  \bF(f_{0\#}T_k - g_{0\#}T_k) \leq 2 L^{m+1} \bF(T_k).
  \]
  Hence
  \[
  \bF(f_{0\#}T_k - g_{0\#}T_k) \leq C L^{m+\alpha} \| f - g\|_{\infty,
    K}^{1 - \alpha} \bN(T_k)^{1-\alpha} \bF(T_k)^\alpha.
  \]
  By summing over $k$, one obtains
  \[
  \bF(f_{0\#}T - g_{0\#}T) \leq C L^{m+\alpha} \| f - g\|_{\infty,
    K}^{1 - \alpha} \bF^\alpha(T).
  \]
  Together with~\eqref{eq:almostthere}, we deduce that
  \begin{equation}
    \label{eq:goal}
  \bF(f_\# T - g_\# T) \leq C L^{m+(1+\beta)/2} \|f - g\|_{\infty,
    K}^{(1 - \beta)/2} \bF^\alpha(T) + C L^{m+\alpha} \| f -
  g\|_{\infty, K}^{1 - \alpha} \bF^\alpha(T).
  \end{equation}
  This clearly implies (C). Indeed, if $(f_n)$ is a sequence of Hölder
  maps as in (C), one applies~\eqref{eq:goal} to $g = f_n$, hence
  $f_{n\#}T \to f_\#T$ in flat norm.
\end{proof}
The basic properties of the pushforward operator are summarized in the
following proposition.
\begin{prop}
  \label{prop:basicHPF}
  Let $K, d, d', \alpha, \beta, f, \gamma$ as in
  Theorem~\ref{thm:holderpush} and $T \in \bF_m^\alpha(K)$.
  \begin{itemize}
  \item[(A)] $\spt f_\# T \subset f(\spt T)$.
  \item[(B)] $f_\# \partial T = \partial f_\# T$ for $m \geq 1$.
  \item[(C)] Let $L
  \subset \R^{d'}$ be a compact set such that $f(K) \subset L$ and $g
  \colon L \to \R^{d''}$ a $\gamma'$-Hölder continuous map, where $0 <
  \gamma' < 1$. We suppose that $\beta' \in (\beta, 1)$ is such that
  \[
  \frac{m+\alpha}{\gamma \gamma'} = \frac{m+\beta}{\gamma'} = m +
  \beta'.
  \]
  Then $(g \circ f)_\# T = g_\# f_\# T$.
  \item[(D)] The current $f_\#T$ depends only on $T$ and the
    restriction of $f$ to $\spt T$.
\end{itemize}
\end{prop}

\begin{proof}
  (A). Let $(f_n)$ be a sequence of Lipschitz maps $K \to \R^{d'}$
  that converge pointwise to $f$ with equi-bounded Hölder
  constants. Then, it is clear that, for all $n_0$ and $n \geq n_0$,
  one has
  \[
  \spt f_{n\#} T \subset \bigcup_{n \geq n_0} B\left(f(\spt T), \|f -
  f_n\|_{\infty, K}\right).
  \]
  As $f_{n\#} T$ converges weakly to $f_\#T$ by
  Theorem~\ref{thm:holderpush}(C), we deduce that
  \[
  \spt f_\#T \subset \bigcap_{n_0 \geq 0} \bigcup_{n \geq n_0}
  B\left(f(\spt T), \|f - f_n\|_{\infty, K}\right) = f(\spt T).
  \]

  (B). First we observe that the left-hand side $f_\# \partial T$
  makes sense, as $\partial T$ is an $\alpha$-fractional
  $(m-1)$-current and
  \[
  \frac{m-1 + \alpha}{\gamma} = m + \beta - \frac{1}{\gamma} < m + 1 - 1 = m.
  \]
  The equality $f_\# \partial T = \partial f_\# T$ is obviously true
  if $f$ is Lipschitz continuous, and the Hölder case is obtained by a
  density argument.

  We next prove (C) in the special (unmentioned) case $\gamma = 1$ and
  $\beta = \alpha$ where $f$ is Lipschitz continuous. Let $(g_n)$ be a
  sequence in $\rmLip^{\gamma'}(L; \R^{d''})$ that converges pointwise
  to $g$ with equi-bounded Hölder constants. Then $(g_n \circ f)$ is a
  sequence in $\rmLip^{\gamma'}(K; \R^{d''})$ that converges pointwise
  to $g \circ f$ with equi-bounded Hölder constants. Clearly, $(g_n
  \circ f)_\# T = g_{n\#} f_\# T$ for all $n$, and we obtain our first
  result by passing to the weak limit.

  (D). Let $f_1, f_2$ be two Hölder functions that agree on $\spt
  T$. Then $f_1 \circ \iota = f_2 \circ \iota$, where $\iota \colon
  \spt K \to \R^d$ is the injection map. By~(C), one has $f_{1\#} T =
  f_{2\#} T$.

  The general case of (C) is treated by considering a sequence $(f_n)$
  in $\rmLip^\gamma(K; \R^d)$ that converges pointwise to $f$ with
  equi-bounded Hölder constant, extending the map $g$ to $\R^{d'}$ in a
  Hölder manner, and then taking the limit in $(g \circ f_n)_\# T =
  g_\# f_{n\#} T$.
\end{proof}

\subsection{Top dimensional case}
Let us first recall the cone construction. If $K$ is a compact convex
subset of $\R^d$, $a \in K$ and $T \in \bF_m(K)$ is a flat
$m$-current, then $a \cone T$ is the flat $(m+1)$-current
\begin{equation}
  \label{eq:coneConstr}
a \cone T = h_\# \left( \llbracket 0, 1 \rrbracket \times T \right)
\end{equation}
where $h(t, x) = (1-t)a + tx$.  Due to the convexity of $K$, the cone
remains supported in $K$. Moreover, the formula~\eqref{eq:coneConstr}
makes it clear that $T \mapsto a \cone T$ is a continuous linear map
$\bF_m(K) \to \bF_{m+1}(K)$. For further details,
see~\cite{Fede}[4.1.11], which establishes that
\begin{equation}
  \label{eq:coneMass}
\bM(a \cone T) \leq C (\operatorname{diam} K) \bM(T), \text{ and }
\partial (a \cone T) = T - a \cone \partial T,
\end{equation}
where $C = C(d)$ is a constant, and
\begin{equation}
  \label{eq:coneBord}
  \partial (a \cone T) = \begin{cases}
    T - a \cone \partial T & \text{if } m \geq 1 \\
    T & \text{if } m = 0
  \end{cases}
\end{equation}
From~\eqref{eq:coneMass} and~\eqref{eq:coneBord}, one easily infers
that
\[
\bF(a \cone T) \leq C(1 + \operatorname{diam} K) \bF(T), \qquad
\bN(a \cone T) \leq C(1 + \operatorname{diam} K) \bN(T).
\]
Now, if $T$ is $\alpha$-fractional, then so is $a \cone T$. Indeed,
one easily shows, using the preceding estimates, that any
$\alpha$-fractional decomposition $(T_k)$ of $T$ induces an
$\alpha$-fractional decomposition $(a \cone T_k)$ of $a \cone
T$. Moreover,
\[
\bF^\alpha(a \cone T) \leq C(1 + \operatorname{diam} K) \bF^\alpha(T).
\]
This construction enables to extend the pushforward defined in
Subsection~\ref{subsec:mainH} in the following top-dimensional
case. We let $d \geq d' \geq 1$ and $m = d'$ ; $K \subset \R^d$ is a
compact subset and $0 \leq \alpha < \beta < 1$ and $0 < \gamma < 1$
are such that
\[
\frac{d'-1 + \alpha}{\gamma} = d' - 1 + \beta.
\]
The pushforward $f_\# T$ is defined to be the unique $d'$-current such
that $\partial( f_\# T) = f_\# \partial T$. The uniqueness of $f_\# T$
is ensured by the constancy theorem, whereas its existence, and the
fact that $f_\# T \in \bF_m^\beta(\R^{d'})$ is guaranteed by the cone
construction above. We let the reader prove that properties (A), (C)
of Theorem~\ref{thm:holderpush} and (A), (B), (C), (D) of
Proposition~\ref{prop:basicHPF} still hold in this case.

In the even more special case where $d=d' = m$, this is related to the
Brouwer degree.

\begin{prop}
  \label{prop:83}
  Let $U \subset \R^d$ be a bounded set. Suppose that $U$ has finite
  $(1-\alpha)$-perimeter (where $0 \leq \alpha < 1$) and $f \colon
  \operatorname{cl} U \to \R^d$ be a $\gamma$-Hölder map with $0 <
  \gamma < 1$ and
  \begin{equation}
    \label{eq:condB}
    \frac{d-1+\alpha}{\gamma} = d - 1 + \beta < d.
  \end{equation}
  Then the density function of the flat current $f_\# \llbracket
  U \rrbracket$ agrees with $\operatorname{deg}(f, U, \cdot)$ almost
  everywhere on $\R^d \setminus f(\partial U)$. In particular, if
  $f(\partial U)$ is Lebesgue negligible, then $\operatorname{deg}(f,
  U, \cdot)$ belongs to $W^{1-\beta,1}_c(\R^d)$.
\end{prop}

\begin{proof}
  This proof is almost identical to the one
  of~\cite[Lemma~5.6]{Zust}. If $\varphi$ is the restriction of a
  smooth map $\R^d \to \R^d$, then the density function $v$ of
  $f_\# T$ is given by
  \[
  v(y) = \sum_{x \in f^{-1}(x) \cap U} \operatorname{sign} \det (D
  f_x))
  \]
  for almost all $y \in \R^d$, by~\cite[4.1.26]{Fede}. This agrees
  with $\operatorname{deg}(f, U, y)$ for almost all $y \in \R^d
  \setminus f(\partial U)$ by Sard's theorem.

  For the general case, we extend $f$ to $\R^d$ in a Hölder manner and
  consider the sequence $f_n = f * \Phi_{1/n}$, that is smooth and
  converges pointwise to $f$, with equi-bounded Hölder constants. Then
  $\partial f_{n\#} \llbracket U \rrbracket = f_{n \#} \partial
  \llbracket U \rrbracket \to f_\# \partial \llbracket U \rrbracket$
  in flat norm by Theorem~\ref{thm:holderpush}(C). As the map $a \cone
  \cdot$ is $\bF$-continuous, the sequence $(f_{n\#} \llbracket U
  \rrbracket)$ converges to $f_\# \llbracket U \rrbracket$ in flat
  norm. Let us denote by $v_n$ (resp. $v$) the density function of
  $f_{n\#} \llbracket U \rrbracket$ (resp. $f_{\#} \llbracket U
  \rrbracket$). Then $v_n \to v$ in $L^1$.

  Let $H_n \colon [0, 1] \times \operatorname{cl} U \to \R^d$ the
  homotopy $H(t, x) = (1-t)f(x) + t f_n(x)$ between $f$ and
  $f_n$. From the homotopy invariance of the Brouwer degree
  (see~\cite[Chapter~IV, Proposition~2.6]{OuteRuiz}), we have
  \[
  \operatorname{deg} (f, U, y) = \operatorname{deg} (f_n, U, y)
  \]
  for all $y \in \R^d$ such that $y \not\in H(t, \partial U)$ for all
  $t \in [0, 1]$. In particular, $\operatorname{deg}(f, U, \cdot)$ and
  $\operatorname{deg}(f_n, U, \cdot)$ agree on $\R^d \setminus B(
  \partial U, \| f - f_n\|_{\infty, \partial U})$, and that
  \[
  \operatorname{deg} (f, U, \cdot) \ind_{\R^d \setminus B( \partial U,
    \| f - f_n\|_{\infty, \partial U})} = v_n \ind_{\R^d \setminus B(
    \partial U, \| f - f_n\|_{\infty, \partial U})}
  \]
  By letting $n \to \infty$, we establish that $\operatorname{deg}(f,
  U, \cdot) = v$ almost everywhere on $\R^d \setminus f(\partial U)$.
\end{proof}
The paper~\cite{Klin} suggests that the condition $\calL^d(f(\partial
U)) = 0$, that appears in Proposition~\ref{prop:83} should follow
automatically from the assumptions. We were however not able to find
an elementary proof.

We point the reader towards \cite[Theorem~1.1]{DiPiLippSporVald} for a
proof of the fractional Sobolev embeddings
\[
W_c^{1- \beta, 1}(\R^d) \subset W_c^{s, d/(d-s+1-\beta)}(\R^d) \subset
L^{d/(d+1- \beta)}_c(\R^d) \text{ for all } s \in (0, 1 - \beta).
\]
We also recall from Subsection~\ref{subsec:firstexample} for
$\llbracket U \rrbracket$ is $\alpha$-fractional for all $\alpha >
\dim_{\mathrm{box}}(\partial U) - (d - 1)$. Those facts imply that the
results of De Lellis and Inauen~\cite{DeLeInau} mentioned in the
introduction are a particular case of our Proposition~\ref{prop:83}.

Higher integrability of the Brouwer degree was previously established
by Olbermann in~\cite{Olbe}, who showed that
\[
\operatorname{deg}(f, U, \cdot) \in L^p_c(\R^d) \quad \text{for all }
1 \leq p < \frac{d}{d-1+\beta},
\] 
provided that \( U \) is a bounded open set with
\(\dim_{\mathrm{box}}(\partial U) = d-1 + \alpha\), where \( f \) is
\(\gamma\)-Hölder and satisfies~\eqref{eq:condB}.  Our results offer a
slight improvement, as we extend this integrability to the critical
case when the condition on the upper box dimension of \( \partial U \)
is replaced by the finiteness of the fractional perimeter of \( U \).

Additionally, we mention another higher integrability result due to
Züst, not covered by Proposition~\ref{prop:83}, which addresses the
case of a map \( f = (f_1, \dots, f_d) \), where the components are
Hölder continuous with potentially different Hölder exponents.

\section{Wedge product of fractional charges}
\label{sec:ext}

\begin{thm}
  \label{thm:main}
  Let $\alpha, \beta$ be parameters such that $0 < \alpha, \beta \leq
  1$ and $\alpha + \beta > 1$. There is a unique map
  \[
  \wedge \colon \bCH^{m, \alpha}(\R^d) \times \bCH^{m', \beta}(\R^d)
  \to \bCH^{m+m', \alpha + \beta - 1}(\R^d)
  \]
  such that
  \begin{enumerate}
  \item[(A)] $\wedge$ extends the pointwise wedge product between smooth forms;
  \item[(B)] Weak*-to-weak* continuity: if $(\omega_n)$ and $(\eta_n)$
    are two sequences that converge weakly-* to $\omega$ and $\eta$ in
    $\bCH^{m, \alpha}(\R^d)$ and $\bCH^{m', \beta}(\R^d)$
    respectively, then $\omega_n \wedge \eta_n$ converge weakly-* to
    $\omega \wedge \eta$ in $\bCH^{m+m', \alpha + \beta - 1}(\R^d)$.
  \end{enumerate}
  Moreover, $\wedge$ is continuous.
\end{thm}

\begin{proof}
  We proved in Proposition~\ref{prop:estimatesSmoothing} that $\omega
  * \Phi_\varepsilon \to \omega$ weakly-* in $\bCH^{m, \alpha}(\R^d)$,
  and similarly, $\eta * \Phi_\varepsilon \to \eta$ weakly-* in
  $\bCH^{m', \beta}(\R^d)$. Taking into account that $\omega *
  \Phi_\varepsilon$ and $\eta * \Phi_\varepsilon$ are smooth by
  Proposition~\ref{prop:smooth}, the uniqueness of the map $\wedge$
  follows.

  We now address the issue of existence. First we treat the case
  $(\alpha, \beta) \neq (1, 1)$. Abbreviate $\omega_n = \omega *
  \Phi_{2^{-n}}$ and $\eta_n = \eta * \Phi_{2^{-n}}$. We shall prove
  that the weak* limit of $(\omega_n \wedge \eta_n)$ exists in
  $\bCH^{m+m', \alpha + \beta - 1}(\R^d)$. This will be achieved if we
  manage to prove that the following series converges weakly-*
  \begin{equation}
    \label{eq:seriesW*}
    \omega_0 \wedge \eta_0 + \sum_{n=0}^\infty \left( \omega_{n+1}
    \wedge (\eta_{n+1} - \eta_n) + (\omega_{n+1} - \omega_n) \wedge
    \eta_n \right)
  \end{equation}
  in $\bCH^{m+m', \alpha + \beta - 1}(\R^d)$. Let us note in passing
  that both sums $\sum_{n=0}^\infty \omega_{n+1} \wedge (\eta_{n+1} -
  \eta_n)$ and $\sum_{n=0}^\infty (\omega_{n+1} - \omega_n) \wedge
  \eta_n$ can be interpreted as paraproducts.  The weak* convergence
  will follow from an application of Corollary~\ref{cor:LP}. To this
  end, we let $K$ be a closed (non degenerate) ball and estimate
  \begin{align*}
  \| \eta_{n+1} - \eta_n \|_{\bCH^{m'},K} & \leq \| \eta_{n+1} - \eta
  \|_{\bCH^{m'},K} + \| \eta_n - \eta \|_{\bCH^{m'},K} \\ & \leq C 2^{-n\beta}
  \| \eta \|_{\bCH^{m', \beta}, B(K, 1)}
  \end{align*}
  by Proposition~\ref{prop:estimatesSmoothing} and similarly
  \[
  \| \omega_{n+1} - \omega_n \|_{\bCH^m, K} \leq C 2^{-n\alpha}
  \|\omega\|_{\bCH^{m, \alpha}, B(K, 1)}.
  \]
  Next we wish to control the $\bCH^{m+m'}(K)$ seminorm of the
  wedge product $\omega_{n+1} \wedge (\eta_{n+1} - \eta_n)$. Let $T
  \in \bN_{m+m'}(K)$. We recall that
  \[
  \partial (T \hel \omega_{n+1}) = (-1)^m (\partial T) \hel
  \omega_{n+1} + (-1)^{m+1} T \hel \diff \omega_{n+1}.
  \]
  Hence
  \begin{align*}
    \bN(T \hel \omega_{n+1}) & \leq C \bN(T) \max \{
    \|\omega_{n+1}\|_{\infty, K}, \| \diff \omega_{n+1} \|_{\infty, K}
    \} \\ & \leq C \bN(T) \| \omega_{n+1} \|_{\bCH^{m, 1}, K} &
    \text{by Proposition~\ref{prop:CH1smooth}} \\ & \leq C \bN(T) \|
    \omega \|_{\bCH^{m, \alpha}, B(K, 1)} 2^{n(1 - \alpha)} & \text{by
      Proposition~\ref{prop:estimatesSmoothing}(B)}
  \end{align*}
  We deduce that
  \begin{align*}
    \left|\omega_{n+1} \wedge (\eta_{n+1} - \eta_n) (T)\right| & =
    \left| (\eta_{n+1} - \eta_n)(T \hel \omega_{n+1}) \right| \\ &
    \leq \| \eta_{n+1} - \eta_n \|_{\bCH^{m'}, K} \bN(T \hel
    \omega_{n+1}) \\ & \leq C \|\omega\|_{\bCH^{m,\alpha}, B(K, 1)} \|
    \eta \|_{\bCH^{m',\beta}, B(K,1)} 2^{n(1 - \alpha - \beta)} \bN(T).
  \end{align*}
  As a result,
  \[
  \|\omega_{n+1} \wedge (\eta_{n+1} - \eta_n) \|_{\bCH^m, K} \leq C
  \|\omega\|_{\bCH^{m,\alpha}, B(K,1)} \| \eta \|_{\bCH^{m', \beta}, B(K,1)}
  2^{n(1 - \alpha - \beta)}.
  \]
  Next we estimate the $\| \cdot \|_{\bCH^{m,1},K}$ seminorm of $\omega_{n+1}
  \wedge (\eta_{n+1} - \eta_n)$. By
  Proposition~\ref{prop:estimatesSmoothing}(B),
  \[
  \|\omega_{n+1}\|_{\bCH^{m,1}, K} \leq C 2^{n(1 - \alpha)} \| \omega
  \|_{\bCH^{m, \alpha}, B(K,1)}
  \]
  and
  \begin{align*}
  \| \eta_{n+1} - \eta_n \|_{\bCH^{m', 1}, K} & \leq
  \|\eta_{n+1}\|_{\bCH^{m', 1}, B(K,1)} + \|\eta_n \|_{\bCH^{m', 1},K}
  \\ & \leq C 2^{n(1 - \beta)} \| \eta \|_{\bCH^{m', \beta}, B(K,1)}
  \end{align*}
  By Corollary~\ref{cor:411},
  \[
  \|\omega_{n+1} \wedge (\eta_{n+1} - \eta_n) \|_{\bCH^{m,1},K} \leq
  C \| \omega \|_{\bCH^{m,\alpha}, B(K,1)} \|\eta\|_{\bCH^{m',
      \beta}, B(K, 1)} 2^{n(1 - (\alpha + \beta - 1))}.
  \]
  Similarly, we have
  \begin{gather*}
    \| (\omega_{n+1} - \omega_n) \wedge \eta_n \|_{\bCH^{m+m'},K}
    \leq C \|\omega\|_{\bCH^{m,\alpha},B(K,1)} \| \eta
    \|_{\bCH^{m',\beta}, B(K,1)} 2^{n(1 - \alpha - \beta)} \\ \| (\omega_{n+1}
    - \omega_n) \wedge \eta_n \|_{\bCH^{m+m', 1},K} \leq C \| \omega
    \|_{\bCH^{m,\alpha}, B(K,1)} \|\eta\|_{\bCH^{m', \beta}, B(K,1)} 2^{n(1
      - (\alpha + \beta - 1))}
  \end{gather*}
  Thus the series in~\eqref{eq:seriesW*} converges weakly-* and we
  naturally define $\omega \wedge \eta$ to be the weak* limit of
  $(\omega_n \wedge \eta_n)$. By Corollary~\ref{cor:411}, we have
  \begin{equation}
    \label{eq:wedgeCont}
  \| \omega \wedge \eta \|_{\bCH^{m, \alpha + \beta - 1}, K} \leq C \|
  \omega \|_{\bCH^{m,\alpha}, B(K, 1)} \| \eta \|_{\bCH^{m', \beta}, B(K, 1)}.
  \end{equation}
  Thus, $\wedge$ is continuous, as desired.

  Now we need to prove that (A) and (B) hold. (A) is easy, for if
  $\omega$ and $\eta$ are already smooth, then $(\omega_n \wedge
  \eta_n)$ converges locally uniformly (and thus weakly) to the
  pointwise wedge product of $\omega$ and $\eta$. Then the weak*
  and weak limits coincide, so we can conclude that $\omega \wedge
  \eta$ has its natural meaning.

  Finally we prove (B). Let $(\omega^{(p)})$ and $(\eta^{(p)})$ two
  sequences, indexed by $p \geq 0$, that converge weakly-* towards
  $\omega \in \bCH^{m, \alpha}(\R^d)$ and $\eta \in \bCH^{m',
    \beta}(\R^d)$ as $p \to \infty$. As before, we fix a closed ball
  $K$ and we set $\omega_n^{(p)} = \omega^{(p)} * \Phi_{2^{-n}}$ and
  $\eta_n^{(p)} = \eta^{(p)} * \Phi_{2^{-n}}$ for all $n,
  p$. Inequality~\eqref{eq:wedgeCont}, applied to $\omega^{(p)}$ and
  $\eta^{(p)}$, shows that the sequence $(\omega^{(p)} \wedge
  \eta^{(p)})$ is bounded in $\bCH^{m+m', \alpha + \beta - 1}(\R^d)$.

  Proposition~\ref{prop:smooth}, more precisely
  formula~\eqref{eq:formulaConvOmega}, entails that, for all integer
  $n$, the smooth forms $\omega_n^{(p)}$ and $\eta_n^{(p)}$ converge
  locally uniformly to $\omega_n$ and $\eta_n$ as $p \to
  \infty$. Thus, for a fixed normal current $T$ with support in a non
  degenerate closed ball, we have
  \begin{multline*}
  T\left( \omega_{n+1} \wedge (\eta_{n+1} - \eta_n) + (\omega_{n+1} -
  \omega_n) \wedge \eta_n \right) = \\
  \lim_{p \to \infty} T\left(
  \omega_{n+1}^{(p)} \wedge (\eta_{n+1}^{(p)} - \eta_n^{(p)}) +
  (\omega_{n+1}^{(p)} - \omega_n^{(p)}) \wedge \eta_n^{(p)} \right).
  \end{multline*}
  Likewise,
  \[
  T(\omega_0 \wedge \eta_0) = \lim_{p \to \infty} T(\omega_0^{(p)}
  \wedge \eta_0^{(p)}).
  \]
  Arguing as before, one has
  \begin{multline*}
  \left| T \left( \omega_{n+1}^{(p)} \wedge (\eta_{n+1}^{(p)} -
  \eta_n^{(p)}) + (\omega_{n+1}^{(p)} - \omega_n^{(p)}) \wedge
  \eta_n^{(p)} \right) \right| \\ \leq C \|\omega^{(p)}\|_{\bCH^{m,
      \alpha}, B(K, 1)} \| \eta^{(p)} \|_{\bCH^{m', \beta}, B(K, 1)} 2^{n(1 -
    \alpha - \beta)} \bN(T).
  \end{multline*}
  As the sequences $(\omega^{(p)})$ and $(\eta^{(p)})$ are bounded,
  the previous bound can be made uniform in $p$, allowing us to apply
  Lebesgue's dominated convergence theorem in
  \begin{multline*}
    \lim_{p \to \infty} \omega^{(p)} \wedge \eta^{(p)}(T) \\
    \begin{aligned}
      & = \lim_{p \to \infty} \left(T(\omega_0^{(p)} \wedge
      \eta_0^{(p)}) + \sum_{n=0}^\infty T\left( \omega_{n+1}^{(p)}
      \wedge (\eta_{n+1}^{(p)} - \eta_n^{(p)}) + (\omega_{n+1}^{(p)} -
      \omega_n^{(p)}) \wedge \eta_n^{(p)} \right) \right)\\ & = \omega
      \wedge \eta(T)
    \end{aligned}
  \end{multline*}

  The case $\alpha = \beta = 1$, though simpler, requires special
  attention. The uniqueness of $\wedge$ is already established. As
  previously, we define $\omega \wedge \eta$ to be the weak* limit
  $\lim_n \omega_n \wedge \eta_n$. It exists because of the
  compactness theorem. Indeed, for a suitable $K$, we have, by
  Corollary~\ref{cor:411} and
  Proposition~\ref{prop:estimatesSmoothing}(A)
  (or~\ref{prop:estimatesSmoothing}(B))
  \begin{align*}
    \| \omega_n \wedge \eta_n \|_{\bCH^{m+m', 1}, K} & \leq C
    \|\omega_n\|_{\bCH^{m,1}, K} \| \eta_n \|_{\bCH^{m',1}, K} \\ &
    \leq C \| \omega \|_{\bCH^{m,1}, B(K, 1)} \| \eta \|_{\bCH^{m',
        1}, B(K, 1)}
  \end{align*}
  It remains to show that $(\omega_n \wedge \eta_n)$ converges
  weakly. This is the case because, for $T \in \bN_m(K)$,
  \begin{equation}
    \label{eq:cas11}
  T(\omega_N \wedge \eta_N) = T(\omega_0 \wedge \eta_0) +
  \sum_{n=0}^{N-1} T\left( \omega_{n+1} \wedge (\eta_{n+1} - \eta_n) +
  (\omega_{n+1} - \omega_n) \wedge \eta_n \right)
  \end{equation}
  As before, we establish
  \[
  \left| T\left( \omega_{n+1} \wedge (\eta_{n+1} - \eta_n) +
  (\omega_{n+1} - \omega_n) \wedge \eta_n \right) \right| \leq \frac{C
    \| \omega \|_{\bCH^{m, 1}, B(K, 1) } \| \eta \|_{\bCH^{m', 1},
      B(K, 1)} }{2^n}
  \]
  which ensures the absolute convergence of the series in~\eqref{eq:cas11}.

  Now that the wedge product is well-defined as a map
  $\bCH^{m,1}(\R^d) \times \bCH^{m', 1}(\R^d) \to \bCH^{m+m',
    1}(\R^d)$, properties (A) and (B) are shown as before.
\end{proof}

We remark that the Young-type integral defined in~\cite{Boua}
corresponds to a special case of this construction, namely the product
of a Hölder function (seen as a fractional $0$-charge) with a
$d$-charge.

Using the weak* density of smooth forms, it is easy to extend the
well-known formulae of exterior calculus to fractional charges. Among
them, we have, for $\omega \in \bCH^{m, \alpha}(\R^d)$, $\eta \in
\bCH^{m', \beta}(\R^d)$ and $\alpha + \beta > 1$,
\begin{itemize}
\item[(A)] $\wedge$ is bilinear;
\item[(B)] $\omega \wedge \eta = (-1)^{mm'} \eta \wedge \omega$;
\item[(C)] $\diff (\omega \wedge \eta) = \diff \omega \wedge \eta +
  (-1)^m \omega \wedge \diff \eta$.
\end{itemize}
The last item uses the weak*-to-weak* continuity of the exterior
derivative.

Regarding the product of many fractional charges, we notice that
$\omega_1 \wedge \cdots \wedge \omega_k$ makes sense (and the product
is associative) as long as $\omega_i$ is $\alpha_i$-fractional and
$\alpha_1 + \cdots + \alpha_k > k-1$. In this case, the result is an
$(\alpha_1 + \cdots + \alpha_k - (k-1))$-fractional charge.

Pullbacks of fractional charges by locally Hölder maps can also be
defined via duality with the pushforward operator on fractional
currents. However, this operation and its compatibility with the wedge
product is not studied in this work.

\section{Fractional currents as metric currents}
\label{sec:metricCurrents}

We conclude this paper with a section that discusses some questions
related to the characterization of metric currents of snowflaked
Euclidean spaces. This problematic was introduced in~\cite{Zust2},
upon which we further build.

Ambrosio and Kirchheim introduced metric
currents in~\cite{AmbrKirc}, while Lang provided in~\cite{Lang} a
definition of metric currents in locally compact metric spaces without
relying on a finite mass assumption.  Here, we reproduce the
definition of Hölder currents, a variant of Lang's definition, that
was introduced by Züst in~\cite[Definition~2.2]{Zust}. For simplicity,
we consider currents in \(\R^d\) with compact support.

\begin{defi}
Let $(\alpha_0, \dots, \alpha_m) \in \mathopen{(} 0, 1
\mathclose{]}^{m+1}$. An $m$-dimensional $(\alpha_0, \dots,
  \alpha_m)$-Hölder current is a multilinear map
\[
T \colon \prod_{k=0}^{m} \rmLip^{\alpha_k}(\R^d) \to \R
\]
that satisfies:
\begin{enumerate}
\item[(A)] Locality: $T(g^0, g^1, \dots, g^m) = 0$ whenever some there
  is $1 \leq k \leq m$ such that $g^k$ is constant on an open
  neighborhood of $\spt g^0$.
\item[(B)] Compact support: there is a compact set $K \subset \R^d$
  such that $T(g^0, g^1, \dots, g^m) = 0$ whenever $\spt g^0 \cap K =
  \emptyset$.
\item[(C)] Continuity: for any sequences $(g^k_n)_n$ ($0 \leq k \leq
  m$) that converge uniformly to $g^k$ with equi-bounded
  $\alpha_k$-Hölder constants, one has
  \[
  \lim_{n \to \infty} T(g^0_n, \dots, g^m_n) = T(g^0, \dots, g^m).
  \]
\end{enumerate}
The space of such currents is denoted $\calD_m^{(\alpha_0, \dots,
  \alpha_m)}(\R^d)$.
\end{defi}

The particular case $\alpha_0 = \cdots = \alpha_m = 1$ corresponds to
the compactly supported currents in the sense of Lang. More generally,
for any $\theta \in \mathopen{(} 0, 1 \mathclose{]}$, the
  $m$-dimensional $(\theta, \dots, \theta)$-Hölder currents are the
  compactly supported currents of the snowflaked Euclidean space
  $(\R^d, d_{\mathrm{Eucl}}^\theta)$, in the sense of Lang.

We expect that the class of $(\alpha_0, \dots, \alpha_m)$-Hölder
currents possesses a geometric description that depends solely on the
parameter
\[
\alpha = \alpha_0 + \cdots + \alpha_m - m.
\]
If $\alpha \leq 0$, then it is known (see \cite[Theorem~4.5]{Zust2})
that $\calD_m^{(\alpha_0, \dots, \alpha_m)}(\R^d) = 0$. If $\alpha =
1$, \textit{i.e.} $\alpha_0 = \cdots = \alpha_m = 1$, then it was
conjectured in~\cite{Lang} that the corresponding currents are in
correspondence with Federer-Fleming's flat currents. However, this has
been recently disproved by J. Tak\'a\v{c} in \cite{Takac}.

Next we treat the case $0 < \alpha \leq 1$. If $g^k \in
\rmLip^{\alpha_k}(\R^d)$ for all $0 \leq k \leq m$, then the functions
$g_k$ can be seen as $\alpha_k$-fractional $0$-charges over
$\R^d$. Therefore, the product of charges
\[
g^0 \wedge \diff g^1 \wedge \cdots \wedge \diff g^m
\]
makes sense, and is, according to Section~\ref{sec:ext}, an
$\alpha$-fractional charge. Therefore, to each $S \in
\bF^\alpha_m(\R^d)$, we may associate the multilinear map
\[
\Upsilon(S) \colon \prod_{k=0}^m \rmLip^{\alpha_k}(\R^d) \to \R \colon
(g^0, \dots, g^m) \mapsto \langle g^0 \wedge \diff g^1 \wedge \cdots
\wedge \diff g^m, S \rangle.
\]
It is easily proven that $\Upsilon$ is injective, and that, using the
continuity properties of $\wedge$, $\Upsilon(S)$ is an $m$-dimensional
$(\alpha_0, \dots, \alpha_m)$-Hölder current. We believe that
$\Upsilon$ is in fact a bijection $\calD_m^{(\alpha_0, \dots,
  \alpha_m)}(\R^d) \to \bF^\alpha_m(\R^d)$ if $\alpha \neq 1$. This
assertion may be regarded as a fractional variant of the flat chain
conjecture.

In particular, taking all $\alpha_k$ to be equal, we may regard
$\alpha$-fractional currents as metric currents in the sense of Lang
of the snowflaked space $(\R^d,
d_{\mathrm{Eucl}}^{\frac{m+\alpha}{m+1}})$.

\bibliographystyle{amsplain}
\bibliography{phil.bib}
%\printindex

\end{document}